\documentclass[a4paper,reqno,10pt,oneside]{elsarticle}
\usepackage{amsmath,geometry}
\usepackage[english]{babel}
\usepackage{amsmath}
\usepackage[utf8]{inputenc}
\usepackage{amssymb}
\usepackage{amsthm}
\usepackage{newlfont}
\usepackage{pifont}
\usepackage{geometry}
\usepackage{bbm}

\usepackage{etoolbox}
\makeatletter
\def\@seccntformat#1{\@ifundefined{#1@cntformat}%
   {\csname the#1\endcsname.\hskip0.5em}    
   {\csname #1@cntformat\endcsname}
}
\patchcmd{\appendix}{\appendixname}{}{}{}
\appto{\appendix}{%
    \newcommand{\section@cntformat}{\appendixname\ \thesection.\hskip0.5em}}
\makeatother

\newtheorem{teo}{Theorem}[section]

\newtheorem{lemma}[teo]{Lemma}
\newtheorem{prop}[teo]{Proposition}
\newtheorem{cor}[teo]{Corollary}
\theoremstyle{remark}
\newtheorem{rem}[teo]{Remark}

\newtheorem{cla}{Step}

\newcommand{\R}{\mathbb{R}}

\newcommand{\N}{\mathbb{N}}
\newcommand{\Sf}{\mathbb{S}}

\newcommand{\de}{\,\mathrm{d}}

\numberwithin{equation}{section}

\begin{document}
\begin{frontmatter}

\title{On the structure of the nodal set and asymptotics of least energy sign-changing radial solutions of the fractional Brezis-Nirenberg problem}

\author[gc]{G.\ Cora\fnref{ack}}
\ead{gabriele.cora@unito.it}
\address[gc]{Dipartimento di Matematica ``G. Peano'', Università degli Studi di Torino, via Carlo Alberto, 10 - 10123 Torino, Italy}

\author[ai]{A.\ Iacopetti\fnref{ack}}
\ead{alessandro.iacopetti@ulb.ac.be}
\address[ai]{Département de Mathématique, Université Libre de Bruxelles, Campus de la Plaine - CP213 Boulevard du Triomphe, 1050 Bruxelles, Belgium}

\fntext[ack]{\textit{Funding}: Research partially supported by the project ERC Advanced Grant 2013 n.~339958 Complex Patterns for Strongly Interacting Dynamical Systems COMPAT, and by Gruppo Nazionale per l'Analisi Matematica, la Pro\-ba\-bi\-li\-t\`a e le loro Applicazioni (GNAMPA) of the Istituto Nazionale di Alta Matematica (INdAM)}


\begin{abstract}
In this paper we study the asymptotic and qualitative properties of least energy radial sign-changing solutions of the fractional Brezis--Nirenberg problem ruled by the s-Laplacian, in a ball of $\R^n$, when $s \in (0,1)$ and $n > 6s$. As usual, $\lambda$ is the (positive) parameter in the linear part in $u$.
We prove that for $\lambda$ sufficiently small such solutions cannot vanish at the origin, we show that they change sign at most twice and their zeros coincide with the sign-changes. Moreover, when $s$ is close to $1$, such solutions change sign exactly once. Finally we prove that least energy nodal solutions which change sign exactly once have the limit profile of a ``tower of bubbles'', as $\lambda \to 0^+$, i.e. the positive and negative parts concentrate at the same point with different concentration speeds.
\end{abstract}
\end{frontmatter}
\begin{section}{Introduction}

Let $s \in (0,1)$, $\lambda>0$, let $n \in \mathbb{N}$ be such that $n>2s$ and let $\Omega \subset \R^n$ be a bounded domain with smooth boundary. Consider the the following non local semilinear elliptic problem:
\begin{equation}
\label{fracBrezis}
\begin{cases}
(-\Delta)^s u = \lambda u + |u|^{2_s^*-2}u &\hbox{in}\ \Omega, \\
u = 0 &\hbox{in}\ \R^n \setminus \Omega,
\end{cases}
\end{equation}
where $2_s^* := \frac{2n}{n-2s}$ is the critical fractional Sobolev exponent for the embedding of $\mathcal{D}^s(\R^n)$ into $L^{2_s^*}(\R^n)$, and $(-\Delta)^s$ is the s-Laplacian operator, which is defined as
\[
(-\Delta)^s u(x) := C_{n,s} P.V. \int_{\R^n}\frac{u(x)-u(y)}{|x-y|^{n+2s}}\de y= C_{n,s} \lim_{\varepsilon \to 0^+} \int_{\R^n \setminus B_\varepsilon(x)}\frac{u(x)-u(y)}{|x-y|^{n+2s}}\de y,
\]
where the constant $C_{n,s}$ is given by

\[
C_{n,s} := \frac{2^{2s} \Gamma\left(\frac{n}{2}+s\right)}{\pi^{\frac{n}{2}}|\Gamma(-s)|}.
\]
We denote by $X^s_0(\Omega)$ the Sobolev space of the functions $u \in H^s(\R^n)$ such that $u = 0 $ in $\R^n\setminus \Omega$, endowed with the norm
\[
\|u\|^2_s := {\frac{C_{n,s}}{2}\int_{\R^{2n}} \frac{|u(x)-u(y)|^2}{|x-y|^{n+2s}}\de x \de y},
\] 
whose associated scalar product is
\[
(u, v)_s := \frac{C_{n,s}}{2}\int_{\R^{2n}} \frac{(u(x)-u(y))(v(x)-v(y))}{|x-y|^{n+2s}}\de x \de y.
\]

Weak solutions of Problem \eqref{fracBrezis} correspond to critical points of the energy functional 
\[
I(u) := \frac{1}{2}(\|u\|_s^2-\lambda|u|^2_2)- \frac{1}{2_s^*}|u|_{2_s^*}^{2_s^*},\ u \in X_0^s(\Omega),
\]
 where $|\cdot|_p$ is the standard $L^p$-norm for $p\geq 1$. We point out that, in view of known regularity results for the fractional Laplacian, weak solutions $u \in X_0^s(\Omega)$ of \eqref{fracBrezis} turn out to be of class $C^{0,s}(\R^n)$ (as it follows by combining \cite[Theorem 1.1]{HolderReg} and \cite[Theorem 3.2]{IanMosSqua}), and this regularity is optimal. For the interior regularity in $\Omega$ we have better results (see \cite{HolderReg}).

Problem \eqref{fracBrezis} is known as the fractional Brezis--Nirenberg problem since in the local case the first existence  result for positive solutions was given in the celebrated paper \cite{BN}. In \cite{BN}, Brezis and Nirenberg overcame the difficulties due to the lack of compactness of the embedding $H_0^1 \hookrightarrow L^{2^*}$, and showed that the dimension plays a crucial role in the problem. In fact, they proved that when $n\geq4$ there exist positive solutions for every $\lambda \in (0,\lambda_1(\Omega))$, where $\lambda_1(\Omega)$ denotes the first eigenvalue of the classical Dirichlet-Laplacian on $\Omega$. The case $n=3$ is more delicate. Brezis and Nirenberg proved that there exists $\lambda^*(\Omega)>0$ such that  positive solutions exist for every $\lambda \in (\lambda^*(\Omega),\lambda_1(\Omega))$.  When $\Omega=B_R$ is a ball, they also proved that $\lambda^*(B_R)=\frac{\lambda_1(B_R)}{4}$ and 
a positive solution exists if and only if $\lambda \in \left(\frac{\lambda_1(B_R)}{4}, \lambda_1(B_R)\right)$.\\ 

After the pioneering paper \cite{BN}, many results have been obtained concerning the asymptotic analysis of positive solutions, multiplicity, existence and nonexistence of sign-changing solutions (see \cite{HAN, Rey, CFP, CSS, ABP, CW, Pacella, SZ, IacVair, IacVair2, IacPac, Demmak}). 
We point out that in the sign-changing case the dimension $n=3$ exhibits additional difficulties: it is not yet known if there exist non radial sign-changing solutions for $\lambda \in (0,\frac{\lambda_1(B_R)}{4})$. A partial answer to this question was given by Ben Ayed, El Mehdi and Pacella in \cite{Pacella2}. Nevertheless, even in the other dimensions  several interesting phenomena are observed. In fact, Atkinson, Brezis and Peletier in \cite{ABP}, Adimurthi and Yadava in \cite{AY2}  showed, with different proofs, that for $n=4,5,6$ there exists $\lambda^{**}=\lambda^{**}(n)>0$ such that there is no radial sign-changing solution of the Brezis--Nirenberg problem in the ball for $\lambda \in (0,\lambda^{**})$. Instead, they do exist for any $\lambda \in (0,\lambda_1(B_R))$, if $n\geq 7$, as proved by Cerami, Solimini and Struwe in \cite{CSS}.\\

In recent years, a great attention has been devoted to studying non local equations and a natural question is if it is possible to extend the known results about semilinear elliptic problems in the fractional framework. In the case of positive solutions of the fractional Brezis--Nirenberg problem, the picture is quite clear. Servadei and Valdinoci in \cite{ValSerLD}, \cite{ValSer}, proved existence of positive solutions for Problem \eqref{fracBrezis} and their results perfectly agree with the classical ones: if $\lambda_{1,s} = \lambda_{1,s}(\Omega)$ is the first eigenvalue of the fractional Laplacian with homogeneous Dirichelet boundary condition, then Problem \eqref{fracBrezis} admits a nontrivial solution whenever $n \geq 4s$ and $\lambda \in (0, \lambda_{1,s})$. When $2s<n<4s$ there exists $\lambda^*_s=\lambda^*_s(\Omega)$ such that a solution of Problem \eqref{fracBrezis} exists for $\lambda>\lambda^*_s$, and $\lambda$ different from the eigenvalues of the fractional Laplacian. Other interesting results have been obtained by Musina and Nazarov in \cite{Musina2}, for the fractional Dirichlet-Laplace operator $(-\Delta)^m$, $0<m<\frac{n}{2}$.

The asymptotic behavior of least energy positive solutions of Problem \eqref{fracBrezis} (in the case of the spectral fractional Laplacian), as $\lambda \to 0^+$, has been studied by Choi, Kim and Lee in \cite{CHoiKimLee}. Even in this case the results perfectly fit with the classical ones of Han and Rey (see \cite{HAN}, \cite{Rey}).

On the contrary, there is not much literature for sign-changing solutions (see \cite{BFS}) and very few is known about their qualitative properties. In fact, even in the radial case, due to the non local interactions, there are serious difficulties when trying to determine the number of sign changes of the solutions, and this number does not correspond, in general, to the number of connected components of the complement of the nodal set (minus one). Moreover, since we deal with sign-changing solutions, no information is available about their monotonicity via the fractional moving plane method (see \cite{chen}). In addition, as pointed out in the seminal paper of  Frank, Lenzmann and Silvestre \cite{FrLe2}, we cannot apply standard ODE techniques for the fractional Laplacian and this technical gap causes serious troubles.\\   

In this work we face with the following problems.\\ 

\textbf{Problem a):} Let $B_R \subset \R^n$ be the ball of radius $R$ centered at the origin. Consider the following simple property:
\[
(\mathcal{P})\ \ \ \ \ \ \ \ \ \hbox{if}\ u \ \hbox{is a radial solution of Problem \eqref{fracBrezis} in $B_R$ and}\ u(0)=0\ \hbox{then}\ u\equiv0.\\[6pt]
\]
It is well known that in the local case $(\mathcal{P})$ holds, but in the fractional framework it is basically unknown when dealing with nodal solutions. The only result in this direction is due to Frank, Lenzmann and Silvestre, who, in \cite{FrLe2}, by using a monotonicity formula argument, showed that $(\mathcal{P})$ holds for radial solutions vanishing at infinity of fractional linear equations of the kind $(-\Delta)^s u + Vu=0$ in $\R^n$, where $V=V(r)$ is radial and non-decreasing, $r=|x|$. 

Unfortunately, in the case of bounded domains this argument does not work properly. In fact, let $u$ be a radial solution of Problem \eqref{fracBrezis} in $B_R$ and let $W:\overline{\R^{n+1}_+} \to \R$ be the extension of $u$ to the upper half space $\R^{n+1}_+ = \R^n \times \R_+$ (see Sect. 2.4 for the definition). The function $W$ is also known as the Caffarelli-Silvestre extension in view of their celebrated paper \cite{CS}. Recalling that $W=W(x,y)$ is cylindrically symmetric with respect to $x\in \R^n$, let us formally write the expression
\[
H(r)= d_s \int_0^{+\infty} \frac{t^a}{2}[W_r^2(r,t)-W^2_y(r,t)] \ dt - \frac{\lambda}{2}u^2(r)-\frac{1}{2_s^*}|u(r)|^{2_s^*},\ r \geq 0,
\]
where $d_s = \frac{1}{2^{{1-2s}}}\frac{\Gamma\left(s\right)}{\Gamma\left(1-s\right)}$. Then, when  trying to repeat the proof of the monotonicity formula, as in the remarkable paper of Cabr\'e and Sire (see \cite[Lemma 5.4]{Yannick}), we cannot deduce that $H$ is decreasing for all $r>0$ because $-d_s\lim_{y\to 0^+}\ y^{1-2s}W_y(r,y)=\lambda u + |u|^{2_s^*-s}u$ just on $(0,R)$.

 Now, since $W$ is cylindrically symmetric we have $W_r(0,y)\equiv 0$ for any $y>0$, and assuming that $u(0)=0$ we deduce that $H(0)\leq 0$. But, even if $H$ is decreasing in $(0,R)$, we have no information on the value $H(R)=d_s \int_0^{+\infty} \frac{t^a}{2}[W_r^2(R,t)-W^2_y(R,t)] \ dt$, while, in \cite{FrLe2}, by proving that $H$ is decreasing in $(0,+\infty)$, and since $\lim_{r \to +\infty} H(r)=0$, $H(0)\leq 0$, they deduce that $H\equiv 0$ and $u\equiv 0$.

We stress that even other approaches fail in the nodal case. For example, if we try to apply the strong maximum principle, as in the version stated by Cabr\'e and Sire in \cite[Remark 4.2]{Yannick}, assuming that $u\geq 0=u(0)$ in a neighborhood of the origin we must find a small positive $\epsilon>0$ such that the extension $W$ is non negative in $\Gamma^+_\epsilon=  \{ (x,y) \in \overline{\R^{n+1}_+} \ | \ y\geq0, \sqrt{|x|^2+y^2} = \epsilon\}$. Unfortunately, if $u$ changes sign, then also $W$ changes sign (see Sect. 5) and it can happen that for any small $\epsilon>0$ the set $\Gamma^+_\epsilon$ intersects $\{W<0\}$, and thus we cannot exclude that $u(0)=0$. This is not surprising because, due to the nonlocal interaction terms, we have that $u^+$, $u^-$ are not weak super, sub solutions of Problem \eqref{fracBrezis} in $\{u>0\}$, $\{u<0\}$, respectively.

Also with the recent version of the fractional strong maximum principle stated by Musina and Nazarov in \cite[Corollary 4.2]{Musina}, considering any subdomain of $B_R \cap \{u\geq0\}$, we deduce only that $u> \inf_{\R^n} u$. Clearly if $u$ is a nodal solution of Problem \eqref{fracBrezis}, again, we cannot exclude that $u(0)=0$ and $u\not\equiv 0$.\\

\textbf{Problem b):} Determine the number of connected components of the complement of the nodal set and the number of sign changes of least energy nodal solutions of \eqref{fracBrezis}, when $\lambda$ is close to zero.\\

We say that $u_\lambda$ is a least energy sign-changing solution of \eqref{fracBrezis} if $I(u_\lambda)= \inf_{\mathcal{M}} I$, where $\mathcal{M}$ is the nodal Nehari set, i.e.  
\[
 \mathcal{M} := \{u \in X_0^s(\Omega) \ |\ u^\pm \not \equiv 0, I'(u)[u^\pm] = 0 \}.
\]

In view of the previous discussion we remark again that the number of connected components of 
\[
\{ u_\lambda \neq 0\}
\]
does not correspond, in general, to the number of sign changes (plus one). 
Despite that, even assuming that these numbers coincide, in view of the non local interactions between the nodal components, it is not possible, via standard energy arguments, to determine them. In fact, let $u_\lambda$ be a least energy solution of Problem \eqref{fracBrezis} and let $K_{i,\lambda}$ be a connected component of $\{u_\lambda \neq 0\}$. Setting  
\[
u_{i,\lambda}:=u_{\lambda} \ \mathbbm{1}_{K_{i,\lambda}},
\]
where $\mathbbm{1}_{K_{i,\lambda}}$ is the characteristic function of ${K_{i,\lambda}}$, then, from $I^\prime (u_\lambda) [u_{i,\lambda}] =0$ we have
\[
\|u_{i,\lambda}\|^2_s + (u_{i,\lambda},u_\lambda-u_{i,\lambda})_s = \lambda |u_{i,\lambda}|^2_{2,K_{i,\lambda}} +  |u_{i,\lambda}|^{2^{*}_s}_{{2^{*}_s},K_{i,\lambda}},
\]
and by a simple computation we see that

\begin{equation}\label{interactionterm}
(u_{i,\lambda},u_\lambda-u_{i,\lambda})_s = -C_{n,s} \int_{\mathbb{R}^{2n}} \frac{u_{i,\lambda}(x) (u_{\lambda}(y)-u_{i,\lambda}(y))}{|x-y|^{n+2s}} \ dx dy.
\end{equation}

Even if it is not difficult to show that $\|u_\lambda\|_s^2 \to 2 S_s^{n/2s}$, as $\lambda \to 0^+$, where $S_s$ is the best fractional Sobolev constant (see \eqref{nonlocSob}), we do not have any information about the limit value of \eqref{interactionterm} nor on its sign. In particular, the presence of this interaction term between $u_{i,\lambda}$ and the other nodal components does not allow us to replicate the proof of Ben Ayed, El Mehdi and Pacella (see \cite[Proof of Theorem 1.1]{Pacella}). In fact, in the local case, by using Poincar\'e and Sobolev inequalities, one can deduce that
\[
\int_{K_{i,\lambda}} |\nabla u_{i,\lambda}|^2 \geq \left(1+o(1)\right)  S_1^{n/2},
\]
and being $\|u_\lambda\|_1^2 \to 2 S_1^{n/2}$ it follows that $u_\lambda$ cannot have more than two nodal components.\\

\textbf{Problem c):} Determine the asymptotic profile of least energy nodal solutions of \eqref{fracBrezis} as $\lambda \to 0^+$.\\

The aim of this paper is to contribute to Problem a), Problem b) and Problem c) in the case of least energy nodal radial solutions of the fractional Brezis--Nirenberg problem in the ball. We remark that existence of radial sign-changing solutions in the ball for Problem \eqref{fracBrezis} is granted for any $s \in (0,1)$, $n>6s$, $\lambda \in (0,\lambda_{1,s}(B_R))$. The proof is essentially the same of Cerami, Solimini and Struwe, \cite{CSS}, with slight changes. For the sake of completeness we give the proof in Section 3.

Our results are the following:

\begin{teo}\label{mainteoproba}
Let $n>6s$, $s \in (\frac{1}{2},1)$ and let $R>0$. There exists $\bar\lambda>0$ such that for any $\lambda \in (0,\bar \lambda)$, any least energy sign-changing radial solution $u_{\lambda}$ to \eqref{fracBrezis} in $B_R$ does not vanish at zero.
\end{teo}

\begin{teo}\label{mainteocc}
Let $n > 6s$, $s \in \left(0, 1\right)$ and let $R>0$. There exists $\hat \lambda_s>0$ such that, for any $\lambda \in (0, \hat \lambda_s)$, any least energy sign-changing radial solution $u_{s,\lambda}$ to \eqref{fracBrezis} in $B_R$ changes sign at most twice. Moreover, the zeros of $u_{s,\lambda}= u_{s,\lambda}(r)$ in $(0,R)$ coincide with its nodes, i.e. with the sign-changes of $u_{s,\lambda}$. More precisely, one and only one of the following hold:
\begin{enumerate}[(a)]
\item if $u_{s, \lambda}$ changes sign twice then it vanishes in $[0,R)$ only at the nodes,
\item if $u_{s, \lambda}$ changes sign once then then it vanishes in $(0,R)$ only at the node and it can vanish also at the origin. 
\end{enumerate}

\end{teo}

\begin{teo}\label{mainteo}
Let $n\geq 7$ and let $R>0$. There exist $\tilde\lambda>0$ such that for any $\lambda \in (0,\tilde\lambda)$ there exists $\bar s \in (0,1)$ such that for any $s \in (\bar s, 1)$, any least energy sign-changing radial solution $u_{\lambda}$ to \eqref{fracBrezis} in $B_R$ changes sign exactly once. 
\end{teo}

\begin{teo}\label{mainteo2}
Let $n>6s$, $s \in (\frac{1}{2},1)$ and let $R>0$. Let $(u_\lambda)$ be a family of least energy sign-changing radial solutions to \eqref{fracBrezis} in $B_R$, such that $u_\lambda(r)$ changes sign exactly once in $(0,R)$ for all sufficiently small $\lambda>0$. Assume, without loss of generality, that $u_\lambda(0)\geq0$ in a neighborhood of the origin, and set $M_{\lambda,\pm}:=\|u_\lambda^\pm\|_\infty$. Then:
\begin{itemize}
\item[i)] $M_{\lambda,\pm}\to + \infty $ as $\lambda \to 0^+$,
\item[ii)] denoting by $r_\lambda \in (0,R)$ the node of $u_\lambda$ and by $s_\lambda \in (r_\lambda, R)$ any point where $u_\lambda=u_\lambda(r)$ achieves $-M_{\lambda,-}$ we have $r_\lambda, s_\lambda \to 0$ as $\lambda \to 0^+$,
\item[iii)] $\frac{M_{\lambda,+}}{M_{\lambda,-}} \to + \infty $ as $\lambda \to 0^+$,
\item[iv)] setting $\beta:= \frac{2}{n-2s}$, then the rescaled function
\[
\tilde u^+_\lambda(x) := \frac{1}{M_{\lambda,+}}u^+_\lambda\left( \frac{x}{M_{\lambda,+}^\beta}\right), \quad x \in \R^n,
\]
converges in $C^{0,\alpha}_{loc}(\R^n)$, as $\lambda \to 0^+$, for some $\alpha=\alpha(s) \in (0,1)$, to the fractional standard bubble $U_s$ in $\R^n$ centered at $0$ and such that $U_s(0)=1$.
\end{itemize}
\end{teo}

Theorem \ref{mainteoproba} is a consequence of a more general result, which ensures that $u_\lambda(0)$ is bounded away from zero, by a constant which is uniform with respect to $\lambda$. The idea is to argue by contradiction and to construct a family of rescaled functions $\tilde u_\lambda$ such that $\tilde u_\lambda(0) \to 0$ as $\lambda \to 0^+$. By a standard argument $\tilde u_\lambda$ converges, in compact subsets of $\R^n$, to a solution $\tilde u$ of the fractional critical problem  $(-\Delta)^s U=|U|^{2_s^*-2}U$ in $\R^n$. Then, by energy considerations and the fractional strong maximum principle, we deduce that $\tilde u$ has to be positive in $\R^n$, contradicting that $\tilde u(0)=0$.\\ 

 The proofs of Theorem \ref{mainteocc}, Theorem \ref{mainteo} rely on the combination of several tools. The first step is to prove that the number of the nodal components of the extension $W$ is two. This is done by arguing as in the papers \cite{FrLe1}, \cite{FrLe2}, and then, exploiting the radiality of the solutions, we prove that our solutions change sign at most twice. In view of this information Theorem \ref{mainteocc} follows from a topological argument based on the Jordan's curve theorem, the fractional strong maximum principle and on a nice result of Fall and Felli (see \cite[Theorem 1.4]{fellifall}) which ensures that our solutions cannot vanish in a set of positive measure.
 
 For Theorem \ref{mainteo}, the fundamental step is to argue by contradiction and to prove that if two nodes exist for $s$ close to $1$ then they persist for the limit profile. This is done by performing an asymptotic analysis of the nodes of the solutions when $s \to 1^-$, fine energy estimates, a quite complex technical result (see the Appendix, Theorem \ref{subsol}) and the strong maximum principle for the standard Laplacian. At the end, it is not difficult to prove that the limit function is a nodal solution of the classical Brezis--Nirenberg problem, and it is of least energy, and thus we get a contradiction since such solutions change sign exactly once.

We point out that the restriction to $n\geq 7$ is essential for the result because existence of sign-changing radial solutions in the ball for the classical Brezis--Nirenberg problem, when $\lambda$ is close to $0$, holds only for $n\geq 7$ (see \cite{ABP}, \cite{AY2}, \cite{CSS}).\\ 

The proof of Theorem \ref{mainteo2} is based on the analysis of rescaled functions. We observe that statement iii) strongly relies on the fact that $u_\lambda$ possesses exactly one node. In fact, assuming that $u_\lambda$ has at least two nodes, then, we still have that $u_\lambda^+$ and $u_\lambda^-$ carry the same energy as $\lambda \to 0^+$. In particular $\{u_\lambda>0\}$ has at least two components and spreading the energy between these components does not allow us to establish the leading term between $M_{\lambda,+}$ and $M_{\lambda,-}$.


We point out that no information about the limit profile of suitable rescalings of $u_\lambda^-$ is provided. The reason is that, differently from the results of \cite{Iacopetti}, $u_\lambda^-$ is not a solution of Problem \eqref{fracBrezis} in $\{u_\lambda <0\}$, and we cannot apply ODE techniques. Finally, the restriction $s>\frac{1}{2}$ is technical because we make use intensively of the fractional Strauss inequality, as in the version stated in \cite[Proposition 1]{choozawa}, and it is known that such inequality fails for the values $0<s \leq \frac{1}{2}$ (see \cite[Remark 2, Remark 4]{choozawa}).\\

In a separate paper, we aim to extend the results in the whole interval $s \in (0,1)$ by a continuation argument. We also remark that our proofs work, with slightly adjustments, for fractional semilinear problems with subcritical nonlinearities.\\

The outline of the paper is the following: in Sect. 2 we fix the notation and we recall some known results, in Sect. 3 we prove the existence of radial solutions of Problem \eqref{fracBrezis} in the ball. In Sect. 4 we prove some preliminary results about the asymptotic analysis of the energy as $\lambda\to0^+$, and in Sect. 5 we study the nodal set of the extension. In Sect. 6 we provide uniform bounds, with respect to the parameter $s$, for the $L^\infty$-norm and the energy of the solutions. Finally in Sect. 7, 8, 9, 10 we prove, respectively, Theorem \ref{mainteoproba}, Theorem \ref{mainteocc}, Theorem \ref{mainteo}, and Theorem \ref{mainteo2}. At the end in the Appendix we prove some technical results and Theorem \ref{subsol}.
\end{section}

\begin{section}{Notation and preliminary results}\label{SectionIntro}
In this section we fix the notation and we recall some known results which will be used in the present paper. 

\begin{subsection}{Functional framework}

We denote by $\omega_n$ the $n$-dimensional measure of the unit sphere $\Sf^n$ and by $B_R(x_0) \subset \R^n$ the ball centered at $x_0 \in \R^n$ with radius $R>0$. If $x_0 = 0$ we simply write $B_R$. 

Let $\Omega$ be a domain in $\R^n$, we denote by $|\cdot|_p$ the usual $L^p(\Omega)$ norm, for $p \in [1, \infty]$.
Moreover, for $k \in \N$, $\alpha \in (0,1)$ we set
\[
\begin{aligned}
|D^k u|_{\infty; \Omega} := \sup_{\underset{|\gamma| = k}{\gamma \in \N^n}}\sup_{x \in \Omega} |D^\gamma u(x)|, \quad 
[u]_{k,\alpha; \Omega} := \sup_{\underset{|\gamma| = k}{\gamma \in \N^n}}\sup_{x,y \in \Omega} \frac{|D^\gamma u(x) - D^\gamma u (y)|}{|x-y|^\alpha},
\end{aligned}
\]
so that
\[
\|u\|_{k, \alpha; \Omega} := \sum_{j=0}^k |D^k u|_{\infty; \Omega} + [u]_{k,\alpha; \Omega}
\]
is the standard norm in $C^{k,\alpha}(\overline{\Omega})$. If $\Omega = \R^n$ we omit the subscript in the above norms.\\

Let $\Omega$ be a smooth bounded domain. In Sect. 1 we have introduced the Sobolev spaces $X^s_0(\Omega)$, for $s \in (0,1)$. For further properties of such space we refer to \cite{Ser, ValSerLD, ValSer} and the references therein. A weak solution for \eqref{fracBrezis} is defined as a function $u \in X^s_0(\Omega)$ such that
\[
(u,\varphi)_s = \lambda \int_\Omega u \varphi \de x + \int_\Omega |u|^{2^*_s-2}u\varphi \de x,
\]
for every $\varphi \in X^s_0(\Omega)$.\\[6pt] 
It is well known (see e.g. \cite[Corollary 4.2, Remark 4.3]{Hitch}) that
\[
\lim_{s \to 1^-}\|u\|_s^2 = |\nabla u|_2^2  \quad\forall u \in H^1(\R^n),
\]
and in order to simplify the presentation of some statements, with a slight abuse of notation, we will denote by $(-\Delta)^1$ the usual Laplace operator $-\Delta$ and with $\|u\|_1^2 = |\nabla u|_2^2$ the usual $H^1$-seminorm. 

We recall that the fractional Laplacian and the fractional Sobolev spaces $H^s(\R^n)$ can also be defined via the Fourier transform for every $s > -\frac{n}{2}$. When $s \in (0,1)$, this definition is equivalent to the standard one (see e.g. \cite[Proposition 3.3, Proposition 3.4]{Hitch}).

We introduce also the homogeneous Sobolev spaces $\mathcal{D}^s(\R^n)$, defined as as the completion of $C^\infty_c(\R^n)$ with respect to the norm $\|\cdot \|_s$. When $n > 2s$ it holds that $\mathcal{D}^s(\R^n) \hookrightarrow L^{2^*_s}(\R^n)$ and also the usual Sobolev and Rellich-Kondrakov embeddings hold true (see e.g. \cite[Theorem 6.7, Corollary 7.2]{Hitch}).

\end{subsection}

\begin{subsection}{Fractional Sobolev constant and Dirichlet eigenvalues}
Let us recall the definition of the best Sobolev constant for the embedding $\mathcal{D}^s(\R^n) \hookrightarrow L^{2^*_s}(\R^n)$,
\begin{equation}
\label{nonlocSob}
S_s := \inf_{u \in \mathcal{D}^s(\R^n)\setminus\{0\}} \frac{\|u\|^2_s}{|u|_{2_s^*}^{2}}.
\end{equation}

The value of $S_s$ is explicitly known (see {\cite[Theorem 1.1]{CotTav}}) and it is bounded, both form above and from below, by two positive constants depending only on $n$ (and hence not on $s$).  
When $n>2s$, the infimum \eqref{nonlocSob} is achieved only by functions of the family 
\[
k \frac{1}{(\mu^2 + |x-x_0|^2)^{\frac{n-2s}{2}}},
\]
where $k \in \R$, $\mu>0$ and $x_0 \in \R^n$. 
In particular, if we take 
\begin{equation}\label{optbubbcost}
k = k_\mu := \left[S_s^{\frac{n}{2s}}\mu^{n}\left(\int_{\R^n}\frac{1}{(1+|x|^2)^n}\de x\right)^{-1}\right]^{\frac{1}{2^*_s}}
\end{equation}
then the functions
\begin{equation}\label{eq:bubble}
U_{x_0, \mu}(x) : = k_\mu \frac{1}{(\mu^2 + |x-x_0|^2)^{\frac{n-2s}{2}}},
\end{equation}
also known as ``standard bubbles'', satisfy the equation
\[
(-\Delta)^s U_{x_0, \mu} = {U_{x_0, \mu}}^{2^*_s-1} \quad \text{in } \R^n
\]
for all $\mu >0$, $x_0 \in \R^n$ and 
\[
\|U_{x_0, \mu}\|^2_s = |U_{x_0, \mu}|^{2^*_s}_{2^*_s} = S_s^{\frac{n}{2s}}.
\]
The following estimates have a central role in the present work (for the proofs see \cite[Proposition 12]{Ser} and \cite[Proposition 21, 22]{ValSer}).
\begin{prop}\label{stimebubble}
Let $s\in(0,1)$ and $n>2s$. Let $\Omega \subset \R^n$ be a domain and let $x_0 \in \Omega$ and $\rho>0$ be such that $B_{4\rho}(x_0) \subset \Omega$. Let $\varphi \in C^\infty_c(B_{2\rho}(x_0); [0,1])$ be such that $\varphi \equiv 1$ in $B_\rho(x_0)$.
Let be
\[
u^s_\varepsilon (x) := \varphi(x)\varepsilon^{-\frac{n-2s}{2}}U_{x_0, \mu}^s\left(\frac{x-x_0}{\varepsilon}+x_0\right)
\]
where $U_{x_0, \mu}$ is as in \eqref{eq:bubble}.  
Then the following estimates hold:
\begin{equation}\label{stime}
\begin{aligned}
&\|u^s_\varepsilon\|_s^2 \leq S_s^{\frac{n}{2s}} + C \varepsilon^{n-2s} \\
&S_s^{\frac{n}{2s}}-C\varepsilon^n\leq |u^s_\varepsilon|_{2^*_s}^{2^*_s} \leq S_s^{\frac{n}{2s}}\\
&0\leq|u^s_\varepsilon|_{2^*_s-1}^{2^*_s-1} \leq C \varepsilon^{\frac{n-2s}{2}}\\
&0\leq|u^s_\varepsilon|_1 \leq C\varepsilon^{\frac{n-2s}{2}}\\
&|u^s_\varepsilon|_2^2 \geq
\begin{cases}
C\varepsilon^{2s} - C\varepsilon^{n-2s} & \text{if } n>4s \\
C\varepsilon^{2s}|\ln \varepsilon| + C\varepsilon^{2s} & \text{if } n=4s \\
C\varepsilon^{n-2s} - C\varepsilon^{2s} & \text{if } n<4s 
\end{cases}
\end{aligned}
\end{equation}
where all the constants are positive and depend on $n$, $\mu$, $x_0$, $\rho$ and $s$.
\end{prop}
\begin{rem}
Since the quantities $S_s$, $C_{n,s}$, and $\frac{C_{n,s}}{s(1-s)}$ are uniformly bounded with respect to $s \in (0,1)$ then an elementary computation shows that, for any fixed $0<s_0<s_1 \leq 1$ and $n>4s_1$, the constant appearing in the previous proposition are uniformly bounded with respect to $s \in (s_0, s_1)$.
\end{rem}

Another quantity which plays a central role in this work is the first eigenvalue of the $s$-Laplacian under homogeneous Dirichlet conditions, whose variational characterization is given by
\[
\lambda_{1,s} := \inf_{u \in X^s_0(\Omega)\setminus\{0\}} \frac{\|u\|_s^2}{|u|^2_2}.
\]

We recall also the fractional Poincar\'e inequality (see e.g., \cite[Proposition 2.7]{braparsqu}): for every $u \in X^s_0(\Omega)$ it holds that $C |u|_2 \leq \|u\|_s$, where the constant $C>0$ depends only on $n$, $s$ and $\text{diam }\Omega$. 
As a matter of fact, since $\frac{C_{n,s}}{s(1-s)}$ is uniformly bounded with respect to $s \in (0,1)$, it follows that the constant $C$ is uniformly bounded when $s$ is close to one. This implies that for every $s_0 \in (0,1)$ we have
\begin{equation}\label{eigenbounds}
\underline \lambda(s_0) := \inf_{s \in [s_0,1)} \lambda_{1,s} >0. 
\end{equation}
Moreover, recall that the following basic fact holds (see e.g. \cite[Lemma 3.5]{Bogdan}): for every $\varphi \in C^\infty_c(\R^n)$ it holds 
\begin{equation}\label{bogdy}
|(-\Delta)^s \varphi (x) | \leq C(|\varphi|_\infty + |D^2\varphi|_\infty) \frac{1}{(1 + |x|)^{n+2s}}\quad \forall x \in \R^n. 
\end{equation}
We point out that a simple computation shows that the constant $C$ depends only on $n$ and $\text{supp }\varphi$, but not on $s$. 
Moreover, we have that for every $u \in \mathcal{D}^s(\R^n)$ and $\varphi \in C^\infty_c(\Omega)$ it holds that
\begin{equation}\label{veryweak}
(u, \varphi)_s = \int_{\R^n}u (-\Delta)^s\varphi \de x, 
\end{equation}
and a simple computation shows that
\begin{equation}\label{eigenbounds2}
\overline \lambda := \sup_{s \in (0,1)} \lambda_{1,s} < \infty. 
\end{equation}
\end{subsection}
\begin{subsection}{Regularity of solutions}
We collect here some regularity results that will be used through the paper. 
First of all, we recall that, by \cite[Theorem 3.2]{IanMosSqua}, every weak solution $ u \in X^s_0(\Omega)$ of Problem \eqref{fracBrezis} belongs to $L^\infty(\R^n)$. The following result is a consequence of \cite[Corollary 2.4, Corollary 2.5]{HolderReg}, {\cite[Proposition 2.8, Proposition 2.9]{sylv}} and  \cite[Lemma 2.2]{uniquenondeg}, \cite[Lemma 4.4]{Yannick}:\\
\begin{prop}
Let $\Omega\subset \R^n$ be a domain. Let $u \in \mathcal{D}^s(\R^n) \cap L^\infty(\R^n)$ be a weak solution of $(-\Delta)^s u = g$ in $\Omega$. Then for every $K' \subset \subset K \subset \subset \Omega$ the following hold:
\begin{enumerate}[(a)]
\item Let be $s_0 \in (0,1)$ and $s \in [s_0, 1)$. Assume that $g \in L^\infty(\Omega)$. Then $u \in C^{0,s}(K')$ and it holds that
\begin{equation}\label{soave1}
\|u\|_{0, s;K'}\leq C(|u|_\infty + |g|_{\infty;K}),
\end{equation}
for a constant $C>0$ depending on $n$, $K$, $K'$ and $s_0$.
\item Let be $s_0 \in \left(\frac{2}{3}, 1\right)$ and $s \in [s_0, 1)$. Assume that $g \in C^{0,s}(\overline\Omega)$. Then $u \in C^{2,3s-2}(K')$. Moreover,
\begin{equation}\label{soave2}
\|u\|_{2, 3s-2; K'}\leq C(|u|_\infty + \|g\|_{0,s ;K}),
\end{equation}
for a constant $C>0$ depending only on $n$, $K$, $K'$ and $s_0$.
\end{enumerate}
\end{prop}

\begin{rem}\label{gilbcont}
Let $s_0 \in (0,1)$, let $(s_j) \subset [s_0, 1)$ and let $(\Omega_j)$ be a family of domains such that $\Omega_j \subset \Omega_{j+1}$, which invades $\R^n$ as $j \to +\infty$. Assume now that $(u_j)$ and $(g_j)$ are two families such that $u_j \in H^{s_j}(\R^n) \cap L^\infty(\R^n)$ and $g_j \in L^\infty(\R^n)$, which satisfy in the weak sense $(-\Delta)^{s_j} u_j = g_j$ in $\Omega_j$. 
Then fixing two compact sets $K_1 \subset \subset  K_2 \subset\subset \R^n$, we have that $u_j$ satisfy $(-\Delta)^{s_j}u_j = g_j$ definitely in $K_2$, and then by \eqref{soave1} we get that 
\[
\|u_j \|_{0, s_j; K_1} \leq C (|u_j|_\infty + |g_j|_{\infty; K_2})
\]
where $C>0$ depends only on $s_0$, $K_1$ and $K_2$. 

If $(u_j)$ and $(g_j)$ are uniformly bounded in $L^\infty(\R^n)$, this implies that $\|u_j\|_{0, s_0; K_1} \leq C$ where $C$ does not depend on $j$. Hence, thanks to \cite[Lemma 6.36]{giltru} we have that 
\[
u_j \to u \quad \text{in }C^{0, \alpha}(K_1)
\]
for any fixed $0<\alpha < s_0$. If in addiction $s_0 > \frac{2}{3}$ and $\|g_j\|_{0, s_j; K_2} \leq C$, with the same argument as before and using \eqref{soave2} we can prove that

\[
\|u_j \|_{2, 3s_0-2; K_1} \leq C  \quad \text{ and } \quad u_j \to u \quad \text{in }C^{2, \alpha}(K_1)
\]
for any fixed $0<\alpha < 3s_0 -2$.

\end{rem}

We conclude this subsection recalling the following:
\begin{teo}[{\cite[Proposition 1.1, Theorem 1.2]{HolderReg}}]\label{boundregRO}
Let $\Omega$ be a bounded $C^{1,1}$ domain, $g \in L^\infty(\Omega)$, let $u$ be a solution of
\[
\begin{cases}
(-\Delta)^s u = g & \text{in }\Omega, \\
u = 0 & \text{in }\R^n \setminus \Omega,
\end{cases}
\]
and $\delta(x) := \text{dist}(x, \partial \Omega)$. Then the following holds.
\begin{enumerate}
\item $u \in C^s(\R^n)$, 
\item the function $\frac{u}{\delta^s}_{|\Omega}$ can be continuously extended to $\overline \Omega$. Moreover, we have $\frac{u}{\delta^s} \in C^\alpha(\overline \Omega)$ and 
\[
\left\| \frac{u}{\delta^s}\right\|_{0, \alpha; \overline \Omega} \leq C |g|_{\infty; \Omega}
\]
for some $\alpha>0$ satisfying $\alpha < \min\{s, 1-s\}$. The constant $\alpha$ and $C$ depend only on $\Omega$ and $s$.
\end{enumerate}
\end{teo}
\begin{rem}
The constant $C$ appearing in Thereom \ref{boundregRO} is not, in general, bounded as $s \to 1^-$.
\end{rem}
\end{subsection}

\begin{subsection}{Extension properties for the fractional Laplacian}

We introduce now the extension properties of $\mathcal{D}^s(\R^n)$ functions. All results are well known and can be found in \cite{CS, FrLe1, FrLe2}. 

Let $ s \in (0,1)$ and $n > 2s$. We set $\R^{n+1}_+ := \R^n \times \R_+$, we write $z \in \R^{n+1}_+$ as $z= (x, y)$ where $x \in \R^n$ and $y >0$, and we set $|z|=|(x,y)|:=\sqrt{x^2+y^2}$. We define $\mathcal{D}^{1,s}(\R^{n+1}_+)$ as the completion of $C^\infty_c(\overline {\R^{n+1}_+})$ with respect to the quadratic form 
\[
D^2_s(u) := d_s \int \int_{\R^{n+1}_+} y^{1-2s} |\nabla u|^2 \de x \de y, 
\]
where 
\[
d_s := \frac{2^{2s}}{2}\frac{\Gamma\left(s\right)}{\Gamma\left(1-s\right)}.
\]

Let $P_{n,s}: \R^{n+1}_+ \to \R$ be the function defined by
\[
P_{n,s}(x, y) := p_{n,s} \frac{y^{2s}}{(y^2+|x|^2)^{\frac{n+2s}{2}}}, 
\] 
where 
\[
p_{n,s} := \frac{\Gamma\left( \frac{n+2s}{2}\right)}{\pi^{\frac{n}{2}}\Gamma(s)}
\]
is such that $p_{n,s}\int_{\R^n}\frac{y^{2s}}{(y^2 + |x|^2)^{\frac{n+2s}{2}}} \de x = 1$ for every $y>0$. 

Given $u \in \mathcal{D}^s(\R^n)$, we define the extension $E_s u: \R^{n+1}_+ \to \R$ of $u$ as the function
\begin{equation}\label{extensiondefinition}
E_s u(x, y) := \int_{\R^n}P_{n,s}(x-\xi, y) u(\xi) \de \xi.
\end{equation}

\begin{prop}\label{extension}
Let $s \in (0,1)$. The following properties holds:
\begin{enumerate}
\item If $u \in \mathcal{D}^s(\R^n)$, then $E_s u \in \mathcal{D}^{1,s}(\R^{n+1}_+)$ and satisfies
\begin{equation}\label{energyext}
D^2_s(E_s u) = \| u\|_s^2.
\end{equation}
Moreover $E_s u $ is a weak solution to the problem 
\[
-\text{div }\left( y^{1-2s} \nabla U\right) = 0 \quad \text{in }\R^{n+1}_+,
\]
and satisfies 
\[
\lim_{\varepsilon \to 0^+}\|E_s u(\cdot , \varepsilon) - u \|_s =  0.
\]
In addition, it holds that
\begin{equation}\label{dualconv}
\lim_{\varepsilon \to 0^+} \left \| \left(-d_s \varepsilon^{1-2s}\frac{\partial E_s u}{\partial y}(\cdot, \varepsilon)\right) - (-\Delta)^s u\right \|_{-s} = 0.
\end{equation}

\item There exists a unique linear bounded operator $T$, such that $T: \mathcal{D}^{1,s}(\R^{n+1}_+) \to \mathcal{D}^s(\R^n)$ and $Tu(x,y) = u(x, 0)$ whenever $u \in C^\infty_c(\overline{\R^{n+1}_+})$.
Moreover, the following inequality holds for all $u \in \mathcal{D}^{1,s}(\R^{n+1}_+)$:
\begin{equation}\label{traceinequality}
D^2_s(u) \geq \| Tu\|^2_s. 
\end{equation}

\item The extension $E_su$ is unique: if a function $U$ is such that $TU(x,y) = u(x)$ and it satisfies the properties in (1), then $U = E_s u$. On the other hand, the equality in \eqref{traceinequality} is attained if and only if $ u = E_s f$ for some $f \in \mathcal{D}^s(\R^n)$. 
\end{enumerate}
\end{prop}

A consequence of the previous proposition is the following:
\begin{lemma}\label{extreg}
\begin{enumerate}[(i)]
\item If $u \in C^{0, s}(\R^n)$, then $E_s u \in C^2(\R^{n+1}_+) \cap C^{0, s}(\overline{\R^{n+1}_+})$;
\item If $u \in \mathcal{D}^s(\R^n)$, then 
\[
\lim_{\varepsilon \to 0^+}\int_{\R^n}\left(-2sd_s \frac{E_s u (x, \varepsilon) - E_s u(x, 0)}{\varepsilon^{2s}}\right)\varphi(x) \de x = (u, \varphi)_s \quad \forall \varphi \in C^\infty_c(\R^n),
\]
\item Moreover, if $u \in H^s(\R^n)$, then
\[
\lim_{\varepsilon \to 0^+}\int_{\R^n}\left(-d_s \varepsilon^{1-2s} \frac{\partial E_s u}{\partial y}(x,\varepsilon)\right)\varphi(x) \de x = (u, \varphi)_s \quad \forall \varphi \in C^\infty_c(\R^n). 
\]
\item For every $u \in H^s(\R^n)$ and $\varphi \in \mathcal{D}^{1,s}(\R^{n+1}_+)$ it holds
\[
d_s\int_{\R^{n+1}_+} y^{1-2s} \nabla E_s u \cdot \nabla \varphi \de x \de y = (u, T\varphi)_s.
\]
\end{enumerate}
\end{lemma}

We conclude this subsection by recalling the following version of the strong maximum principle.

\begin{prop}[{\cite[Remark 4.2]{Yannick}}]\label{yannick}
Let $u: \R^{n+1}_+\to \R$ be a weak solution of 
\[
\begin{cases}
-div (y^{1-2s}\nabla u) \geq 0 & B_R^+,\\
-y^{1-2s} \frac{\partial u}{\partial y} \geq 0 & \Gamma^0_R,\\
u \geq 0 & \Gamma^+_R,
\end{cases}
\]
where
\[
\begin{aligned}
&B^+_R = \{ (x,y) \in \R^{n+1}_+ \ | \ y>0, |(x,y)| < R\}\\
&\Gamma^+_R =  \{ (x,y) \in \R^{n+1}_+ \ | \ y\geq0, |(x,y)| = R\}\\
&\Gamma^0_R =  \{ (x,y) \in \partial \R^{n+1}_+ \ | \ |x| < R\}
\end{aligned}
\]

Then either $u >0$ or $u \equiv 0$ on $B_R^+ \cup \Gamma^0_R$.
\end{prop}

\end{subsection}
\begin{subsection}{Miscellanea}
We conclude this section by recalling the fractional Strauss lemma for radial functions.
\begin{prop}[{\cite[Proposition 1]{choozawa}}]\label{strauss}
Let $n \geq 2$ and $s \in \left(\frac{1}{2}, 1\right)$. Then, for all $u \in \mathcal{D}^s(\R^n)$ such that $u = u(|x|)$, it holds
\begin{equation}\label{straussineq}
\sup_{x \in \R^n \setminus\{0\}}|x|^{\frac{n-2s}{2}}|u(x)| \leq K_{n,s} \|u\|^2_s
\end{equation}
where
\[
K_{n,s} = \left(\frac{\Gamma(2s-1)\Gamma\left( \frac{n-2s}{2}\right)\Gamma \left(\frac{n}{2} \right)}{2^{2s}\pi^{\frac{n}{2}}\Gamma(s)^2\Gamma \left( \frac{n-2(1-s)}{2}\right)}\right).
\]
\end{prop}



\end{subsection}
\end{section}

\begin{section}{Existence of sign changing solutions}
In this Section we prove the existence of sign-changing solutions for Problem \eqref{fracBrezis}. Since through all this section the parameters $s \in(0,1)$ and $\lambda \in (0, \lambda_{1,s})$ are fixed, we will often omit them in the subscripts.\\[4pt] 

\noindent Let us define the functional $I_{s, \lambda}: X_0^s(\Omega)\to \R$ as
\[
I_{s, \lambda}(u) =  I(u) := \frac{1}{2}(\|u\|_s^2-\lambda|u|^2_2)- \frac{1}{2^*_s}|u|_{2^*_s}^{2^*_s}
\]
Every critical point of $I$ is a solution of Problem \eqref{fracBrezis}, in fact we have that 
\[
I'(u) [\varphi] = (u, \varphi)_s - \lambda \int_\Omega u \varphi \de x - \int_\Omega |u|^{2^*_s-2}u\varphi \de x.
\]
Let us consider the Nehari manifold
\[
\mathcal{N}_{s, \lambda} = \mathcal{N} := \left\{u \in X^s_0(\Omega) \ |\ u \not \equiv 0, I'(u)[u]=0 \right\},
\]
and define $c_\mathcal{N}(s, \lambda) = c_\mathcal{N} := \inf_\mathcal{N} I(u)$. 
In the case of $\Omega = B_R$ we define also the radial Nehari manifold as
\[
\mathcal{N}_{s, \lambda;rad} = \mathcal{N}_{rad} := \{ u \in \mathcal{N} \ | \ u \text{ is radial}\},
\]
and we set $c_{\mathcal{N}_{rad}}(s, \lambda) = c_{\mathcal{N}_{rad}} := \inf_{\mathcal{N}_{rad}} I(u)$.

\begin{rem}
The functional $I$ is even, i.e. $I(-u) = I(u)$ for any $u \in X^s_0(\Omega)$, and hence, without loss of generality, if $u$ is a critical point of $I$, we can always assume that $u(0) \geq 0$.
\end{rem}

Let us consider also the functional $J_{s, \lambda}: X^s_0(\Omega)\setminus \{0\}\to \R$ defined by
\[
J_{s, \lambda}(u) = J(u) := \frac{\|u\|_s^2-\lambda|u|^2_2}{|u|^2_{2^*_s}},
\]
and set 
\[
S_{s, \lambda} := \inf_{X^s_0(\Omega)\setminus\{0\}}J(u). 
\]
\begin{rem}\label{valdinergia}
As proved in \cite[Section 4]{ValSer} if $n\geq 4s$, $s \in (0, 1)$, then $S_{s, \lambda} < S_s$, for every $\lambda \in (0, \lambda_{1,s})$. 
\end{rem}

\begin{prop}\label{posSol}
Let $n \geq 4s$. Then there exists $u^0 \in \mathcal{N}$ such that $I(u^0) = c_\mathcal{N}$ and $u^0 >0$ in $\Omega$. Furthermore, it holds that
\[
c_\mathcal{N} = \frac{s}{n}S_{s, \lambda}^{\frac{n}{2s}}.
\] 
If $\Omega = B_R$ then $u^0$ is also radially symmetric and decreasing as a function of the radius and $c_{\mathcal{N}} = c_{\mathcal{N}_{rad}}$.
\end{prop}
\begin{proof}

From the results of \cite[Proposition 20, Chapter 4]{ValSer}, we know that there exists a minimizer $\overline u \in X^s_0(\Omega)\setminus \{0\}$ for the functional $J$. Moreover, since in general $J(|u|) \leq J(u)$, such a minimizer has to be non negative. After a rescaling (notice that $J(u) = J(Ku)$ for every $K>0$), we have that there exists $\hat K$ such that $u^0 := \hat K\overline u \in \mathcal{N}$ and $I'(u^0) = 0$, so that $u^0$ is a solution of Problem \eqref{fracBrezis}. Then we can apply the fractional strong maximum principle (see e.g., \cite[Corollary 4.2]{Musina}) and infer that $u^0 > 0$ in $\Omega$. We observe that if $u \in \mathcal{N}$ it holds
\[
I(u) = \frac{s}{n}(J(u))^{\frac{n}{2s}}.
\]
Therefore $u^0$ is also a minimizer of $I$ in $\mathcal{N}$, and we obtain that $c_\mathcal{N} = \frac{s}{n}S_{s, \lambda}^{\frac{n}{2s}}$. Finally, since by \cite[Theorem 3.2]{IanMosSqua} we have that $u^0 \in L^\infty(\R^n)$, when $\Omega = B_R$ we can apply \cite[Theorem 4.1]{BirkWako}, and hence in this case $u^0$ is radially symmetric and decreasing. The proof is complete.
\end{proof}

If $u \in X^s_0(\Omega)$ we denote as usual by $u^+$, $u^-$, respectively, the positive and the negative parts of $u$, i.e. the functions defined by
\[
\begin{aligned}
&u^+ (x):= \max(u(x), 0) &&  x \in \Omega, \\
&u^-(x) := \max(-u(x), 0) &&  x \in \Omega,
\end{aligned}
\]
so that $u = u^+ - u^-$ and $|u| = u^+ + u^-$. We define the nodal Nehari set as
\[
\mathcal{M}_{s, \lambda} = \mathcal{M} := \{u \in X_0^s(\Omega) \ |\ u^\pm \not \equiv 0, I'(u)[u^\pm] = 0 \},
\]
and when $\Omega = B_R$ we define also the radial nodal Nehari set as
\[
\mathcal{M}_{s, \lambda; rad} = \mathcal{M}_{rad} := \{ u \in \mathcal{M} \ | \ u \text{ is radial}\}.
\]

Let $u \in \mathcal{M}$, by definition we have 
\[
0 = I'(u)[u^+] = (u,u^+)_s - \lambda|u^+|^2_2 - |u^+|^{2^*_s}_{2^*_s} = \|u^+\|_s^2 - (u^-, u^+)_s  - \lambda|u^+|^2_2 - |u^+|^{2^*_s}_{2^*_s}.
\]
Setting $\Omega^+:= \{ u>0\}$ and $\Omega^- := \{ u < 0\}$ we observe that
\[
(u^-, u^+)_s = - \frac{C_{n,s}}{2}\int_{\Omega^+\times \Omega^-} \frac{u^+(x)u^-(y)}{|x-y|^{n+2s}}\de x \de y - \frac{C_{n,s}}{2}\int_{\Omega^-\times \Omega^+} \frac{u^+(y)u^-(x)}{|x-y|^{n+2s}}\de x \de y.
\]
Now, defining the function $\eta_s : X^s_0(\Omega) \to [0, +\infty)$ as 
\begin{equation}\label{etadefinition}
\eta_s(u) = \eta(u) := \frac{C_{n,s}}{2} \int_{\R^{2n}}\frac{u^+(x)u^-(y)}{|x-y|^{n+2s}}\de x \de y,
\end{equation}
we conclude that if $u\in \mathcal{M}$ then $\eta(u)>0$ and 
\begin{equation}\label{eqcarattnehari}
\|u^\pm\|_s^2- \lambda|u^\pm|^2_2 = |u^\pm|^{2^*_s}_{2^*_s} - 2 \eta (u).
\end{equation}
Motivated by that, we define the functionals $f_{s, \lambda}^\pm: X^s_0(\Omega) \to \R$ as
\begin{equation}\label{caratterizzazione}
f^\pm_{s, \lambda}(u) = f^\pm(u) := 
\begin{cases}
0 & \text{if } u^\pm  = 0,\\
\frac{|u^\pm|^{2^*_s}_{2^*_s} - 2 \eta (u)}{\|u^\pm\|_s^2- \lambda|u^\pm|^2_2} & \text{if } u^\pm \neq 0,
\end{cases}
\end{equation}
and we give a characterization of the nodal Nehari set as
\[
\mathcal{M} = \{ u \in X^s_0(\Omega) \ |\ f^+(u) = 1 = f^-(u) \}.
\]
\begin{rem}
We observe that $\mathcal{M} \subset \mathcal{N}$ and $\mathcal{M} \neq \emptyset$. The first fact is obvious, for the second we observe that for every sign-changing function $u \in X^s_0(\Omega)$ we can always find $\alpha, \beta >0$ such that $\alpha u^+-\beta u^- \in \mathcal{M}$  by solving the system
\[
\begin{cases}
\alpha^{2^*_s-2}|u^+|^{2^*_s}_{2^*_s} - \frac{\beta}{\alpha}2\eta(u) = \|u^+\|_s^2- \lambda|u^+|^2_2,\\
\beta^{2^*_s-2}|u^-|^{2^*_s}_{2^*_s} - \frac{\alpha}{\beta}2\eta(u) = \|u^-\|_s^2- \lambda|u^-|^2_2.
\end{cases}
\]
\end{rem}
Let us define
\[
c_\mathcal{M}(s,\lambda) = c_\mathcal{M} = \inf_{u \in \mathcal{M}}I(u),
\]
and similarly
\[
c_{\mathcal{M}_{rad}}(s,\lambda) = c_{\mathcal{M}_{rad}} = \inf_{u \in \mathcal{M}_{rad}}I(u).
\]

\begin{teo}\label{conditionedexistence}
Let $n>2s$ and $\lambda \in (0, \lambda_{1,s})$. If 
\begin{equation}\label{condition}
c_{\mathcal{M}} < c_\mathcal{N} + S_s^{\frac{n}{2s}},
\end{equation}
there exists a sign-changing solution $u \in \mathcal{M}$ of Problem \eqref{fracBrezis} such that $I(u) = c_{\mathcal{M}}$.
If $\Omega = B_R$ and 
\begin{equation}\label{conditionrad}
c_{\mathcal{M}_{rad}} < c_\mathcal{N} + S_s^{\frac{n}{2s}},
\end{equation}
there exists a radial sign-changing solution $u \in \mathcal{M}_{rad}$ of Problem \eqref{fracBrezis} such that $I(u) = c_{\mathcal{M}_{rad}}$.
\end{teo}
\begin{proof}
We divide the proof in several steps. Let us set
\[
\mathcal{V}_{s, \lambda} = \mathcal{V} := \left\{ u \in X^s_0(\Omega) \ |\ |f^\pm(u) - 1 | < \frac{1}{2}\right\},
\]
where $f^\pm$ is defined in \eqref{caratterizzazione}.
Since $\lambda \in (0, \lambda_{1,s})$, if $u \in \mathcal{V}$ then $u^\pm \not \equiv 0$, and 
\[
|u^\pm|^{2_s^*-2}_{2_s^*} \geq \frac{S_s}{2}\left(1 - \frac{\lambda}{\lambda_{1,s}}\right)> 0. 
\]

\begin{cla}\label{Compactness}
If $(u_j) \subset \mathcal{V}$ is a sequence such that
\[
I(u_j) \to c \quad \text{and} \quad I'(u_j) \to 0\quad \text{in }X^{-s}_0(\Omega) \quad \text{as }j \to +\infty, 
\]
and if we assume that $c$ satisfies
\[
c < c_\mathcal{N} + \frac{s}{n}S_s^{\frac{n}{2s}}
\]
then $(u_j)$ is strongly relatively compact in $X^s_0(\Omega)$.\\ 
\end{cla}

The proof is standard. Indeed, it is sufficient to argue as in \cite[Theorem 1, Claim 2-3]{ValSer} to show that $(u_j)$ is bounded in $X^s_0(\Omega)$ and $u_j \rightharpoonup u$ in $X^s_0(\Omega)$, where $u$ is a solution of Problem \eqref{fracBrezis}, and to conclude we can argue as in \cite[Lemma 3.1]{CSS}, taking into account the non local term $\eta$ defined in \eqref{etadefinition}.\\ 

Let us denote by $\mathcal{C}_P$ the cone of non-negative functions in $X^s_0(\Omega)$, and let $\Sigma$ be the set of maps $\sigma$ such that
\[
\begin{cases}
\sigma \in C(Q, X^s_0(\Omega)) & \text{where }Q = [0,1]\times[0,1] \\
\sigma(s, 0) = 0 & \forall s \in [0,1]\\
\sigma(0,t) \in \mathcal{C}_P & \forall t \in [0,1]\\
\sigma(1, t) \in -\mathcal{C}_P & \forall t \in [0,1]\\
f^+(\sigma(s,1))+f^-(\sigma(s,1)) \geq 2 & \forall s \in [0,1]\\
I(\sigma(s,1)) < 0 & \forall s \in [0,1]
\end{cases}
\]
We have that $\Sigma \neq \emptyset$. For instance, take $u \in \mathcal{M}$ and consider
\[
\sigma(s,t) = t((1-s)\alpha u^+ - s\alpha u^-);
\]
if $\alpha>0$ is large enough then $\sigma \in \Sigma$.\\

\begin{cla} \label{Miranda}
We claim that
\[
\inf_{\sigma \in \Sigma}\sup_{u \in \sigma(Q)}I(u) = \inf_{u \in \mathcal{M}}I(u).
\]
\end{cla}
We begin the proof of the Step by showing that
\[
\forall\  \overline u \in \mathcal{M} \quad \exists\ \overline \sigma \in \Sigma \ \ \hbox{s.t.}\ I(\overline u) = \sup_{u \in \overline \sigma (Q)}I(u). 
\]
Indeed, given $\overline u \in \mathcal{M}$ we know that there exists a map $\overline \sigma \in \Sigma$ such that
\begin{equation}\label{maps}
\begin{cases}
\overline \sigma(Q) \subset A:=\{\alpha \overline u^+ - \beta \overline u^- \ |\ \alpha, \beta \geq 0\},\\
\exists x_0 \in Q \ | \ \overline \sigma(x_0) = \overline u,
\end{cases}
\end{equation}
and thus we readily get that $I (\overline u) \leq \sup_{u \in \overline \sigma(Q)}I(u) \leq \sup_{u \in A}I(u)$. 
Moreover, for every $\alpha, \beta \geq 0$, we have
\[
I(\alpha  \overline u^+ - \beta  \overline u^-) = \left(\frac{\alpha^2}{2}- \frac{\alpha^{2^*_s}}{2^*_s}\right)| \overline u^+|_{2^*_s}^{2^*_s} + \left(\frac{\beta^2}{2}- \frac{\beta^{2^*_s}}{2^*_s}\right)| \overline u^-|_{2^*_s}^{2^*_s} - (\alpha-\beta)^2 \eta ( \overline u).
\]
Since the maximum of the function $f(t) = \frac{t^2}{2}- \frac{t^{2^*_s}}{2^*_s}$ for $t \geq 0$ is attained for $t = 1$, and $\eta(\overline u) \geq0$, we infer that
\[
\sup_{u \in A}I(u) = \sup_{\alpha, \beta \geq 0}I(\alpha \overline  u^+ - \beta \overline  u^-) \leq \frac{s}{n}| \overline u|_{2^*_s}^{2^*_s} = I(\overline u) \quad \forall \alpha, \beta.
\]
which proves the claim.
To conclude the proof of Step \ref{Miranda}, we can argue as in the proof of \cite[Lemma 3.2]{CSS} by using Miranda's Theorem (see e.g. \cite{Mirranda}). The proof of Step \ref{Miranda} is complete.\\

\begin{cla}\label{PSstrLemma}
Consider a minimizing sequence $(\overline u_j) \subset \mathcal{M}$ and denote by $\overline \sigma_j$ the corresponding sequence of maps in the class $\Sigma$ satisfying \eqref{maps}. By Step \ref{Miranda} it holds that
\[
\lim_{j\to +\infty} \max_{\overline \sigma_j(Q)}I(u) = \lim_{j\to +\infty} I(\overline u_j) = c_\mathcal{M}.
\]

We claim that there exists $(u_j) \subset X^s_0(\Omega)$ such that
\begin{equation}
\label{PSstr}
\begin{aligned}
&\lim_{j\to +\infty} \text{d}(u_j, \overline \sigma_j(Q))=0,\\
&\lim_{j\to +\infty} I'(u_j) = 0 && \text{in }X^{-s}_0(\Omega),\\
&\lim_{j \to +\infty}I(u_j) = c_\mathcal{M}.
\end{aligned}
\end{equation}
\end{cla}

The proof is essentially the one contained in \cite[Theorem A]{CSS} and is based on a standard deformation lemma argument (see e.g., \cite[Lemma 1]{Hofer},\cite[Theorem 3.4]{Struwe}), Step \ref{Compactness} and Step \ref{Miranda}.

\begin{cla}
Proof of the existence.
\end{cla}

Let $(\overline u_j) \subset \mathcal{M}$ be a minimizing sequence for $c_\mathcal{M}$ and let $(u_j) \subset X^s_0(\Omega)$ be the associated sequence built in Step \ref{PSstrLemma}. By \eqref{PSstr} and recalling \eqref{maps} we know that there exists a sequence $(v_j)$ which can be written in the form 
\[
v_j = \alpha_j \overline u^+_j - \beta_j \overline u^-_j \in \overline \sigma_j(Q)
\]
where $\alpha_j$, $\beta_j \geq 0$, such that
\[
\text{dist}(u_j, v_j) \to 0.
\]
Notice that, arguing as in \cite[Claim 2]{ValSer} we get that $(\overline u_j)$ is bounded in $X_0^s(\Omega)$. Then, by definition and Cauchy's inequality we get that 
\[
\eta(\overline{u}_j) = -(\overline{u}_j^+, \overline{u}_j^-) \leq \|\overline{u}_j^+\|_s\|\overline{u}_j^-\|_s \leq \frac{1}{2}(\|\overline{u}^+_j\|_s^2 + \|\overline{u}_j^-\|_s^2) \leq \frac{\|\overline{u}_j\|_s^2}{2} \leq \infty.
\]
Taking this into account, we can follow the proof of \cite[Theorem A]{CSS} and infer that $u_j \in \mathcal{V}$ for $j$ large enough. 
Therefore, thanks to hypothesis \eqref{condition}, we can apply Step \ref{Compactness}, hence $u_j \to u \in X^s_0(\Omega)$, where $u$ is such that $I(u) = c_\mathcal{M}$ and $I'(u) = 0$. Then $u$ is a critical point for $I$ and is a solution of Problem \eqref{fracBrezis}. Since also $u_j^{\pm}\to u^\pm$ strongly in $X^s_0(\Omega)$ and $u_j \in \mathcal{V}$, we deduce that $u^\pm \not \equiv 0$. In particular, using $u^\pm$ as test functions we obtain
\[
0= I'(u)[u^\pm] = \pm\|u^\pm\|^2 \mp \lambda |u^\pm|^2_2 \mp | u^\pm|^{2^*}_{2^*} \pm 2\eta(u),
\]
and then $u \in \mathcal{M}$. The proof of the first part is then complete. Since the proof of the radial case is identical to the previous one, we omit it.
\end{proof}

In the next Lemma we show that condition \eqref{condition} and \eqref{conditionrad} are satisfied, respectively, when $n \geq 6s$ and $n>6s$.
\begin{lemma}\label{EnergyBound}
Let $s \in (0,1)$, $\lambda \in (0, \lambda_{1,s})$. If $n\geq 6s$ then
\[
c_{\mathcal{M}}< c_{\mathcal{N}} + \frac{s}{n}S_s^{\frac{n}{2s}}.
\]

If $\Omega = B_R$ and $n > 6s$, then
\[
c_{\mathcal{M}_{rad}}< c_{\mathcal{N}} + \frac{s}{n}S_s^{\frac{n}{2s}}.
\] 

\end{lemma}
\begin{proof}
Thanks to Step \ref{Miranda} of Theorem \ref{conditionedexistence} it suffices to show that
\[
\sup_{\alpha, \beta\geq 0} I(\alpha u^0 - \beta u_\varepsilon) < c_\mathcal{N} + \frac{s}{n}S_s^{\frac{n}{2s}},
\]
where $u^0$ is as in Proposition \ref{posSol} and $u_\varepsilon$ is as in Proposition \ref{stimebubble}.

First of all we notice that
\[
\frac{1}{2}\|\alpha u^0 - \beta u_\varepsilon\|_s^2 \leq \alpha^2 \|u^0\|_s^2 + \beta^2 \|u_\varepsilon\|_s^2.
\]
Thanks to the properties of $u^0$ and by \eqref{stime} we get that, if we take $\varepsilon < 1$, 
\begin{equation}\label{boundcrit}
\|\alpha u^0 - \beta u_\varepsilon\|_s^2 \leq \left( 1-\frac{\lambda}{\lambda_{1,s}}\right)^{-1}S_{s,\lambda}^{\frac{n}{2s}}\alpha^2 + (S_s^{\frac{n}{2s}} + C_1\varepsilon^n)\beta^2 \leq C_2(\alpha+\beta)^2,
\end{equation}
where $C_1>0$ is as in \eqref{stime} and $C_2 >0$ depends on $n$, $s$, and $\lambda$, but not on $\varepsilon$.  
Let us focus now on the $L^{2^*_s}$-norm. By mean value theorem and the fundamental theorem of calculus we obtain that
\begin{equation}
\label{stimacrit}
\begin{aligned}
&\left||\alpha u^0 -\beta u_\varepsilon|^{2^*_s}_{2^*_s} -  |\alpha u^0|^{2^*_s}_{2^*_s} -  |\beta u_\varepsilon|^{2^*_s}_{2^*_s}\right|\\
\leq &\ C_* \left( \int_\Omega |\alpha u^0|^{2^*_s-1}|\beta u_\varepsilon| \de x + \int_\Omega |\beta u_\varepsilon|^{2^*_s-1}|\alpha u^0| \de x\right).
\end{aligned}
\end{equation}
where $C_* = 2^*_s(2^*_s-1)\max\{1, 2^{2^*_s-3}\}$. 
Since $u^0 \in L^\infty(\R^n)$, by Young's inequality and \eqref{stime} we get that
\[
\begin{aligned}
&\left||\alpha u^0 -\beta u_\varepsilon|^{2^*_s}_{2^*_s} -  |\alpha u^0|^{2^*_s}_{2^*_s} -  |\beta u_\varepsilon|^{2^*_s}_{2^*_s}\right| \\
\leq&\ C_* |\alpha u^0|^{2^*_s-1}_\infty | \beta u_\varepsilon|_1 + C_* |\alpha u^0|_\infty |\beta u_\varepsilon|_{2^*_s-1}^{2^*_s-1}  \\
\leq&\ \frac{\theta}{2} |\alpha u^0|^{2^*_s}_{\infty} +  C_\theta \beta^{2^*_s} \varepsilon^{n} + C_* |\alpha u^0|_\infty \beta^{2^*_s-1}\varepsilon^{\frac{n-2s}{2}},
\end{aligned}
\]
where $\theta \in (0,1)$ will be chosen later. Here and in the following, $C_\theta$, $C^\prime_\theta$, and $C''_\theta$, denote positive constants depending on $n$, $s$ and $\theta$, and such that $C_\theta, C'_\theta, C''_\theta \to +\infty$ as $\theta \to 0^+$.
Applying Young's inequality again we obtain that for any sufficiently small $\varepsilon>0$ 
\[
\begin{aligned}
&\left||\alpha u^0 -\beta u_\varepsilon|^{2^*_s}_{2^*_s} -  |\alpha u^0|^{2^*_s}_{2^*_s} -  |\beta u_\varepsilon|^{2^*_s}_{2^*_s}\right| \\
\leq &\ \theta |\alpha u^0|^{2^*_s}_{\infty} + C_\theta\beta^{2^*_s} \varepsilon^{n} + C'_\theta\beta^{2^*_s}\varepsilon^{\frac{n(n-2s)}{n+2s}}\ \leq \theta |\alpha u^0|^{2^*_s}_{\infty} + C''_\theta\beta^{2^*_s} \varepsilon^{\frac{n(n-2s)}{n+2s}},
\end{aligned}
\]
so that
\[
\begin{aligned}
|\alpha u^0 -\beta u_\varepsilon|^{2^*_s}_{2^*_s} \geq &\ |\alpha u^0|^{2^*_s}_{2^*_s} - \theta|\alpha u^0|^{2^*_s}_\infty + |\beta u_\varepsilon|^{2^*_s}_{2^*_s} - C_\theta\beta^{2^*_s} \varepsilon^{\frac{n(n-2s)}{n+2s}} \\
=&\ \alpha^{2^*_s}\left(|u^0|^{2^*_s}_{2^*_s} - \theta|u^0|^{2^*_s}_\infty\right) + \beta^{2^*_s}\left(|u_\varepsilon|^{2^*_s}_{2^*_s} - C''_\theta\varepsilon^{\frac{n(n-2s)}{n+2s}}\right).
\end{aligned}
\]
Taking $\theta\in(0,1)$ such that $|u^0|^{2^*_s}_{2^*_s} - \theta|u^0|^{2^*_s}_\infty >0$ and taking $\tilde C>0$ such that
$|u^0|^{2^*_s}_{2^*_s} - \theta|u^0|^{2^*_s}_\infty \geq \tilde C>0$, so that $\tilde C$ and $\theta$ depend on $n$, $\lambda$ and $s$, then using again \eqref{stime} we infer that
\[
\begin{aligned}
|\alpha u^0 -\beta u_\varepsilon|^{2^*_s}_{2^*_s} \geq&\ \tilde  C\alpha^{2^*} + \beta^{2^*}\left(S^{\frac{n}{2s}}_s - C_1\varepsilon^n - C''_\theta\varepsilon^{\frac{n(n-2s)}{n+2s}}\right)\\
\geq&\ \hat C(\alpha^{2^*} + \beta^{2^*})  \geq \hat C2^{1-2^*_s}(\alpha + \beta)^{2^*},
\end{aligned}
\]
for $\varepsilon$ small enough so that 
\[
S^{\frac{n}{2s}}_s - C_1\varepsilon^n - C''_\theta\varepsilon^{\frac{n(n-2s)}{n+2s}} \geq \hat C := \min\left\{\tilde C, \frac{S_s^{\frac{n}{2s}}}{2}\right\}.
\]
This implies, together with \eqref{boundcrit}, that there exists $C_3$, $C_4 >0$ which depends only on $n$, $s$ and $\lambda$ such that 
\[
I(\alpha u^0 - \beta u_\varepsilon) \leq C_3(\alpha + \beta)^2 (C_4 - (\alpha + \beta)^{2_s^*-2})
\]
and then if $(\alpha + \beta)^{2_s^*-2} \geq C_4$ we get $I(\alpha u^0 - \beta u_\varepsilon) \leq 0$. Hence we can restrict to $\alpha$ and $\beta$ such that $\alpha + \beta \leq C_4^{\frac{1}{2_s^*-2}}$. Using again \eqref{stimacrit} we get that
\[
\begin{aligned}
I(\alpha u^0- \beta u_\varepsilon)\leq &\  \frac{\alpha^2}{2}(\|u^0\|_s^2-\lambda |u^0|^2_2) + \frac{\beta^2}{2}\|u_\varepsilon\|_s^2 -\alpha\beta\left[(u^0, u_\varepsilon)_s-\lambda \int_\Omega u^0u_\varepsilon \de x\right]  \\
 -&\ \lambda \frac{\beta^2}{2} |u_\varepsilon|_2^2 - \frac{\alpha^{2^*_s}}{2^*_s}|u^0|^{2^*_s}_{2^*_s}- \frac{\beta^{2^*_s}}{2^*_s}|u_\varepsilon|^{2^*_s}_{2^*_s}+C_5\int_\Omega |u_\varepsilon|^{2^*_s-1}u^0 \de x + C_5\int_\Omega |u^0|^{2^*_s-1}u_\varepsilon \de x,
\end{aligned}
\]
where $C_5$ depends on $C_*$ and $C_4$. 
Since $u^0$ is a solution of Problem \eqref{fracBrezis} and $u_0 \in L^\infty(\R^n)$, we obtain that
\[
\begin{aligned}
I(\alpha u^0- \beta u_\varepsilon) \leq& \left(\frac{\alpha^2}{2}- \frac{\alpha^{2^*_s}}{2^*_s}\right)|u^0|^{2^*_s}_{2^*_s} + \frac{\beta^2}{2}\|u_\varepsilon\|_s^2 +\alpha\beta |u^0|_{\infty;B_\rho(x_0)}^{2^*_s-1}|u_\varepsilon|_1 \\
 -&\ \lambda \frac{\beta^2}{2} |u_\varepsilon|_2^2 - \frac{\beta^{2^*_s}}{2^*_s}|u_\varepsilon|^{2^*_s}_{2^*_s}+C_5|u_\varepsilon|_{2^*_s-1}^{2^*_s-1}|u^0|_{\infty;B_\rho(x_0)} + C_5|u^0|_{\infty;B_\rho(x_0)}^{2^*_s-1}|u_\varepsilon|_1 
\end{aligned}
\]
where $x_0 \in \Omega$ and $\rho>0$ are as in the definition of $u_\varepsilon$. Recalling that $\sup_{\alpha \geq 0}\left(\frac{\alpha^2}{2}- \frac{\alpha^{2^*_s}}{2^*_s}\right) \leq \frac{s}{n}$, and since $u^0 \in \mathcal{N}$ implies that $\frac{s}{n}|u^0|^{2^*_s}_{2^*_s} = I(u^0) = c_\mathcal{N}$, we deduce that
\[
\begin{aligned}
I(\alpha u^0- \beta u_\varepsilon) \leq &\  c_\mathcal{N} + \frac{\beta^2}{2}\|u_\varepsilon\|_s^2 - \frac{\beta^{2^*_s}}{2^*_s}|u_\varepsilon|^{2^*_s}_{2^*_s} \\
 -&\ \lambda \frac{\beta^2}{2} |u_\varepsilon|_2^2 +C_5|u^0|_{\infty;B_\rho(x_0)}|u_\varepsilon|_{2^*_s-1}^{2^*_s-1} + C_6|u^0|^{2^*_s-1}_{\infty;B_\rho(x_0)}|u_\varepsilon|_1, 
\end{aligned}
\]
where $C_6$ comes from $C_5$ and $C_4$.
Now, using \eqref{stime} and since $\sup_{\beta \geq 0}\left(\frac{\beta^2}{2}- \frac{\beta^{2^*_s}}{2^*_s}\right) \leq \frac{s}{n}$ we get that
\[
\begin{aligned}
I(\alpha u^0- \beta u_\varepsilon) \leq &\ c_\mathcal{N}+ \frac{s}{n}S_s^{\frac{n}{2s}}+  C_7\varepsilon^{n} \\
+&\ C_7\lambda (\varepsilon^{n-2s} - C_8\varepsilon^{2s}) + C_7(|u^0|_{\infty;B_\rho(x_0)} + |u^0|_{\infty;B_\rho(x_0)}^{2^*_s-1})\varepsilon^{\frac{n-2s}{2}}. 
\end{aligned}
\]
Once again $C_7$ and $C_8$ depend only on $n$, $s$ and $\lambda$. 
Since $\varepsilon <<1$ and $\lambda \leq \lambda_{1,s}$, this leads to 

\begin{equation}\label{technicality2.0}
I(\alpha u^0-\beta u_\varepsilon) \leq  c_\mathcal{N} + \frac{s}{n}S_s^{\frac{n}{2s}}+ C_7(|u^0|_{\infty;B_\rho(x_0)} + |u^0|_{\infty;B_\rho(x_0)}^{2^*_s-1})\varepsilon^{\frac{n-2s}{2}}- C_9 \lambda \varepsilon^{2s}.
\end{equation}
Since $C_7$ and $C_9$ do not depend on $\varepsilon$, when $n > 6s$, we can always take $\varepsilon$ small enough so that 
\begin{equation}\label{anothertechnicality}
C_7(|u^0|_{\infty;B_\rho(x_0)} + |u^0|_{\infty;B_\rho(x_0)}^{2^*_s-1})\varepsilon^{\frac{n-2s}{2}}- C_9 \lambda \varepsilon^{2s} <0
\end{equation}
and thus we get the thesis.

If $n = 6s$ the sign of the left-hand side in \eqref{anothertechnicality} does not depend on $\varepsilon$ anymore.
Nevertheless, a careful analysis of the proof of Proposition \ref{stimebubble} (in particular of the estimates in \cite[Proposition 12]{Ser}, \cite[Proposition 21, 22]{ValSer}), and of the previous passages, shows that there exists $\overline \tau \in (0,1)$ which depends only on $n$ and $s$ but not on $\rho$ nor $x_0$ such that, taking $n = 6s$, $0<\rho <1$, $\mu = \rho$ and $\varepsilon = \tau \rho$ with $\tau \in (0, \overline \tau)$, inequality \eqref{technicality2.0} can be written as
\[
I(\alpha u^0-\beta u_\varepsilon) \leq  c_\mathcal{N} + \frac{1}{6}S_s^{3}+ (\tilde C_1(|u^0|_{\infty;B_\rho(x_0)} + |u^0|_{\infty;B_\rho(x_0)}^{2^*_s-1})- \tilde C_2 \lambda) \tau^{2s},
\]
where $\tilde C_1$ and $\tilde C_2$ depend only on $n$ and $s$ but not on $\rho$ nor $x_0$. At the end we obtain the desired result observing that, since $u^0$ decreases along the radii, the point $x_0$ and the ball $B_\rho(x_0)$ can be chosen near the boundary of $\Omega$ in such the way that $|u^0|_{\infty;B_\rho(x_0)}$ is so small so that $$\tilde C_1(|u^0|_{\infty;B_\rho(x_0)} + |u^0|_{\infty; B_\rho(x_0)}^{2^*-1})- \tilde C_2\lambda<0.$$ 
In the case of radial functions, if $n=6s$, this last argument fails since we are forced to choose $x_0 = 0$ in the definition of the function $u_\varepsilon(x)$, while the rest of the proof applies verbatim and thus we get existence of radial solutions just for $n>6s$. The proof is complete.
\end{proof}

From Theorem \ref{conditionedexistence} and Lemma \ref{EnergyBound} we obtain the following. 

\begin{teo}\label{exsimm}
Let $s \in (0,1)$, and let $n \geq 6s$, $\lambda \in (0, \lambda_{1,s})$. Then there exists a sign-changing solution $u \in \mathcal{M}$ of Problem \eqref{fracBrezis} such that $I(u) = c_{\mathcal{M}}$.
If $\Omega = B_R$, $n > 6s$ and $\lambda \in (0, \lambda_{1,s})$, there exists radial sign-changing solution $u \in \mathcal{M}_{rad}$ of Problem \eqref{fracBrezis} such that $I(u) = c_{\mathcal{M}_{rad}}$.
\end{teo}

\end{section}

\begin{section}{Asymptotic analysis of the energy as $\lambda\to 0^+$}

In this section we study the asymptotic behavior of the energy of least energy solutions of Problem \eqref{fracBrezis}, as $\lambda \to 0^+$. 

\begin{rem}\label{ss:as}
We observe that, as a straightforward consequence of the definitions of $S_s$, $S_{s, \lambda}$, and $\lambda_{1,s}$, we get that
\[
\left(1-\frac{\lambda}{\lambda_{1,s}}\right)S_{s} \leq S_{s,\lambda}\leq S_s,
\]
so that $S_{s, \lambda}\to S_s$ as $\lambda \to 0^+$.
Moreover, since $S_s$ is uniformly bounded with respect to $s \in (0,1)$ and taking into account \eqref{eigenbounds}, for every $s_0 \in (0,1)$ it holds that
\[
\lim_{\lambda \to 0^+ }\sup_{s \in [s_0, 1)} |S_s - S_{s, \lambda}|\leq \left(\sup_{s \in (0,1)}S_s \right)\lim_{\lambda \to 0^+} \frac{\lambda}{\underline \lambda(s_0)}= 0.
\]
\end{rem}

We begin with studying the asymptotics of the quantities $c_\mathcal{N}(s, \lambda)$, $c_\mathcal{M}(s, \lambda)$ and $c_{\mathcal{M}_{rad}}(s, \lambda)$.

\begin{lemma}\label{eneras}
Let $s \in (0,1)$ and let $n \geq 6s$. As $\lambda \to 0^+$ it holds 
\[
c_\mathcal{N}(s, \lambda) \to  \frac{s}{n}S_s^{\frac{n}{2s}} \quad \text{ and } \quad c_\mathcal{M}(s,\lambda) \to 2 \frac{s}{n}S_s^{\frac{n}{2s}}.
\]
Moreover, for every $s_0 \in (0, 1)$ it holds that 
\[
\hbox{(i)}\ \ \lim_{\lambda \to 0^+}\sup_{s \in [s_0, 1)}\left|\frac{s}{n}S_s^{\frac{n}{2s}} - c_{\mathcal{N}}(s, \lambda)\right| = 0 \quad \text{and}\quad \hbox{(ii)} \ \
\lim_{\lambda \to 0^+}\sup_{s \in [s_0, 1)}\left|2\frac{s}{n}S_s^{\frac{n}{2s}} - c_{\mathcal{M}}(s, \lambda)\right| = 0
\]
If $\Omega = B_R$ and $n>6s$ the same results hold for $c_{\mathcal{M}_{rad}}(s, \lambda)$.
\end{lemma}
\begin{proof} 
By Proposition \ref{posSol} we have $c_\mathcal{N}(s, \lambda) = \frac{s}{n}S_{s, \lambda}^{\frac{n}{2s}}$, and thus $(i)$ is a consequence of Remark \ref{ss:as}. 
In fact, recalling that $S_s$ is uniformly bounded with respect to $s \in (0,1)$ and thanks to \eqref{eigenbounds}, we get that 
\begin{equation}\label{uniformasintenergy}
0 \leq \frac{s}{n}S_s^{\frac{n}{2s}} - c_{\mathcal{N}}(s, \lambda) \leq \frac{s}{n}\left(1 - \left(1- \frac{\lambda}{\lambda_{1,s}}\right)^{\frac{n}{2s}}\right)S_s^{\frac{n}{2s}} \leq C(s_0, \lambda)
\end{equation}
where $C(s_0, \lambda) >0$ is such that $C(s_0,\lambda) \to 0$ as $\lambda \to 0^+$. 
 
For $(ii)$, let us recall that by Lemma \ref{EnergyBound} it holds $c_\mathcal{M}(s, \lambda) < c_\mathcal{N}(s, \lambda) + \frac{s}{n}S_s^{\frac{n}{2s}}$.
Let $u_{s,\lambda}$ be a minimizer of $c_\mathcal{M}(s, \lambda)$. As seen in the proof of Step \ref{Miranda} of Theorem \ref{conditionedexistence}, we have that, for every $\alpha$, $\beta \in \R_+$ it holds
\[
I_{s,\lambda}(\alpha u^+_{s,\lambda} - \beta u^-_{s,\lambda}) \leq I_{s,\lambda}(u_{s,\lambda}) = c_\mathcal{M}(s,\lambda).
\]
On the other hand, we can always choose $\alpha$ and $\beta$ such that $\alpha u^+_{s,\lambda}$, $\beta u^-_{s,\lambda} \in \mathcal{N}_{s,\lambda}$, and since $\eta_s(u_{s, \lambda}) >0$ we get that
\[
I_{s,\lambda}(\alpha u^+_{s,\lambda} - \beta u^-_{s,\lambda}) > I_{s,\lambda}(\alpha u^+_{s,\lambda}) + I_{s,\lambda}(\beta u^-_{s,\lambda}) \geq 2 c_\mathcal{N}(s, \lambda).
\]
At the end we obtain
\begin{equation}\label{energyrelation}
2c_\mathcal{N}(s, \lambda) < c_\mathcal{M}(s, \lambda)< c_\mathcal{N}(s, \lambda) + \frac{s}{n}S_s^{\frac{n}{2s}},
\end{equation}
and the result easily follows. Indeed, since \eqref{energyrelation} can be rewritten as 
\[
0 < 2\frac{s}{n}S_s^{\frac{n}{2s}} - c_{\mathcal{M}}(s, \lambda) < 2 \left(\frac{s}{n}S_s^{\frac{n}{2s}} - c_{\mathcal{N}}(s, \lambda) \right),
\]
the limit is uniform with respect to $s \in [s_0, 1)$ thanks to \eqref{uniformasintenergy}. In the radial case the proof is identical and we omit it. 
\end{proof}

\begin{lemma}\label{As:energas}
Let $s \in (0,1)$ and let $n \geq 6s$. Let $(u_{s, \lambda}) \subset \mathcal{M}_{s, \lambda}$ be a family of solutions of Problem \eqref{fracBrezis} such that $I_{s,\lambda}(u_{s,\lambda}) = c_\mathcal{M}(s, \lambda)$ and set $M_{s, \lambda, \pm} := |u^\pm_{s, \lambda}|_\infty$. As $\lambda \to 0^+$ we have:
\begin{enumerate}[(i)]
\item $\|u_{s,\lambda}^\pm\|_s^2 \to S_s^{\frac{n}{2s}} $;
\item $|u_{s,\lambda}^\pm|_{2_s^*}^{2_s^*} \to S_s^{\frac{n}{2s}} $;
\item $\lambda|u_{s, \lambda}^\pm|_2^2 \to 0$;
\item $\eta_s(u_{s,\lambda}) \to 0$ ;
\item $u_{s,\lambda} \rightharpoonup 0$ in $X^s_0(\Omega)$;
\item $M_{s, \lambda,\pm} \to + \infty$. 
\end{enumerate}
where $\eta_s$ is as in \eqref{etadefinition}.
When $n>6s$ the same results hold for a family $(u_{s, \lambda}) \subset \mathcal{M}_{s, \lambda; rad}$ of radial solutions of Problem \eqref{fracBrezis} such that $I_{s, \lambda}(u_{s,\lambda}) = c_{\mathcal{M}_{rad}}(s, \lambda)$.
Moreover, for every $s_0 \in (0,1)$ the limits $(i)-(iv)$ are uniform with respect to $s \in [s_0,1)$. 
\end{lemma}
\begin{proof}
Let $u_{s,\lambda} \in \mathcal{M}_{s, \lambda}$. From the definition of $\mathcal{M}_{s, \lambda}$ (see also \eqref{eqcarattnehari}), \eqref{nonlocSob} and the variational characterization of the eigenvalues, we get that
\[
\begin{aligned}
0 =&\ \|u^\pm_{s,\lambda}\|_s^2+2\eta_s (u_{s,\lambda}) - \lambda |u^\pm_{s,\lambda}|^2_2-|u^\pm_{s,\lambda}|_{2_s^*}^{2_s^*}\\
\geq&\ \|u^\pm_{s,\lambda}\|_s^2 \left(\left(1-\frac{\lambda}{\lambda_{1,s}}\right)- S_s^{-\frac{2_s^*}{2}} \|u^\pm_{s,\lambda}\|_s^{2_s^*-2}\right),
\end{aligned}
\]
so that
\begin{equation}\label{energyext6}
\liminf_{\lambda \to 0^+} \|u^\pm_{s,\lambda}\|_s^2 \geq S_s^{\frac{n}{2s}}.
\end{equation}
Since
\begin{equation}\label{energysplit}
\|u_{s,\lambda}\|^2_s = \|u^+_{s,\lambda}\|^2_s + \|u^-_{s,\lambda}\|^2_s + 4\eta_s(u_{s,\lambda}) \geq \|u^+_{s,\lambda}\|^2_s + \|u^-_{s,\lambda}\|^2_s
\end{equation}
it follows that
\begin{equation}\label{energyext1}
\liminf_{\lambda \to 0^+} \|u_{s,\lambda}\|_s^2 \geq 2S_s^{\frac{n}{2s}}.
\end{equation}
On the other hand, since $u_{s,\lambda} \in \mathcal{N}_{s, \lambda}$ and $I_{s, \lambda}(u_{s, \lambda}) = c_\mathcal{M}(s, \lambda)$, thanks to Lemma \ref{eneras} we have
\[
\lim_{\lambda \to 0^+}\frac{s}{n}|u_{s,\lambda}|_{2_s^*}^{2_s^*} = \lim_{\lambda \to 0^+}I_{s, \lambda}(u_{s,\lambda}) =\lim_{\lambda \to 0^+} c_\mathcal{M}(s,\lambda) = 2 \frac{s}{n}  S_s^{\frac{n}{2s}}, 
\]
and then
\begin{equation}\label{energyext2}
\lim_{\lambda \to 0^+}|u_{s,\lambda}|_{2_s^*}^{2_s^*} = 2S_s^{\frac{n}{2s}}.
\end{equation}
Using again that $u_{s, \lambda} \in \mathcal{N}_{s, \lambda}$ and the characterization of the eigenvalues we get that
\[
\left(1-\frac{\lambda}{\lambda_{1,s}}\right) \|u_{s,\lambda}\|_s^2 \leq |u_{s,\lambda}|_{2_s^*}^{2_s^*}.
\]
From the previous inequality, \eqref{energyext1} and \eqref{energyext2} it follows that
\begin{equation}\label{energyext3}
\lim_{\lambda \to 0^+}\|u_{s,\lambda} \|_s^2 = 2S_s^{\frac{n}{2s}}.
\end{equation}
Therefore, from \eqref{energyext2}, \eqref{energyext3} and since $u_{s,\lambda} \in \mathcal{N}_{s,\lambda}$ we deduce $(iii)$.

Now observe that, in view of \eqref{energyext6} and \eqref{energysplit}, we have
\[
\begin{aligned}
2S_s^{\frac{n}{2s}} &=& \lim_{\lambda \to 0^+} \|u_{s,\lambda} \|_s^2 \geq \limsup_{\lambda \to 0^+} \left(\|u^+_{s,\lambda}\|_s^2 + \|u^-_{s,\lambda}\|_s^2\right)\\
&\geq & \liminf_{\lambda \to 0^+} \|u^+_{s,\lambda}\|_s^2+ \liminf_{\lambda \to 0^+} \|u^-_{s,\lambda}\|_s^2 \geq 2S_s^{\frac{n}{2s}}.
\end{aligned}
\]
Hence we obtain that
\[
\lim_{\lambda \to 0^+} \left(\|u^+_{s,\lambda}\|_s^2 + \|u^-_{s,\lambda}\|_s^2 \right) = 2S_s^{\frac{n}{2s}}, 
\]
and, in view of \eqref{energyext6}, we deduce that
\[
\lim_{\lambda \to 0^+} \|u^\pm_{s,\lambda}\|_s^2 = S_s^{\frac{n}{2s}},
\]
which proves $(i)$, and $(iv)$ follows from \eqref{energysplit} and \eqref{energyext3}. Then, the relation $(ii)$ is a consequence of $(i)$, $(iii)$, $(iv)$, and the definition of $\mathcal{M}_{s,\lambda}$. 

For $(v)$, from \eqref{energyext3} we get that, up to a subsequence, there exists $u_s \in X^s_0(\Omega)$ such that $u_{s, \lambda} \rightharpoonup u_s$ in $X^s_0(\Omega)$ and $u_{s,\lambda} \to u_s$ a.e. as $\lambda \to 0^+$. Moreover, $u_s$ is a weak solution of the equation
\begin{equation}\label{eqsob}
(-\Delta)^s u_s = |u_s|^{2^*_s-2}u_s \quad \text{in }\Omega.
\end{equation}
In addiction, by $(ii)$ and Fatou's Lemma we have
\begin{equation}\label{fatouext}
|u_s|^{2^*_s}_{2^*_s} \leq \liminf_{\lambda \to 0^+}|u_{s,\lambda}|^{2^*_s}_{2^*_s} = 2 S_s^{\frac{n}{2s}}.
\end{equation}
Suppose that both $u^+_s$ and $u^-_s$ are not trivial. Then, using $u^\pm_s$ as test functions in \eqref{eqsob} we get
\[
\|u_s^\pm\|^2_s + 2 \eta_s (u_s) = |u^\pm_s|^{2^*_s}_{2^*_s},
\]
and then, by definition of $S_s$, we deduce that
\[
S_s \leq \frac{\|u^{\pm}_s\|^2_s}{|u^\pm_{s}|^{2}_{2^*_s}} = |u^{\pm}_s|^{2^*_s-2}_{2^*_s}-2\frac{\eta_s(u_s)}{|u^\pm_s|^2_{2^*_s}} < |u^{\pm}_s|^{2^*_s-2}_{2^*_s}.
\]
This implies that
\[
2S_s^{\frac{n}{2s}}< |u_s|^{2^*_s}_{2^*_s},
\]
which contradicts \eqref{fatouext}. As a consequence, either $u^+_{s} \equiv 0$ or $u^-_s \equiv 0$ i.e. $u_s$ is of constant sign. Assume for instance that $u_s \geq 0$. Hence, being $u_s \in L^\infty(\R^n)$ thanks to \cite[Theorem 3.2]{IanMosSqua}, non-negative, and a solution of \eqref{eqsob}, then $(v)$ is a consequence of the fractional Pohozaev identity (see \cite[Corollary 1.3]{Pohozaev}). 

To prove the last point of the Lemma, we argue again by contradiction. Let $C>0$ be such that $M_{s,\lambda,+} \leq C$ for all $\lambda$. Then $|u_{s,\lambda}^+|^{2_s^*} \leq (M_{s,\lambda,+})^{2_s^*} \leq C^{2_s^*}$. Since by the previous point we have also that $u_{s,\lambda}^+ \to 0$ a.e, we can apply Lebesgue's convergence theorem to obtain that, as $\lambda \to 0^+$, 
\[
|u_{s,\lambda}^+|_{2_s^*}^{2_s^*} = \int_\Omega |u^+_{s,\lambda}|^{2_s^*}\de x \to 0,
\]  
which contradicts $(ii)$. The same proof holds for $M_{s,\lambda,-}$. 

As for the radial case, the proof is identical. Finally, since $S_s$ is uniformly bounded with respect to $s \in (0,1)$, and in view of \eqref{eigenbounds}, \eqref{eigenbounds2}, and Lemma \ref{eneras}, the limits $(i)-(iv)$ are uniform with respect to $s \in [s_0, 1)$.
\end{proof}

\section{Nodal components of the extension and nodal bounds}
In this section we study the nodal set of the extension of least energy sign-changing solutions of Problem \eqref{fracBrezis}. Let $u_{s, \lambda}$ be such a solution and let $W_{s, \lambda} = E_s u_{s,\lambda}$ be the extension of $u_{s, \lambda}$ (see Sect. \ref{SectionIntro}). Since $W_{s, \lambda}$ is continuous up to the boundary (see Lemma \ref{extreg}) and its restriction to $\R^n$ is $u_{s,\lambda}$, then also $W_{s, \lambda}$ changes sign. Next result states that the number of nodal regions of $W_{s, \lambda}$, i.e., the number of the connected components of $\left\{x \in \R^{n+1}_+ \ | \ W_{s, \lambda}(x) \neq 0\right\}$, is two.

\begin{teo}\label{3nodal}
Let $s \in (0,1)$, $n\geq 6s$ and let $\Omega \subset \R^n$ be a smooth bounded domain. Then, there exists $0<\hat \lambda_s \leq \lambda_{1,s}$ such that for every $\lambda \in (0, \hat \lambda_s)$ the function $W_{s,\lambda}$ has exactly two nodal regions. Moreover, for every $s_0 \in (0,1)$ there exists $\hat \lambda(s_0)>0$ which depends on $n$ and $s_0$ but not on $s$ such that for every $\lambda \in (0, \hat \lambda(s_0))$ and $s \in [s_0, 1)$, previous result holds.  

\end{teo}
\begin{proof}
Let $\{\Omega_i\}$ be the set of the nodal regions of $W_{s,\lambda}$ in $\R^{n+1}_+$ and for each of them let us set $W^i_{s, \lambda} := W_{s, \lambda} \mathbbm{1}_{\overline{\Omega_i}}$, where $\mathbbm{1}_{\overline{\Omega_i}}$ is the characteristic function of $\overline{\Omega_i}$. 
First of all we notice that it cannot happen that $W^i_{s,\lambda}(x, 0 ) = 0$ for all $x \in \R^n$. Indeed, by \eqref{traceinequality} we have
\[
\begin{aligned}
D^2_s(W_{s,\lambda}) =&\ d_s \int_{\R^{n+1}_+}y^{1-2s} |\nabla W^i_{s,\lambda}|^2 \de x \de y + d_s \int_{\R^{n+1}_+}y^{1-2s} |\nabla (W_{s,\lambda}-W_{s, \lambda}^i)|^2 \de x \de y  \\
 \geq&\ d_s \int_{\R^{n+1}_+}y^{1-2s} |\nabla (W_{s,\lambda}-W_{s,\lambda}^i)|^2 \de x \de y \geq \|(W_{s,\lambda}-W^i_{s,\lambda})(x, 0)\|_s^2\\
= &\ \|W_{s,\lambda}(x, 0)\|_s^2,  
\end{aligned}
\]
so that, thanks to \eqref{energyext}, $\|u_{s,\lambda}\|_s^2 = d_s \int_{\R^{n+1}_+}y^{1-2s} |\nabla (W_{s,\lambda}-W^i_{s,\lambda})|^2 \de x \de y$. Since the extension is unique, this implies that $W_{s,\lambda} = W_{s,\lambda} - W_{s,\lambda}^i$ and then $W_{s, \lambda}^i \equiv 0$ in $\R^{n+1}_+$, which contradicts the definition of $\Omega_i$ and proves the claim.

As a consequence, we have that there is no nodal region such that $\overline{\Omega_i} \cap \Omega = \emptyset$. Using also \eqref{traceinequality}, we get that $W^i_{s, \lambda}(x, 0)$ is a non trivial function in $X^s_0(\Omega)$. Moreover, thanks the continuity of $W_{s, \lambda}$, the support of $W^i_{s, \lambda}(x, 0)$ turns out to be a non empty union of subsets of $\Omega$ where $u_{s, \lambda}$ has the same sign. In addiction, for every $i, j$ the intersection between the supports of $W_{s, \lambda}^i(x, 0)$ and $W^j_{s, \lambda}(x, 0)$ consists of a set of null measure.
 
Since $u_{s,\lambda}$ is a solution of Problem \eqref{fracBrezis}, from $(iv)$ of Lemma \ref{extreg}, we obtain that for every $\phi \in \mathcal{D}^{1,s}(\R^{n+1}_+)$ such that $\phi(x, 0) \in X^s_0(\Omega)$
\begin{equation}\label{weakextension}
\begin{aligned}
&d_s\int_{\R^{n+1}_+} y^{1-2s} \nabla W_{s,\lambda}(x,y) \cdot \nabla \phi (x,y) \de x \de y = (u_{s,\lambda} (x), \phi(x, 0))_s \\
=&\ \lambda \int_\Omega u_{s,\lambda}(x) \phi(x, 0) \de x + \int_\Omega |u_{s,\lambda}(x)|^{2_s^*-2}u_{s,\lambda}(x) \phi (x, 0) \de x.
\end{aligned}
\end{equation}

Then, using $W^i_{s,\lambda}$ as a test function in \eqref{weakextension}, we have
\[
d_s\int_{\R^{n+1}_+} y^{1-2s} |\nabla W^i_{s,\lambda}|^2 \de x \de y = \lambda \int_\Omega |W^i_{s,\lambda}(x,0)|^2 \de x + \int_\Omega |W^i_{s,\lambda}(x,0)|^{2_s^*}\de x.
\]
Therefore, by \eqref{traceinequality}, the Sobolev inequality and the variational characterization of $\lambda_{1,s}$ we obtain 
\[
0 \leq D^2_s(W^i_{s,\lambda}) \left[ - \left(1 - \frac{\lambda}{\lambda_{1,s}}\right) + S_s^{-\frac{2_s^*}{2}}D^{\frac{2s}{n-2s}}_s (W^i_{s,\lambda})\right],
\]
and, as $\lambda \to 0^+$, we get that
\[
D^2_s(W^i_{s,\lambda}) \geq S_s^{\frac{n}{2s}}(1 + o(1)).
\]

At the end, let $K$ be the number of nodal regions of $W_{s,\lambda}$, and assume that $K>2$. Thus by Lemma \ref{As:energas} and Proposition \ref{extension} we obtain that 
\begin{equation}\label{3nodaltecnic}
2 S_s^{\frac{n}{2s}} + o(1) = \|u_{s,\lambda}\|^2_s = D^2_s(W_{s,\lambda}) = \sum_{i=1}^K D^2_s (W_{s,\lambda}^i) \geq K S_s^{\frac{n}{2s}}(1 + o(1)),
\end{equation}
which gives a contradiction.

For the last point of the Theorem, let us fix $s_0 \in (0, 1)$. As seen in \eqref{eigenbounds}, there exists $\underline \lambda(s_0)$ such that $\underline \lambda(s_0) \leq \lambda_{1,s}$ for every $s \in [s_0, 1)$. Then, when $\lambda \in \left( 0, \underline \lambda(s_0)\right)$, existence of solutions is ensured by Theorem \ref{exsimm}.
Moreover, as stated in Lemma \ref{As:energas} we have that 
\[
\sup_{s \in [s_0, 1)}\left|S_s^{\frac{n}{2s}}- \|u^\pm_{s, \lambda}\|_s^2\right| \leq C_1(\lambda) \quad \text{ and } \sup_{s \in [s_0, 1)}\eta_s(u_{s, \lambda}) \leq C_2(\lambda),
\]
where the functions $C_1,$ $C_2$ depend on $n$ and $s_0$ but not on $s$, and are such that $C_1(\lambda)$, $C_2(\lambda) \to 0$ as $\lambda \to 0^+$.  
Then, when $\lambda < \underline \lambda(s_0)$, from \eqref{3nodaltecnic} we deduce that   
\[
2 S_s^{\frac{n}{2s}} + C_3(\lambda) > K \left(1- \frac{\lambda}{\lambda_{1,s}}\right) S_s^{\frac{n}{2s}} \geq K \left(1- \frac{\lambda}{\underline \lambda(s_0)}\right) S_s^{\frac{n}{2s}}
\]
where $ C_3(\lambda)$ still depends only on $n$, $s_0$ and $\lambda$. Then, recalling that $S_s$ is uniformly bounded from below by a positive constant when $s \in (0,1)$, we obtain that
\[
2 +  2 o(\lambda) > K \left(1- \frac{\lambda}{\underline \lambda(s_0)}\right)
\]
where $o(\lambda)$ does not depend on $s$. Clearly, if $K >2$, there exists a sufficiently small $\tilde \lambda(s_0)$ such that a contradiction holds. Therefore the only possibility is that $K=2$ for all $\lambda \in (0,\hat\lambda(s_0))$, where $0<\hat\lambda(s_0)<\min\{\tilde \lambda(s_0), \underline \lambda(s_0)\}$. The proof is complete.

\end{proof}

The previous result holds true for least energy sign-changing solutions of Problem \eqref{fracBrezis} in general domains, but gives information just for the nodal set of their extensions. For radial solutions we can say more. 

\begin{teo}\label{twicechange}
Let $n>6s$, $s \in (0,1)$, and $R>0$. Let $u_{s,\lambda}$ be a least energy radial sign-changing solution for Problem \eqref{fracBrezis} in $B_R$. If $\lambda \in (0, \hat \lambda_s)$ where $\hat \lambda_s>0$ is the number given by Theorem \ref{3nodal}, then $u_{s, \lambda} = u_{s,\lambda}(r)$ changes sign at most twice.
Let $s_0 \in (0,1)$. Then the same result holds for every $s \in [s_0, 1)$ and $\lambda \in (0, \hat \lambda(s_0))$ where  $\hat \lambda(s_0)$ is the number given by Theorem \ref{3nodal}.
\end{teo}
\begin{proof}
The proof is the same as in \cite[Proposition 5.3]{FrLe2}, and is based on a known topological result (see \cite[Lemma D.1]{FrLe1}) and the Jordan's curve Theorem.
\end{proof}

Another crucial preliminary result is the following:

\begin{lemma}\label{lemmatecnico}
Let $s_0 \in (0,1)$ and let $\hat\lambda(s_0)>0$ be the number given by Theorem \ref{twicechange}. Let $s \in (s_0,1)$, $n>6s$, $R>0$ and let $u_{s,\lambda}$ be a least energy radial sign-changing solution of Problem \eqref{fracBrezis} in $B_R$, being such that $u_{s,\lambda}$ changes sign exactly twice and $u_{s,\lambda}\geq 0$ in a neighborhood of the origin. Let us denote by $0<r_{s}^1<r_s^2<R$ the nodes of $u_{s,\lambda}$. Let $W_{s, \lambda} = E_s u_{s,\lambda}$ be extension of $u_{s,\lambda}$. Then, for every $\overline \rho \in (r^2_s, R)$ such that $u_{s, \lambda}(\overline \rho)>0$, there exists $\overline \delta = \overline \delta(\overline \rho)>0$ such that 
\begin{equation}\label{asintotre1}
W_{s, \lambda} (x, y) \geq 0 \quad \forall |x| > \overline \rho, \quad \forall y \in (0, \overline \delta).
\end{equation}
\end{lemma}
\begin{proof}
Let $W_{s, \lambda} = E_s u_{s,\lambda}$ be the extension of $u_{s,\lambda}$. Thanks to Theorem \ref{3nodal} the function $W_{s, \lambda}$ has exactly two nodal regions,
\begin{equation}\label{nodaldom1}
\begin{aligned}
&\Omega ^+ := \{ (x, y) \in {\R^{n+1}_+} \ | \ W_{s, \lambda}(x, y) >0 \},\\
&\Omega ^- := \{ (x, y) \in {\R^{n+1}_+} \ | \ W_{s, \lambda}(x ,y) <0 \}.
\end{aligned}
\end{equation}
Moreover, since $W_{s,\lambda}$ is cylindically symmetric, we can define the sets
\begin{equation}\label{nodaldom2}
\begin{aligned}
P &:= \{(|x|, y) \in \{r\geq 0\} \times \{y > 0\} \ | \ (x, y) \in \Omega^+ \} \\
N &:= \{(|x|, y) \in \{r\geq 0\} \times \{y > 0\} \ | \ (x, y) \in \Omega^- \}. 
\end{aligned}
\end{equation}

Since we are assuming that $u_{s,\lambda} = u_{s,\lambda}(r)$ changes sign twice, there exist $\rho_1$, $\rho_2 >0$ such that 
\[
0 < \rho_1 < r^1_s<\rho_2 < r^2_s <\overline \rho <  R
\]
and $u_{s, \lambda}(\rho_1)>0$ while $u_{s,\lambda}(\rho_2)<0$.  
Thanks to the continuity of $W_{s, \lambda}$ we get that $(\rho_1, \varepsilon), (\overline \rho, \varepsilon) \in P$, for all sufficiently small $\varepsilon>0$. 
Fixing $\varepsilon>0$, since $P$ is arcwise connected, there exists a continuous curves $\gamma^\varepsilon_+ \in C^0([0,1]; \{r \geq 0\} \times \{y > 0\})$ such that 
\[
\begin{aligned}
&\gamma^\varepsilon_+(0) = (\rho_1, \varepsilon), \quad \gamma_+(1) = (\overline \rho, \varepsilon), \quad \gamma_+(t) \in P \ \forall t \in [0,1].
\end{aligned}
\]
Moreover, since $W_{s, \lambda}$ is continuous up to the boundary, and since $W_{s, \lambda}(\rho_1), W_{s,\lambda}(\overline \rho) \neq 0$, we can always modify $\gamma^\varepsilon_+$ in order to obtain a curve $\gamma_+ \in C^0([0, 1]; \{r \geq 0\} \times \{y\geq0\})$ such that it it injective and satisfies
\[
\begin{aligned}
&\gamma_+(0) = (\rho_1, 0), \quad \gamma_+(1) = (\overline \rho, 0), \quad \gamma_+(t) \in P \ \forall t \in (0,1).
\end{aligned}
\]
In addition, without loss of generality, we can assume that $\gamma_+([0,1]) \cap \{r=0\}=\emptyset$.
We notice that since $W_{s, \lambda}$ is continuous up to the boundary and  $\gamma_+([0,1])$ is a compact subset of $ \{r\geq0\} \times \{y \geq 0\}$ there exists $\overline\delta>0$ such that $ \text{dist}(\gamma^+([0,1]),N)>\overline\delta>0$.  

Now, by Jordan's curve theorem the closed and simple curve whose support is $\gamma^+([0,1]) \cup ([\rho_1, \overline \rho]\times \{0\})$ divides the set $\{r \geq 0\} \times \{y \geq 0\}$ in two regions, a bounded one which we call $A_b$, and unbounded one $A_u$. Since $u_{s, \lambda}(\rho_2) < 0$ and $\rho_2 \in (\rho_1, \overline \rho)$, by continuity and since $W_{s, \lambda}$ possesses exactly two nodal regions, we deduce that $N \cap A_b \neq \emptyset$. This, together with Jordan's curve theorem implies that $N \subset A_b$. 

Let $(r, y) \in [\overline \rho, +\infty) \times (0, \overline \delta)$, we claim that $W_{s, \lambda}(r, y) \geq 0$. Indeed suppose that there exits a point $(r_0, y_0) \in [\overline \rho, +\infty) \times (0, \overline \delta)$ such that $W_{s, \lambda}(r_0, y_0) <0$. This implies that $(r_0, y_0) \in N \subset A_b$. On the other hand, since $\gamma^+(t) \not \in \{r \geq 0\} \times \{0\}$ when $t \neq 0, 1$, we have that $(r_0, 0) \in A_u$, and thus, as a further consequence of the Jordan  curve theorem, $\gamma^+$ intersects any curve $\gamma_*$ connecting $(r_0, y_0)$ and $(r_0, 0)$, whose support $\gamma_*([0,1])$ intersects $\{y=0\}$ just in $(r_0, 0)$. In particular, choosing as $\gamma_*$ the segment joining $(r_0, y_0)$ and $(r_0, 0)$, there exists $t_0$ such that $\gamma^+(t_0)$  lies in the interior of that segment. But this implies that $\text{dist}(\gamma^+(t_0), (r_0, y_0)) < \overline \delta$, and by the definition of $\overline\delta$ we deduce that $(r_0, y_0)$ cannot belong to $N$, which gives a contradiction. The proof is complete. 
\end{proof}

\section{Uniform bounds with respect to $s$ and pre-compactness}

We begin this section by recalling a general result of approximation for the fractional Laplacian that will be useful in the sequel. For the statement to be meaningful, we remark that the space $H^s(\R^n)$ and the operator $(-\Delta)^s$ can be defined via the Fourier transform also for $s \geq 1$.

\begin{lemma}[{\cite[Lemma 2.4]{uniquenondeg}}]
Let $s, \sigma \in (0,1]$ and $\delta > 2|\sigma -s|$. Then, for any $\varphi \in H^{2(\sigma + \delta)}(\R^n)$, it holds that
\[
|(-\Delta)^\sigma \varphi - (-\Delta)^s \varphi|_2 \leq C |\sigma -s| \|\varphi\|_{2(\sigma+\delta)},
\] 
for some $C = C(\sigma, \delta) >0$. 
\end{lemma}

\begin{rem}\label{valdi}
Let $\varphi \in C^\infty_c(\R^n)$. Since $C^\infty_c(\R^n) \subset H^s (\R^n)$ for all $s \geq 0$ as a consequence of previous Lemma we obtain that for all $\sigma \in (0,1]$,
\[
|(-\Delta)^\sigma \varphi - (-\Delta)^s \varphi|_2  \to 0 \quad \text{when } s \to \sigma.
\]
\end{rem}

In the following lemma we refine the estimate stated in Remark \ref{valdinergia}.

\begin{lemma}\label{asSobcost}
Let $0 < s_0 < s_1 \leq 1$. Let $n>4s_1$ and $\lambda \in (0, \underline \lambda(s_0))$, where $\underline \lambda(s_0)$ is the number given by \eqref{eigenbounds}. Then, for every $s \in (s_0, s_1)$, it holds
\[
S_{s,\lambda} \leq S_s - q(\lambda)
\]
where $q(\lambda) = q(\lambda, s_0, s_1, n, \Omega) >0 $ for $\lambda \in (0, \underline \lambda(s_0))$ and $q(\lambda) \to 0$ as $\lambda \to 0^+$. 
\end{lemma} 

\begin{proof}
Let $u^s_\varepsilon$ be as in Proposition \ref{stimebubble}. For every $s \in (s_0, s_1)$, by definition of $S_{s, \lambda}$ and \eqref{stime}, for $\varepsilon <1$, we have that 
\[
\begin{aligned}
S_{s,\lambda} \leq & \ \frac{\|u^s_\varepsilon\|^2_s - \lambda|u^s_\varepsilon|^2_2}{|u^2_\varepsilon|^{2}_{2^*_s}} \leq \frac{S_s^{\frac{n}{2s}} + C_1\varepsilon^{n-2s} - \lambda C_2 \varepsilon^{2s}}{\left(S_s^{\frac{n}{2s}}-C_3\varepsilon^{n}\right)^{\frac{2}{2^*_s}}} \\
\leq&\ S_s + C_4\varepsilon^{n-2s}-\lambda C_5 \varepsilon^{2s} \leq S_s + C_5\varepsilon^{2s_1}(C_6\varepsilon^{n-4s_1} - \lambda),
\end{aligned}
\]
where the constants do not depend neither on $s$ nor on $\varepsilon$. 
Then taking a fixed $\varepsilon_0$ small enough so that $C_6\varepsilon_0^{n-4s_1} - \lambda < 0$, we obtain the desired result with $q(\lambda) = C_5\varepsilon_0^{2s_1}(\lambda- C_6\varepsilon_0^{n-4s_1})$.
\end{proof}

In view of the previous results we obtain a uniform $L^\infty$-bound for least energy positive solutions of Problem \eqref{fracBrezis}.

\begin{prop}\label{u0bound}
Let $0<s_0<s_1 \leq 1$ and $n>4s_1$. For every $s \in [s_0, s_1)$ and for any fixed $\lambda \in (0, \underline \lambda(s_0))$, where $\underline \lambda(s_0)$ is the number given by \eqref{eigenbounds}, let $u^0_{s,\lambda} \in \mathcal{N}_{s,\lambda}$ be such that $I_{s, \lambda}(u^0_{s, \lambda}) = c_\mathcal{N}(s, \lambda)$. It holds that 
\[
0 < \sup_{s \in [s_0,s_1)}|u^0_{s, \lambda}|_\infty< + \infty.
\]
\end{prop}
\begin{proof}
The first inequality is trivial. For the other inequality we argue by contradiction. Let us set $\delta_s := |u^0_{s, \lambda}|_\infty$ and assume that there exists $\sigma \in [s_0, s_1]$ and a sequence $(s_k) \subset (s_0, s_1)$ such that $\delta_{s_k} \to +\infty$ when $s_k \to \sigma$. From now on, in order to simplify the notation, we omit the subscript $k$.

Let consider the rescaled function
\[
v_{s, \lambda}(x) := \frac{1}{\delta_s}u^0_{s, \lambda}\left( \frac{x}{\delta_s^{\beta_s}}\right), \quad x \in \R^n,
\]
where $\beta_s := \frac{2}{n-2s}$. Notice that $v_{s, \lambda} \in X_0^s\left(B_{\delta_s^{\beta_s} R}\right)$.

A simple computation shows that
\begin{equation}\label{As:rescaling}
\begin{aligned}
&\|u^0_{s, \lambda}\|^2_s = \| v_{s,\lambda}\|^2_s;\\
&|u^0_{s,\lambda}|^{2^*_s}_{2^*_s} = |v_{s,\lambda}|^{2^*_s}_{2^*_s}.
\end{aligned}
\end{equation}

Since $u^0_{s, \lambda} \in \mathcal{N}_{s, \lambda}$ we have that 
\[
\begin{aligned}
c_\mathcal{N}(s, \lambda) =&\ \frac{1}{2}(\|u^0_{s,\lambda}\|_s^2 - \lambda|u^0_{s,\lambda}|^2_2) - \frac {1}{2^*_s}|u^0_{s,\lambda}|^{2^*_s}_{2^*_s} = \frac{s}{n} (\|u^0_{s,\lambda}\|_s^2 - \lambda|u^0_{s,\lambda}|^2_2)\\
\geq &\ \frac{s}{n}\left (1- \frac{\lambda}{\lambda_{1,s}}\right) \|u^0_{s,\lambda}\|^2_s \geq \frac{s_0}{n}\left (1- \frac{\lambda}{\underline \lambda(s_0)}\right) \|u^0_{s,\lambda}\|^2_s,
\end{aligned}
\] 
where $\underline \lambda(s_0)$ is as in \eqref{eigenbounds}.
Hence, thanks to Lemma \ref{eneras} and \eqref{As:rescaling}, together with the fact that $S_s$ is uniformly bounded with respect to $s \in (0,1)$, we get that there exists $\tilde C >0$ such that 
\begin{equation}\label{Hsbound}
0 < \sup_{s \in [s_0, s_1)} \|v_{s, \lambda}\|^2_s = \sup_{s \in [s_0, s_1)} \|u^0_{s, \lambda}\|^2_s \leq \tilde C. 
\end{equation}

An easy computation shows that $v_{s, \lambda}$ is a weak solution of
\begin{equation}\label{ministry}
(-\Delta)^s v_{s, \lambda} = \frac{\lambda}{\delta_s^{2s \beta_s}}v_{s, \lambda} + |v_{s, \lambda}|^{2^*_s-2}v_{s, \lambda} \qquad \text{in }B_{\delta_s^{\beta_s} R}.
\end{equation}

As a consequence of that and since $|v_{s,\lambda}|_\infty = 1$, thanks to Remark \ref{gilbcont}, there exists $v_\lambda$ such that $v_{s, \lambda} \to v_\lambda$ in $C^{0, \alpha}_{loc}(\R^n)$ for any fixed $\alpha < s_0$ as $s \to \sigma$. Moreover, the convergence on compact subsets of $\R^n$ implies that $v_\lambda \not \equiv 0$. Indeed, recall that, as seen in Proposition \ref{posSol}, $u^0_{s,\lambda}$ is radial and achieves its maximum at the origin, hence $v_{s,\lambda}(0) = 1$ and thus $v_\lambda(0)=1$. 

Coming back to the original sequence $u^0_{s, \lambda}$, thanks to Lemma \ref{asSobcost} and being $ u^0_{s, \lambda} \in \mathcal{N}_{s, \lambda}$, we have
\[
\frac{s}{n}|u^0_{s, \lambda}|_{2^*_s}^{2^*_s} = I_{s, \lambda}(u^0_{s, \lambda}) = c_{\mathcal{N}}(s, \lambda) \leq \frac{s}{n} S_s^{\frac{n}{2s}}- q(\lambda).
\]
Therefore, by Fatou's Lemma and \eqref{As:rescaling} we obtain

\begin{equation}\label{upextlinfproof}
|v_\lambda|_{2^*_{\sigma}}^{2^*_{\sigma}} \leq \liminf_{s \to \sigma}|v_{s, \lambda}|_{2^*_{s}}^{2^*_{s}} = \liminf_{s \to \sigma}|u^0_{s, \lambda}|_{2^*_{s}}^{2^*_{s}} \leq S_{\sigma}^{\frac{n}{2\sigma}} - \frac{n}{\sigma}q(\lambda).
\end{equation}

To reach a contradiction we need to obtain also a lower bound for the energy $|v_\lambda|_{2^*_{\sigma}}^{2^*_{\sigma}}$. 
To this end, let us fix $\varphi \in C^\infty_c(\R^n)$. We claim that, as $s \to \sigma$,
\begin{equation}\label{equazionelimite}
\begin{aligned}
(v_{s,\lambda}, \varphi)_s &= \int_{\R^n} v_{s, \lambda} (-\Delta)^s \varphi \de x \\
&= \int_{\R^n} v_{s, \lambda} ((-\Delta)^s \varphi- (-\Delta)^{\sigma} \varphi) \de x + \int_{\R^n} (v_{s, \lambda}- v_\lambda) (-\Delta)^{\sigma} \varphi \de x \\
&+\int_{\R^n} v_{\lambda} (-\Delta)^{\sigma} \varphi \de x = \int_{\R^n}  v_{\lambda} (-\Delta)^{\sigma} \varphi \de x + o(1).
\end{aligned}
\end{equation}

First of all, we point out that since $v_{s, \lambda} \in X^s_0(B_{\delta_s^{\beta_s}R}) \subset \mathcal{D}^s(\R^n)$, the first equality follows from \eqref{veryweak}. Moreover, as a consequence of Remark \ref{valdi}, we have that $(-\Delta)^s \varphi - (-\Delta)^\sigma \varphi \to 0$ a.e. in $\R^n$ as $s \to \sigma$. Furthermore, thanks to \eqref{bogdy} we have that, since $s \in [s_0, 1)$ and $\sigma \in [s_0,1]$, 
\[
\begin{aligned}
\left| (-\Delta)^s \varphi - (-\Delta)^\sigma \varphi  \right| \leq \frac{C}{(1 + |x|)^{n+2s}} +&  \frac{C}{(1 + |x|)^{n+2\sigma}} \leq 2C \frac{1}{(1 + |x|)^{n+2s_0}} \in L^1(\R^n),
\end{aligned}
\]
where $C>0$ depends on $n$ and $\varphi$ but not on $s$. 
Applying Lebesgue's dominated convergence theorem we get that
\[
\left|\int_{\R^n} v_{s, \lambda} ((-\Delta)^s \varphi- (-\Delta)^{\sigma} \varphi) \de x\right| \leq \int_{\R^n} \left|(-\Delta)^s \varphi- (-\Delta)^{\sigma} \varphi\right| \de x \to 0. 
\]
In a similar way, considering that $v_{s, \lambda} \to v_\lambda$ a.e., we prove that
\[
\left|\int_{\R^n} (v_{s, \lambda}- v_\lambda) (-\Delta)^{\sigma} \varphi \de x \right| \to 0
\]
and the claim is proved. In view of \eqref{ministry} and \eqref{equazionelimite} we obtain the relation
\begin{equation}\label{lbound:equ}
\int_{\R^n}  v_\lambda (-\Delta)^{\sigma} \varphi \de x = \int_{\R^n} |v_{\lambda}|^{2^*_{\sigma}-2}v_{\lambda}\varphi \de x \quad \forall \varphi \in C^\infty_c(\R^n).
\end{equation}

Now we have to consider two different cases: when $\sigma <1$, we easily deduce that $v_\lambda \in \mathcal{D}^{\sigma}(\R^n)$, as a straightforward consequence of Fatou's Lemma (recall that $C_{n,s}$ is continuous on $s \in [0,1]$) and \eqref{Hsbound}. Indeed
\[
\begin{aligned}
\|v_\lambda\|_\sigma^2 =&\ \frac{C_{n,\sigma}}{2}\int_{\R^{2n}}\frac{|v_{\lambda}(x)-v_{\lambda}(y)|^2}{|x-y|^{n+2\sigma}}\de x \de y  \\
\leq &\ \liminf_{s \to \sigma} \frac{C_{n,s}}{2}\int_{\R^{2n}}\frac{|v_{s,\lambda}(x)-v_{s,\lambda}(y)|^2}{|x-y|^{n+2s}}\de x   = \liminf_{s \to \sigma} \|v_{s, \lambda}\|_s^2 \leq C,
\end{aligned}
\]
so that $v_\lambda \in \mathcal{D}^\sigma(\R^n)$. Then we can apply \eqref{veryweak} again, and by density, we obtain that $v_\lambda$ weakly satisfies the equation
\[
(-\Delta)^{\sigma} v_\lambda = |v_\lambda|^{2^*_{\sigma}-2}v_\lambda \quad \text{in}\ \R^n.
\]

Therefore, using $v_\lambda \in \mathcal{D}^\sigma(\R^n)$ as a test function and since $v_\lambda \not \equiv 0$, we obtain
\[
S_\sigma \leq \frac{\|v_\lambda\|_\sigma^2}{|v_\lambda|_{2^*_{\sigma}}^2} = |v_\lambda|_{2^*_\sigma}^{2^*_\sigma-2},
\]
i.e., $S_\sigma^{\frac{n}{2\sigma}} \leq |v_\lambda|_{2^*_\sigma}^{2^*_\sigma}$, which, in view of \eqref{upextlinfproof}, readily gives a contradiction. 

When $\sigma =1$ the argument via Fatou's Lemma fails since $C_{n,s} \to 0$ as $s \to 1^-$, and a more careful approach is needed.  
First of all notice that since in this case $s \to 1^-$, then, passing if necessary to a subsequence, we can assume that $\frac{2}{3}<  s_0  < s <1$. Thus, by Remark \ref{gilbcont}, we get that $v_{s, \lambda}\to v_\lambda$ in $C^{2,\gamma}_{loc}(\R^n)$ for $\gamma < 3 s_0-2$. In particular $v_\lambda \in C^2(\R^n)$. This allows us to integrate by parts in \eqref{lbound:equ} so that we obtain that $v_\lambda$ weakly satisfies the equation
\[
-\Delta v_\lambda = |v_\lambda|^{2^*_{1}-2}v_\lambda \quad \text{in }\ \R^n.
\]
Since $v_\lambda \in C^2(\R^n)$, this actually implies that $v_{\lambda}$ satisfies $-\Delta v_{\lambda} = |v_{\lambda}|^{2^*_1-2}$ in the classical sense. 
Moreover, by \eqref{upextlinfproof} we have that $v_\lambda \in L^{2^*_1}(\R^n)$. Therefore we can apply \cite[Theorem 2, Corollary 3]{farina}, obtaining that $v_\lambda \in \mathcal{D}^1(\R^n)$ and $\|v_\lambda\|^2_1 = |v_\lambda|^{2^*_1}_{2^*_1}$. Hence, also in this case, we recover the estimate $S_1^{\frac{n}{2}} \leq |v_\lambda|_{2^*_1}^{2^*_1}$ and as before we get a contradiction.
\end{proof}

Thanks to Proposition \ref{u0bound} we can improve the inequality obtained in Lemma \ref{EnergyBound}. More precisely, the following result holds. 
\begin{cor}\label{ImprEnerBound}
Let $0 < s_0< s_1 \leq 1$, $n > 6s_1$ and $\lambda \in (0, \underline \lambda(s_0))$, where $\underline \lambda(s_0)$ is given by \eqref{eigenbounds}. Then there exists $Q=Q(s_0,s_1,\lambda) >0$ such that for every $s \in [s_0, s_1)$ we have
\[
c_\mathcal{M}(s, \lambda) \leq c_\mathcal{N}(s, \lambda) + \frac{s}{n}S_s^{\frac{n}{2s}} - Q.
\]
\end{cor}
\begin{proof}
At the end of the proof of Lemma \ref{EnergyBound} (see \eqref{technicality2.0}, \eqref{anothertechnicality}) we have proved that for every $s \in (0,1)$
\[
c_\mathcal{M}(s, \lambda) \leq  c_\mathcal{N}(s, \lambda) + \frac{s}{n}S_s^{\frac{n}{2s}}+ C_1(s, \lambda)|u^0_{s, \lambda}|_{L^\infty(B_\rho(x_0))}\varepsilon^{\frac{n-2s}{2}}- \lambda C_2(s)\varepsilon^{2s},
\]
for any sufficiently small (depending on $s$) $\varepsilon>0$, and for fixed $x_0 \in \Omega$, $\rho>0$.
Now, by the previous proposition we have that $\sup_{s \in [s_0,s_1)}|u^0_{s, \lambda}|_\infty< C$, and, by a careful inspection of the proof of Lemma \ref{EnergyBound} we get that the $C_1$, $C_2$ are uniformly bounded with respect to $s \in [s_0,s_1)$ and $C_2$ is far from zero. If we set $\overline{C}_1(\lambda):=\sup_{s \in [s_0,s_1)} C_1(s,\lambda)$, $\underline{C}_2:=\inf_{s \in [s_0,s_1)} C_2(s)>0$, then, $\varepsilon$ can be taken sufficiently small (depending only on $\lambda$, $s_0$, $s_1$) in such the way that
\[
c_\mathcal{M}(s, \lambda) \leq  c_\mathcal{N}(s, \lambda) + \frac{s}{n}S_s^{\frac{n}{2s}}+ \overline{C}_1(\lambda)\varepsilon^{\frac{n-2s_1}{2}}- \lambda \underline{C}_2\varepsilon^{2s_1}, \ \forall s \in [s_0,s_1)
\]
and $Q(s_0,s_1,\lambda):= \lambda \underline{C}_2\varepsilon^{2s_1} - \overline{C}_1(\lambda)\varepsilon^{\frac{n-2s_1}{2}}>0$ (where we have used that $n>6s_1$).
The proof is then complete.
\end{proof}

We can now prove a $L^\infty$-bound for the sequence of radial sign-changing solutions. 

\begin{lemma}\label{Linftyboundsc}
Let $\frac{1}{2}<s_0<s_1 \leq 1$. Let $n > 6s_1$, $R>0$ and let us fix $\lambda \in \left(0, \underline \lambda(s_0)\right)$, where $\underline\lambda(s_0)$ is the number given by \eqref{eigenbounds}.
For every $s \in (s_0,s_1)$ let $u_{s,\lambda} \in \mathcal{M}_{s,\lambda; rad}$ be a radial solution of Problem \eqref{fracBrezis} in $B_R$ such that $I_{s, \lambda}(u_{s, \lambda}) = c_{\mathcal{M}_{rad}}(s, \lambda)$. It holds that 
\[
0 < \sup_{s \in [s_0,s_1)}|u_{s, \lambda}|_\infty< + \infty.
\]
\end{lemma}
\begin{proof}
The first inequality is trivial. For the other inequality, we argue by contradiction as in the proof of Proposition \ref{u0bound}. Let us set $\delta_s := |u_{s, \lambda}|_\infty$ and suppose that there exists $\sigma \in [s_0, s_1]$ and a sequence $s \to \sigma$ such that $\delta_s \to +\infty$. Consider the rescaled function
\[
v_{s, \lambda}(x) := \frac{1}{\delta_s}u_{s, \lambda}\left( \frac{x}{\delta_s^{\beta_s}}\right), \quad x \in \R^n, 
\]
where $\beta_s = \frac{2}{n-2s}$. Clearly $v_s \in X_0^s\left(B_{\delta_s^{\beta_s} R}\right)$.
Arguing as in the proof of formula \eqref{Hsbound}, we get that there exists $C>0$ such that  
\begin{equation}\label{energybound17}
0 < \sup_{s \in [s_0, s_1)} \|v_{s, \lambda}\|^2_s = \sup_{s \in [s_0, s_1)} \|u_{s, \lambda}\|^2_s \leq C. 
\end{equation}

An easy computation shows that $v_{s, \lambda}$ weakly satisfies the equation
\[
(-\Delta)^s v_{s, \lambda} = \frac{\lambda}{\delta_s^{2s\beta_s}}v_{s, \lambda} + |v_{s, \lambda}|^{2^*_s-2}v_{s, \lambda} \qquad \text{in }B_{\delta_s^\beta R}.
\]

As a consequence of that, since $|v_{s_\lambda}|_\infty = 1$, by Remark \ref{gilbcont} it follows that as $s \to \sigma$ we have that $v_{s, \lambda} \to v_\lambda$ in $C^{0, \alpha}_{loc}(\R^n)$ for any fixed $\alpha < s_0$. 
Moreover, let us observe that thanks to Proposition \ref{strauss} and \eqref{energybound17}, for any $x_s$ such that $|u(x_s)| = \delta_s$ it holds 
\[
(\delta_s^{\beta_s} |x_s|)^{\frac{n-2s}{2}} = |x_s|^{\frac{n-2s}{2}} \leq |u_{s,\lambda}(x_s)| \leq K_{n,s}\|u_{s, \lambda}\|^2_s \leq \hat C,  
\] 
where $\hat C$ depends only on $s_0$. 
This implies that there exists a compact set $K \subset \subset \R^n$ such that $\delta^{\beta_s}_s x_s \in K$ for all $s$ sufficiently close to $\sigma$, and then by the $C^{0,\alpha}_{loc}$-convergence there exists $\hat x \in K$ such that $v_\lambda(\hat x) = 1$, and thus $v_\lambda \not \equiv 0$. 

As in proof of Proposition \ref{u0bound}, we obtain that $v_\lambda$ is a weak solution of the equation
\begin{equation}\label{tecn:sobolev}
(-\Delta)^{\sigma} v_\lambda = |v_\lambda|^{2^*_{\sigma}-2}v_\lambda \qquad \text{in }\ \R^n,
\end{equation}
and using Corollary \ref{ImprEnerBound} we get that
\[
|v_\lambda|_{2^*_{\sigma}}^{2^*_{\sigma}} \leq \liminf_{s \to \sigma}S_{s, \lambda}^{\frac{n}{2s}} +S_{\sigma}^{\frac{n}{2\sigma}} - \frac{n}{\sigma} Q(\lambda)< {2} S_{\sigma}^{\frac{n}{2\sigma}}.
\]

On the other hand, for every $\sigma \in (0,1]$ if $u$ is a non trivial sign-changing solution of \eqref{tecn:sobolev} then $|u|^{2^*_\sigma}_{2^*_\sigma} > 2S_\sigma$. This is known when $\sigma =1$. When $\sigma < 1$, by definition of the Sobolev constant and testing \eqref{tecn:sobolev} with $u^\pm$ we obtain
\begin{equation}\label{nodalenerbound}
S_\sigma \leq \frac{\|u^\pm \|^2_\sigma}{|u^\pm|^2_{2^*_\sigma}} \leq |u^\pm|^{2^*_\sigma-2}_{2^*_\sigma} - \frac{\eta_\sigma(u)}{|u^\pm|^{2}_{2^*_\sigma}} < |u^\pm|^{2^*_\sigma-2}_{2^*_\sigma}.
\end{equation}
Therefore since $v_\lambda \not \equiv 0$, the only possibility is that $v_\lambda$ is of constant sign. 
Assume for instance that $v_\lambda \geq 0$. Then $v_{s, \lambda}^+ \to v_\lambda$ a.e. and by Fatou's Lemma we get that

\begin{equation}\label{fatoulemma12}
|v_\lambda|^{2^*_\sigma}_{2^*_s} \leq \liminf_{s \to \sigma}|v_{s, \lambda}^+|^{2^*_s}_{2^*_s}.
\end{equation}

Since $u_{s, \lambda} \in \mathcal{M}_{s, \lambda}$ and by definition of $S_{s, \lambda}$ we have
 
\begin{equation}\label{fatoulemma22}
S_{s, \lambda} \leq \frac{\|u^-_{s,\lambda}\|_s^2 - \lambda |u^-_{s, \lambda}|^2_2}{|u^-_{s, \lambda}|_{2^*_s}^2} = |u^-_{s,\lambda}|^{2^*_s-2}_{2^*_s} - \frac{\eta_s (u_{s,\lambda})}{|u^-_{s, \lambda}|_{2^*_s}^2} < |u^-_{s,\lambda}|^{2^*_s-2}_{2^*_s} = |v^-_{s, \lambda}|^{2^*_s-2}_{2^*_s}.
\end{equation}

which together with \eqref{fatoulemma12} implies

\begin{equation}\label{fatoulemma15}
|v_\lambda|^{2^*_\sigma}_{2^*_\sigma} + \liminf_{s \to \sigma}S_{s, \lambda}^{\frac{n}{2s}}  \leq \liminf_{s \to \sigma}\left( |v_{s,\lambda}^+|^{2^*_s}_{2^*_s} + |v_{s,\lambda}^-|^{2^*_s}_{2^*_s}\right) = \liminf_{s \to \sigma}|v_{s,\lambda}|^{2^*_s}_{2^*_s}.
\end{equation}
On the other hand by Corollary \ref{ImprEnerBound} we have that
\[
|v_{s, \lambda}|^{2^*_s}_{2^*_s} = |u_{s,\lambda}|^{2^*_s}_{2^*_s}\leq S_{s, \lambda}^{\frac{n}{2s}} +S_{s}^{\frac{n}{2\sigma}} - Q(\lambda),
\]
and recalling that $Q$ does not depend on $s$, and $S_s$ is continuous with respect to $s$, we deduce that
\begin{equation}\label{fatoulemma16}
\liminf_{s \to \sigma}|v_{s, \lambda}|^{2^*_s}_{2^*_s} < S_{\sigma}^{\frac{n}{2\sigma}}+ \liminf_{s \to \sigma}S_{s,\lambda}^{\frac{n}{2s}}.
\end{equation}

From \eqref{fatoulemma15} and \eqref{fatoulemma16} we obtain that $|v_\lambda|^{2^*_\sigma}_{2^*_\sigma} < S_{\sigma}^{\frac{n}{2\sigma}}$, and this is a contradiction because every non trivial solution $u$ of \eqref{tecn:sobolev} must satisfy $|u|^{2^*_\sigma}_{2^*_\sigma} \geq S^\frac{n}{2\sigma}_\sigma $.
\end{proof}

Thanks to this uniform $L^\infty$-bound on sign-changing solutions of Problem \eqref{fracBrezis}, we have the following result.

\begin{teo}\label{sconverg}
Let $\frac{1}{2}<s_0<s_1 \leq 1$. Let $n > 6s_1$, $R>0$ and let $\hat \lambda(s_0) $ be the number given by Theorem \ref{twicechange}. For any fixed $\lambda \in (0, \hat \lambda(s_0))$, let $(u_{s,\lambda})_s$ be a family, $s \in [s_0, s_1)$, of radial sign-changing solutions of Problem \eqref{fracBrezis} with $I_{s, \lambda}(u_{s, \lambda}) = c_{\mathcal{M}_{rad}}(s, \lambda)$. 
Assume that $s \to \sigma$, for some $\sigma \in [s_0, s_1]$. Then, for any fixed $\alpha \in (0, s_0)$, we have that $u_{s, \lambda} \to u_{\sigma, \lambda}$ in $C^{0,\alpha}_{loc}(B_R)$ up to a subsequences, as $s \to \sigma$. Moreover $u_{\sigma, \lambda} \in X^\sigma_0(B_R)$ and is a weak non trivial solution of
\[
\begin{cases}
(-\Delta)^\sigma u_{\sigma,\lambda} = \lambda u_{\sigma, \lambda} + |u_{\sigma, \lambda}|^{2^*_\sigma-2}u_{\sigma, \lambda} & \text{in }B_R\\
u_{\sigma, \lambda} = 0 & \text{in }\R^n \setminus B_R
\end{cases}
\]
In addition
\[
\lim_{s \to \sigma} I_{s, \lambda}(u_{s, \lambda}) = I_{\sigma, \lambda}(u_{\sigma, \lambda}).
\]
\end{teo}
\begin{proof}
Let us fix $\frac{1}{2}<s_0<s_1 \leq 1$ and $\lambda \in (0, \hat \lambda(s_0))$. Let $(u_{s, \lambda})_s$ be a family of least energy radial sign-changing solutions of Problem \eqref{fracBrezis}, where $s \in [s_0, s_1)$. 
By Corollary \ref{ImprEnerBound} we have that $(u_{s, \lambda})$ is a bounded family in $X^{s_0}_0(B_R)$, and thus up to a subsequence, there exists $u_{\sigma, \lambda} \in X_0^{s_0}(B_R)$ such that:
\[
\begin{aligned}
&u_{s, \lambda} \rightharpoonup u_{\sigma, \lambda} & &\text{in } X^{s_0}_0(B_R)\\
&u_{s, \lambda} \to u_{\sigma, \lambda} & &\text{in } L^{p}(B_R), \forall p \in (1, 2^*_{s_0}),
&u_{s, \lambda} \to u_{\sigma, \lambda} & &\text{a.e. in } \R^n.
\end{aligned}
\] 

On the other hand, thanks Lemma \ref{Linftyboundsc}, we can argue as in Remark \ref{gilbcont}, and obtain that, up to a subsequence 
\[
\begin{aligned}
&u_{s, \lambda} \to u_{\sigma, \lambda} & &\text{in } C^{0, \alpha}_{loc}(B_R),\\
\end{aligned}
\]
for every fixed $0<\alpha < s_0$. 

Exploiting the uniform $L^\infty$-bound given by Lemma \ref{Linftyboundsc} and since $u_{s, \lambda} \equiv 0$ in $\R^n \setminus B_R$, we obtain that
\begin{equation}\label{sconverg1}
\begin{aligned}
&u_{s, \lambda} \to u_{\sigma, \lambda} \text{ in } L^{p}(\R^n), \forall p >1 \\
&|u_{s, \lambda}|_{2^*_s}^{2^*_s} \to |u_{\sigma,\lambda}|_{2^*_\sigma}^{2^*_\sigma}
\end{aligned}
\end{equation}

Arguing as in the proof of Proposition \ref{u0bound} we obtain that $u_{\sigma, \lambda}\in \mathcal{D}^\sigma(\R^n)$, both when $\sigma <1$ or $\sigma =1$, and then, since $u_{\sigma, \lambda} \in L^{2}(\R^n)$, $u_{\sigma, \lambda} \equiv 0$ in $\R^n \setminus B_R$ (because $u_{\sigma, \lambda} \in X^{s_0}_0(B_R)$) we conclude that $u_{\sigma,\lambda} \in X^\sigma_0(B_R)$. Alternatively, a simpler way of proving that $u_{\sigma,\lambda} \in X^\sigma_0(B_R)$ is to use the Fourier transform definition of $\|\cdot\|_s$, unifying both cases. In fact, since $u_{s, \lambda}\to u_{\sigma,\lambda}$ in $L^2(\R^n)$, using the characterization via the Fourier transform of the Sobolev spaces and Fatou's Lemma we get that $u \in H^\sigma(\R^n)$ and, as seen before, we have $u_{\sigma, \lambda} \equiv 0$ in $\R^n \setminus B_R$. Therefore $u_{\sigma, \lambda} \in X^{\sigma}_0(B_R)$ and,  as in the proof of Proposition \ref{u0bound}, we get that $u_{\sigma, \lambda}$ weakly satisfies
\begin{equation}\label{technicaleq27}
\begin{cases}
(-\Delta)^\sigma u_{\sigma,\lambda} = \lambda u_{\sigma, \lambda} + |u_{\sigma,\lambda}|^{2^*_\sigma -2}u_{\sigma, \lambda} & \hbox{in}\ B_R \\
u_{\sigma, \lambda} = 0 & \hbox{in}\  \R^n \setminus B_R.
\end{cases}
\end{equation}

Thanks to our choice of $s_0$ it holds that
\[
\frac{s}{n}|u_{s,\lambda}|^{2^*_s}_{2^*_s} = c_{\mathcal{M}_{rad}}(s, \lambda) \geq 2c_\mathcal{N}(s, \lambda) \geq 2\frac{s}{n}\left(1-\frac{\lambda}{\lambda_{1,s}}\right) S_s \geq C,
\]
for some $C >0$ not depending on $s$, and thus, using \eqref{sconverg1} we obtain that $|u_{\sigma,\lambda}|^{2^*_\sigma}_{2^*_\sigma} \geq C$. This implies that $u_{\sigma, \lambda}$ is not trivial and the first part of the proof is complete. 

For the second part, using $u_{\sigma,\lambda}$ as test function in the equation \eqref{technicaleq27}, and using \eqref{sconverg1} we get that 
\[
\|u_{s, \lambda}\|^2_s = \lambda|u_{s, \lambda}|_2^2 + |u_{s, \lambda}|_{2^*_s}^{2^*_s} \to \lambda|u_{\sigma, \lambda}|_2^2 + |u_{\sigma, \lambda}|_{2^*_\sigma}^{2^*_\sigma} = \|u_{\sigma, \lambda}\|^2_\sigma
\] 
which readily implies that
\[
I_{s, \lambda}(u_{s, \lambda}) \to I_{\sigma, \lambda}(u_{\sigma, \lambda}).
\]
The proof is complete.
\end{proof}

\section{Proof of Theorem \ref{mainteoproba}}
Theorem \ref{mainteoproba} is a consequence of the following result.

\begin{prop}\label{nozeroorigin}
Let $s \in \left(\frac{1}{2}, 1\right)$, $n> 6s$. There exist $\overline \lambda \in (0, \lambda_{1,s}]$ and $C>0$ such that for every $\lambda \in (0, \overline \lambda)$ if $u_{s, \lambda}\subset \mathcal{M}_{s, \lambda; rad}$ is a least energy radial sign-changing solution of Problem \eqref{fracBrezis}, then $|u_{s, \lambda}(0)| >C$.
\end{prop}
\begin{proof}
Let $s \in \left(\frac{1}{2}, 1\right)$, $n> 6s$. Assume that the thesis is false. Then, there exist two sequences $\lambda_k \to 0^+$, $C_k \to 0^+$ and a sequence $u_k := u_{s, \lambda_k}\subset \mathcal{M}_{s, \lambda_k; rad}$ of solutions such that $0 \leq |u_k(0)| \leq C_k$, so that $u_k(0) \to 0$ as $k \to +\infty$. 
Let us set $M_k := |u_k|_\infty$, then, by Lemma \ref{As:energas} $(vi)$, up to a further subsequence, we have $M_k \to \infty$ as $k \to \infty$.

Now consider the rescaled functions
\[
v_k(x) := \frac{1}{M_k} u_k\left( \frac{x}{M_k^{\beta}}\right), \quad x \in \R^n,
\] 
where $\beta = \frac{2}{n-2s}$. By construction we observe that $v_k(0) \to 0$ as $k \to +\infty$. Moreover, if $\tilde x_k  \in M_k^{\beta_k}B_R$ is such that $|v_k(\tilde x_k)| = 1$, then by Proposition \ref{strauss} we obtain that $\tilde x_k$ stays in a compact subset of $\R^n$. Arguing as in Proposition \ref{u0bound} we obtain that there exists $v \in \mathcal{D}^s(\R^n)$ such that $v_k \rightharpoonup v$ in $\mathcal{D}^s(\R^n)$, where $v$ weakly solves
\begin{equation}\label{nozeroorigin17}
(-\Delta)^s v = |v|^{2^*_s-2}v \quad \text{in }\R^n.
\end{equation}

As we have seen in \eqref{nodalenerbound}, if $v$ is a sign-changing solution of \eqref{nozeroorigin17} it must satisfy $|v|^{2^*_s}_{2^*_s} > 2 S_s^{\frac{n}{2s}}$. On the other hand by Fatou's Lemma and Lemma \ref{As:energas} we get that 
\[
|v|^{2^*_s}_{2^*_s} \leq \liminf_{k \to \infty}|v_k|^{2^*_s}_{2^*_s} =  2 S_{s}^{\frac{n}{2s}},
\]
and then the only possibilities are that $v$ is trivial or of constant sign. 

By a standard argument (as seen in Remark \ref{gilbcont}, but here $s$ is fixed),  since $|v_k|_\infty \leq 1$, up to a subsequence, we get that $v_k \to v$ in $C^{0,\alpha}_{loc}(\R^n)$ for some $\alpha < s$. In particular, since we have seen that $ \tilde x_k$ stays in a compact subset of $\R^n$, then, up to a subsequence, setting $\tilde x = \lim_{k \to \infty} \tilde x_k$ it follows that $|v(\tilde x)| = 1$. Hence $v$ is not trivial. Moreover we observe that by construction it holds that $v(0)=0$. 

Therefore $v$ is of constant sign, and without loss of generality we can assume that $v \geq 0$. Since $v$ solves \eqref{nozeroorigin17}, by the strong maximum principle, as stated in \cite[Corollary 4.2]{Musina}, we deduce that $v>0$ in $\R^n$ which gives a contradiction since $v(0) = 0$.

Alternatively, one can argue as follows: since $v \geq0$, we get that that $v_k^+ \to v$ a.e. and then, by Fatou's Lemma and Lemma \ref{As:energas} $(ii)$, and recalling that the norms of $v_k$ and $u_k$ are related as in \eqref{As:rescaling}, we can infer that 
\[
S_s = \frac{\|v\|^2_s}{|v|^2_{2^*_s}},
\]
i.e. $v$ achieves the infimum in the fractional Sobolev inequality. Hence $v$ is in the form \eqref{eq:bubble} so that $v>0$, which once again contradicts the fact that $v(0) = 0$. The proof is complete.
\end{proof}
\end{section}

\section{Proof of Theorem \ref{mainteocc}}
\begin{proof}[Proof of Theorem \ref{mainteocc}]
Let $u_{s, \lambda}$ be a least energy radial sign-changing solution of Problem \eqref{fracBrezis}. The existence of a number $\hat \lambda_s>0$ satisfying the first part of the theorem has been proved in Theorem \ref{twicechange}. Therefore for $0<\lambda < \hat \lambda_s$ we have that $u_{s, \lambda}$ changes sign either once or twice. For the second part of the theorem we begin with proving the following preliminary fact:\\[6pt]
\textbf{Claim:} if there exists $r', r''>0$ such that $u_{s, \lambda}(r') \cdot u_{s, \lambda}(r'')> 0$ and there is no change of sign between $r'$ and $r''$, then $u_{s, \lambda}(r) \neq 0$ for all $r \in [r', r'']$.\\[2pt]  

Indeed, assume without loss of generality that $u_{s, \lambda}(r'), u_{s, \lambda}(r'') > 0$ and $u_{s, \lambda}\geq 0$ in $(r', r'')$.  Let  $W_{s, \lambda}$ be the extension of $u_{s,\lambda}$, let $\Omega ^+$, $\Omega ^-$, $P$ and $N$ as in \eqref{nodaldom1}, \eqref{nodaldom2}. In addition let us recall that under our assumptions, $W_{s, \lambda}$ possesses exactly two nodal domains.

Arguing as in Lemma \ref{lemmatecnico} we get that there exists a Jordan curve $\gamma^+: [0,1] \to \{r\geq 0\} \times \{y \geq 0\}$ which connects $(r', 0)$ and $(r'', 0)$, such that $\gamma^+(t) \in P $ for all $t \in (0, 1)$ and without loss of generality we can assume that $\gamma_+([0,1]) \cap \{r=0\}=\emptyset$.

Then, by Jordan's curve theorem, the curve whose support is $\gamma^+([0,1]) \cup ([r', r'']\times \{0\})$ divides $\{r\geq 0\} \times \{y \geq 0\}$ in two connected regions: a bounded one which we call $A_b$ and a unbounded one $A_u$.

Assume by contradiction that there exists $r_0 \in (r', r'')$ such that $u_{s, \lambda}(r_0) = 0$, and let $B^+_\delta(r_0): = \{(r, y) \in \{r\geq0\} \times \{y > 0\}\ |  \ |(r, y)- (r_0, 0)| < \delta\}$. We claim that $B^+_\delta(r_0) \cap N \neq \emptyset$ for every $\delta >0$. Indeed, assume that this is not the case. Then there exists $\delta>0$ such that $B_\delta^+(r_0) \cap N = \emptyset$, and thus for every $(r, y) \in \overline B_\delta^+(r_0)$ it holds that $W_{s, \lambda}(r, y) \geq 0$. As a consequence of the strong maximum principle (see Proposition \ref{yannick}) we conclude that $u_{s,\lambda}(r)>0$ for every $|r-r_0| < \delta$, and in particular $u_{s, \lambda}(r_0) >0$, which contradicts the assumption on $r_0$.

Therefore, for every $\delta >0$ it holds that $B^+_\delta(r_0) \cap N \neq \emptyset$. On one hand, since $\gamma^+(t) \in P$ for all $t \in (0,1)$ and $r_0 \in (r', r'')$, there exists $\delta$ small enough such that $B^+_\delta(r_0) \subset A_b$. This implies that there exists a point $(r'_-, y'_-) \in A_b \cap N$. 
On the other hand, since $u_{s, \lambda}$ changes sign and $W_{s, \lambda} \geq 0$ in $[r', r''] \times \{0\}$, there exists $r''_- \in \{r \geq 0\} \setminus [r', r'']$ such that $u_{s, \lambda}(r''_-) <0$. Using once again the continuity of $W_{s, \lambda}$ and that $\gamma^+(t) \in P$ for all $t \in (0, 1)$, we get there exists $y''_-$ such that $(r''_-, y''_-) \in N \cap A_u$.

Therefore, being $W_{s, \lambda}$ continuous and $N$ connected, there exists a continuous path joining $(r'_-, y'_-)$ and $(r''_-, y''_-)$ which lies completely in $N$. As a consequence of the Jordan curve theorem, such path must intersect $\gamma^+$, so that $P \cap N$ is not empty, which is absurd. The claim is then proved\\

Now let us prove $(a)$. Assume that $u_{s,\lambda}$ changes sign exactly twice, and denote by $0<r_1<r_2<R$ its nodes. In order to prove the result we must show that $u_{s,\lambda}$ cannot vanish in any other point $r \in [0,R)$ different from the nodes. To this end we argue by contradiction. Assume that there exists $r_0 \in [0,R)$, $r_0\neq r_1,r_2$ such that $u_{s,\lambda}(r_0)=0$. Then, there are only three possibilities: $r_0 \in [0,r_1)$, $r_0 \in (r_1,r_2)$ or $r_0 \in (r_2,R)$. Let us show that $r_0=0$ cannot happen.

Indeed, assume by contradiction that $u_{s, \lambda} (0) = 0$. Then there exist $r'$, $r''$ such that $0 < r' <r_1 < r_2 < r'' <R$ and satisfying $u_{s, \lambda}(r') \cdot u_{s, \lambda}(r'') >0$. Without loss of generality we assume that $u_{s, \lambda}(r')>0$. This implies that there exists a simple continuous curve $\gamma^+$ which connects $(r', 0)$ and $(r'', 0)$, lying completely in $P$ except for its ending point $(r', 0)$ and $(r'', 0)$. In addition, without loss of generality, we can assume that $\gamma_+([0,1]) \cap \{r=0\}=\emptyset$. Hence, the closed simple curve whose support is given by $\gamma^+([0,1]) \cup ([r', r'']\times \{0\})$ divides $\{r \geq 0\}\times \{y \geq 0 \}$ in two regions, a bounded one which we call $A_b$ and a unbounded one $A_u$. Since there exists $r_-$ in $(r_1, r_2)$ such that $u_{s, \lambda}(r_-) <0$, thanks to the continuity of $W_{s, \lambda}$ and since $\gamma^+(t) \not \in \{r\geq0\}\times \{0\}$ when $t \in (0, 1)$, we have that there exists $y_->0$ such that $(r_-, y_-) \in A_b \cap N$. But then, since $N$ is connected, the Jordan curve theorem implies that $N \subset A_b$. 

Since $(0, 0) \in A_u$, this implies that there exists $\delta >0$ such that $\overline{B^+_\delta} \subset \R^{n+1}_+$ does not intersect $N$, i.e. $W_{s, \lambda}(x, y) \geq 0$ in $\overline B_\delta^+$. Then, we reach a contradiction as a consequence of the strong maximum principle (see Proposition \ref{yannick}). Therefore $r_0=0$ cannot happen.\\

If $r_0 \in (0,r_1)$ we can find two points  $0<r' < r_0<r''<r_1$ such that $u_{s, \lambda}(r') u_{s, \lambda}(r'')>0$ and there is no change of sign between $r'$ and $r''$. In fact, $u_{s,\lambda}$ does not change sign in $(0,r_1)$, and in addition if $u_{s, \lambda}$ were identically zero on a subset of positive measure of $(0,r_1)$, then, from \cite[Theorem 1.4]{fellifall}, we would have that $u_{s,\lambda}$ is zero everywhere. Therefore we can find $r'$ and $r''$ satisfying the above  properties and then, by using the Claim, we deduce that $u_{s,\lambda}$ cannot vanish in $(r',r'')$ which leads to a contradiction.\\

The proof of the other cases $r_0 \in (r_1,r_2)$ and $r_0 \in (r_2,R)$ is identical, and thus the proof of $(a)$ is complete.\\

For the proof of $(b)$, let $r_0 \in [0,R)$ be a zero of $u_{s,\lambda}$ different from the node $r_1$. If $r_0 \neq 0$, then  by using the Claim and arguing as before we get a contradiction and we are done. If $u_{s,\lambda}(0)=0$ we observe that since we are assuming that $u_{s,\lambda}$ changes sign exactly once we cannot exclude this possibility by using a merely topological argument as before. 
The proof of $(b)$ is then complete. 
\end{proof}

\begin{section}{Proof of Theorem \ref{mainteo}}
\setcounter{cla}{0}
\begin{proof}[Proof of Theorem \ref{mainteo}]
Let $n \geq 7$ and $R>0$. Let $\hat \lambda\left(\frac{1}{2}\right)$ be the number given by Theorem \ref{twicechange} for $s_0=\frac{1}{2}$, and let $\tilde\lambda>0$  be such that both $\tilde\lambda \leq \hat \lambda\left(\frac{1}{2}\right)$, and 
\[
\sup_{s \in \left(\frac{1}{2}, 1 \right)}\left| c_{\mathcal{M}_{s, \lambda; rad}} - 2\frac{s}{n}S^{\frac{n}{2s}}_s\right|< \frac{1}{n}S^{\frac{n}{2}}_1 \quad \forall \lambda \in (0, \tilde\lambda),
\] 
are satisfied. The existence of such a number $\tilde\lambda$ is ensured by Lemma \ref{As:energas} by taking $s_0=\frac{1}{2}$.

Let us fix $\lambda \in (0,\tilde\lambda)$. We want to prove that there exists $\overline s \in \left(\frac{1}{2}, 1\right)$ such that for every $s \in (\overline s, 1)$ any least energy radial sign-changing solutions of Problem \eqref{fracBrezis} in $B_R$, changes sign exactly once. Indeed assume by contradiction that this is not the case. Then there exists a sequence $s_k\to 1^-$ and a sequence $(u_{s_k,\lambda})_{s_k}$ of least energy radial solutions which change sign at least twice, for any $k$. For brevity we omit the subscript $k$ in the above sequences. Thanks to definition of $\tilde \lambda$, then Theorem \ref{twicechange} holds and thus $u_{s,\lambda}$ changes sign exactly twice, for any $s \in (\frac{1}{2},1)$. By Theorem \ref{sconverg} we have that $(u_{s,\lambda})_s$ converges in $C^{0, \alpha}_{loc}(B_R)$ to $u_{1,\lambda}$ for every fixed $0<\alpha < \frac{1}{2}$, where $u_{1, \lambda}$ is a weak non trivial solution of
\[
\begin{cases}
-\Delta u_{1, \lambda} = \lambda u_{1, \lambda} + |u_{1, \lambda}|^{2^*_1-2}u_{1, \lambda} &\hbox{in}\ B_R,\\
u_{1, \lambda} = 0 &\hbox{in}\ \R^n \setminus B_R.
\end{cases}
\] 

On one hand, the definition of $\tilde \lambda$ imply that
\[
I(u_{s,\lambda}) = c_\mathcal{M}(s,\lambda) < \frac{2s}{n}S^{\frac{n}{2s}}_s + \frac{1}{n}S_1^{\frac{n}{2}}
\] 
so that, passing to the limit as $s \to 1^-$ we obtain that
\[
I(u_{1,\lambda}) < \frac{3}{n}S^{\frac{n}{2}}_1.
\]
This implies (by arguing as in \cite[Theorem 1.1]{Pacella}) that $u_{1,\lambda}$ changes sign once. 
On the other hand, denoting by $r^1_s$ and $r^2_s$ the nodes of $u_{s, \lambda}$, as $s \to 1^-$, the following holds:
\begin{enumerate}[(i)]
\item $r^1_s \not \to 0$;
\item $r^1_s - r^2_s \not \to 0$;
\item $r^2_s \not \to R$.
\end{enumerate}
This, together with the $C^{0, \alpha}$-convergence in compact subsets of $B_R$, implies that $u_{1, \lambda}$ changes sign at least twice, a contradiction. Let us prove $(i)-(iii)$. Without loss of generality, we can assume that $u_{s, \lambda} \geq 0$ in a neighborhood of the origin, for any $s$. 

Property $(ii)$ is a consequence of an energetic argument. Indeed, suppose that $r^1_s -r^2_s\to 0$. Then, we readily obtain a contradiction because by Remark \ref{ss:as} and \eqref{fatoulemma22} we have 
\[
0<C \leq \left( 1 - \frac{\lambda}{ \lambda_{1,s}}\right) S_s \leq S_{s, \lambda} \leq \frac{\|u^-_{s, \lambda}\|^2_s-\lambda|u^-_{s,\lambda}|^2_2}{|u^-_{s, \lambda}|^2_{2^*_s}}\leq |u^-_{s,\lambda}|^{2^*_s-2}_{2^*_s}
\]
and then, using also Lemma \ref{Linftyboundsc}, we obtain
\[
0<C \leq \int_{B_R}|u^-_{s,\lambda}|^{2^*_s}\de x \leq \omega_n|u_{s, \lambda}|^{2^*_s}_\infty\int_{r^1_s}^{r^2_s}r^{n-1}\de r \leq C\left( (r^2_s)^n-(r_s^1)^n\right) \to 0,
\]
which is absurd. We also observe that with the same proof $r^1_s \to 0$ and $r^2_s \to R$ cannot happen at the same time.

For $(i)$, since $u_{s, \lambda} \to u_{1, \lambda}$ in $C^{0, \alpha}_{loc}(B_R)$ for any fixed $0<\alpha < \frac{1}{2}$, it holds that, for a suitable compact $K \subset B_R$ containing the origin as interior point we have
\begin{equation}\label{holdercont}
|(u_{s,\lambda}-u_{1,{\lambda}})(x) - (u_{s,\lambda}-u_{1,{\lambda}})(y)| \leq C_K|x-y|^\alpha
\end{equation}
where $C_K$ is uniformly bounded with respect to $s$ and depends on $K$. 
If we suppose that $r^1_s \to 0$ and evaluate \eqref{holdercont} in $x = r^1_s$ and $y = 0$, we obtain
\[
\left|u_{1,\lambda}(0) - u_{1, \lambda}(r^1_s) - u_{s,\lambda}(0)\right| \to 0 \quad as \quad s \to 1^-.
\]
Since $u_{1, \lambda} \in C^{0,s_0}(K)$ we get $u_{s,\lambda}(0) \to 0$. In addition, since $u_{s,\lambda} \to u_{1,\lambda}$ a.e., this implies also that $u_{1,\lambda}(0) = 0$. As a consequence of \cite[Proposition 2]{Iacopetti}, we get that $0 = |u_{1,\lambda}(0)| = |u_{1, \lambda}|_\infty$ i.e., $u_{1, \lambda} \equiv 0$. This contradicts the non triviality of $u_{1,\lambda}$, which is ensured by Theorem \ref{sconverg}.

To conclude to proof it remains to show that $r_s^2 \to R$ cannot happen. 
Since we are assuming that $u_{s,\lambda}$ changes sign twice, by Theorem \ref{subsol} in the Appendix it follows that $u_{s,\lambda} \in C^{0,s}(\R^n)$ is a weak sub-solution of
\[
(-\Delta)^s u_{s,\lambda}\leq \lambda u_{s,\lambda} + |u_{s,\lambda}|^{2^*_s-2}u_{s,\lambda} \quad \text{in }\R^n.
\]

Then, arguing as in the proof of Theorem \ref{sconverg}, taking the limit as $s \to 1^-$, there exists $u_{1,\lambda} \in X_0^1(B_R)$ such that $u_{s,\lambda} \to u_{1,\lambda}$ in $L^2(\R^n)$, $u_{s,\lambda} \to u_{1,\lambda}$ in $C^{0, \alpha}_{loc}(B_R)$, and for every $\varphi \in C^\infty_c(\R^n)$ such that $\varphi \geq 0$, 
\begin{equation}\label{relazsubsol}
\int_{\R^n} \nabla u_{1,\lambda}\cdot \nabla \varphi \de x \leq \int_{\R^n}(\lambda u_{1,\lambda} + |u_{1,\lambda}|^{2^*_1-2}u_{1,\lambda}) \varphi \de x.
\end{equation}

Suppose now that $r^2_s \to R$ as $s \to 1^-$. Then, in view of $(ii)$, there exists $\delta >0$ such that on the set $\mathcal{I} = \{ x \in \R^n \ |\ R-\delta \leq |x| \leq R+\delta\}$ it holds that $u_{1,\lambda} \leq 0$.
Taking $\varphi \in C^\infty_c(\mathcal{I})$, $\varphi \geq 0$, from \eqref{relazsubsol} we readily get 
\[
\begin{cases}
(-\Delta) u_{1,\lambda} \leq 0 & \hbox{in}\ \mathcal I, \\
u \leq 0 & \hbox{on}\ \partial \mathcal{I},
\end{cases}
\]   
and then, by the strong maximum principle either $u <0$ or $u \equiv 0$ in $\mathcal{I}$, but this is absurd since it holds that $u_{1,\lambda}< 0$ in $R-\delta \leq |x|<R$ and $u_{1,\lambda} \equiv 0 $ in $R < |x| \leq R+\delta$. The proof is then complete. 
\end{proof}
\end{section}

\setcounter{cla}{0}

\begin{section}{Proof of Theorem \ref{mainteo2}}
In this section we study the asymptotic behavior of least energy radial sign-changing solutions of Problem \eqref{fracBrezis} in $B_R$ as $\lambda\to 0^+$. Theorem \ref{mainteo2} will be a consequence of the following results. 
 
Under the hypotheses of Theorem \ref{mainteo2} we set $M_{\lambda, \pm} := |u^\pm_\lambda|_\infty$, $\beta := \frac{2}{n-2s}$, and we denote by
\[
\begin{aligned}
t_\lambda&: = \max\{ t \in [0, R) \ | \ u_\lambda(t) = M_{\lambda, +}\}, \\
 r_\lambda &:=\hbox{the node of} \ u_\lambda,\\
\tau_\lambda& := \max\{ t \in (0, R) \ | \ u_\lambda(t) = -M_{\lambda, -}\}.
\end{aligned}
\]
Let us observe that since $u_\lambda$ changes sign once it holds that $t_\lambda < r_\lambda <\tau_\lambda$. Let us consider also the following quantities:
\[
Q_\lambda := \frac{M_{\lambda,+}}{M_{\lambda,-}}, \quad \sigma_\lambda := M_{\lambda,+}^\beta r_\lambda
\]

By Lemma \ref{As:energas} we already know that as $\lambda \to 0^+$, we have $M_{\lambda, \pm} \to +\infty$. The following result states the asymptotic behavior of the quantities $t_\lambda$, $r_\lambda$, $\tau_\lambda$, as $\lambda \to 0^+$. 
\begin{lemma}
We have that $t_\lambda$, $r_\lambda$, $\tau_\lambda \to 0$ as $\lambda \to 0^+$. 
\end{lemma}
\begin{proof}
Since $0\leq t_\lambda < r_\lambda< \tau_\lambda$, it suffices to prove that $\tau_\lambda\to 0$ as $\lambda \to 0^+$. Evaluating inequality \eqref{straussineq} in a point $x_0$ such that $|x_0| = \tau_\lambda$ we get that
\[
M_{\lambda,-} \leq C\|u_\lambda\|^2\frac{1}{\tau_\lambda^{\frac{n-2s}{2}}} \leq C\frac{1}{\tau_\lambda^{\frac{n-2s}{2}}},
\]
where the uniform bound on the Gagliardo norm is a consequence of Lemma \ref{As:energas}.
Since $M_{\lambda,-}\to + \infty$ we obtain the desired result.
\end{proof}

The following result concerns the asymptotic behavior of $Q_\lambda$ and $\sigma_\lambda$. 

\begin{lemma}\label{asintoticaraggi}
Up to a subsequence, as $\lambda \to 0^+$, we have
\begin{enumerate}[i)]
\item $ Q_\lambda \to +\infty $,
\item $ \sigma_\lambda \to +\infty$. 
\end{enumerate}

\end{lemma}
\begin{proof} The proof is divided in two steps. 
\begin{cla}\label{As:primocaso} The following facts hold:
\begin{enumerate}[(a)]
\item $\sigma_\lambda \to 0$ cannot happen;
\item if either $Q_\lambda \to l \in \R^+\setminus \{0\}$ or $Q_\lambda \to +\infty$, then $\sigma_\lambda \to L \in (0, +\infty)$ cannot happen.
\end{enumerate}
\end{cla}

Property $(a)$ is a straightforward consequence of Lemma \ref{As:energas}. Indeed, assume by contradiction that $\sigma_\lambda \to 0$, then
\[
\begin{aligned}
|u_\lambda^+|_{2^*_s}^{2^*_s} =&\ \int_{B_{r_\lambda}}|u^+_\lambda|^{2^*_s}\de x = \omega_n \int_0^{r_\lambda}|u^+_\lambda(\rho)|^{2^*_s}\rho^{n-1}\de \rho \\
\leq&\ \omega_n (M_{\lambda,+})^{n\beta} \int_0^{r_\lambda}\rho^{n-1}\de \rho = \frac{\omega_n}{n} (M_{\lambda,+}^\beta r_\lambda)^n \to 0,
\end{aligned}
\]
and this is absurd since by Lemma \ref{As:energas} we have $|u_\lambda^+|_{2^*_s}^{2^*_s} \to S_s^{\frac{n}{2s}}$, and $(a)$ is proved.

For $(b)$, let us consider the rescaled functions
\begin{equation}\label{riscalglob}
\tilde u_\lambda (x) = \frac{1}{M_{\lambda,+}} u_\lambda\left(\frac{x}{M_{\lambda,+}^\beta}\right).
\end{equation}
By the assumption on $Q_\lambda$, we have that $\tilde u_\lambda$ is definitely uniformly bounded in $L^\infty$. Moreover $\tilde u_\lambda$ weakly solves the problem
\[
\begin{cases}
(-\Delta)^s u = \frac{\lambda}{M_{\lambda,+}^{2^*_s-2}}u + |u|^{2^*_s-2}u & B_{M^\beta_{\lambda,+}R}\\
u=0 & \R^n \setminus B_{M^\beta_{\lambda,+}R}
\end{cases}
\]
and then, by Remark \ref{gilbcont}, we have that there exists $\tilde u_0$ such that $\tilde u_\lambda \to \tilde u_0$ in $C_{loc}^{0,\alpha}(\R^n)$ for every $\alpha <s$. Suppose by contradiction that $\sigma_\lambda \to L$. Since $M^\beta_{\lambda, +}t_\lambda \leq \sigma_\lambda$, this implies that up to a subsequence, there exists $\tilde x$ such that $|\tilde x| \leq L$ and $\tilde u(\tilde x) = 1$. Together with the $C^{0,\alpha}_{loc}$ convergence this implies that $\tilde u \not \equiv 0$.

Let now $\tilde u^+_\lambda$ be the rescaling of $u^+_\lambda$. Since $\sigma_\lambda \to L$, there exists $\bar R$ such that $\tilde u^+_\lambda \in X^s_0(B_{\bar R})$ for all sufficiently small $\lambda>0$. Moreover by Lemma \ref{As:energas} (noticing that also $\tilde u^+_\lambda$ and $u^+_\lambda$ satisfy relations in the form \eqref{As:rescaling}), we have that $\|\tilde u^+_\lambda \|_s^2 = \|u^+_\lambda\|_s^2 \to S_s^{\frac{n}{2s}}$ so that $(\tilde u_\lambda^+)$ is a bounded sequence in $X^s_0(B_{\bar R})$. Hence, there exists $u_* \in X^s_0(B_{\overline R})$ such that $\tilde u^+_\lambda \rightharpoonup u_*$ in $X^s_0(B_{\overline R})$ and $\tilde u^+_\lambda \to u_*$ almost everywhere.  But then $u_* = u_0^+ \geq 0$ and $u_* \in X^s_0(B_L)$. Moreover, there exists $\rho  \in [0, L)$ such that for every $|x| = \rho$ then $u_*(x) = 1 $, so that $u_* \not \equiv 0$.

Since $|\tilde u^+_\lambda|_\infty \leq 1 $ and $\text{supp }\tilde u^+_\lambda \subset B_{\overline R}$, by Lebesgue's dominated convergence theorem we get that $\tilde u^+_\lambda \to u_*$ strongly in $L^{2^*_s}$. Using this and Fatou's Lemma we obtain
\[
S_s \leq \frac{\|u_*\|_s^2}{|u_*|^2_{2^*_s}} \leq \liminf_{\lambda \to 0^+} \frac{\|\tilde u^+_\lambda\|_s^2}{|\tilde u^+_\lambda|^2_{2^*_s}} = S_s,
\]
i.e. $u_*$ realizes the infimum $S_s$ despite being supported on a bounded domain $B_L$ and thus contradicting \cite[Theorem 1.1]{CotTav}.

\begin{cla}\label{As:terzocaso}
The following holds:
\begin{enumerate}[(a)]
\setcounter{enumi}{2}
\item $M_{\lambda_-}^\beta \tau_\lambda \to +\infty$ cannot happen.
\item $M_{\lambda,+}^\beta t_\lambda \to +\infty$ cannot happen.
\end{enumerate}

Since the proofs are identical, we show only $(c)$. Since by Lemma \ref{As:energas} we have that $(u_\lambda)$ is a bounded sequence in $\mathcal{D}^s(\R^n)$, evaluating \eqref{straussineq} in a point $x_0$ such that $|x_0| = \tau_\lambda$ we get that
\[
(M^\beta_{\lambda_-}\tau_\lambda)^{\frac{n-2s}{2}} = \tau_\lambda^{\frac{n-2s}{2}}M_{\lambda,-} = |x_0|^{\frac{n-2s}{2}}|u_\lambda(x_0)|\leq \sup_{x \in \R^n \setminus\{0\}}|x|^{\frac{n-2s}{2}}|u_\lambda(x)|\leq K_{n,s}\|u_\lambda\|_s^2 \leq C,
\]
which proves the claim.
\end{cla}

Now we can prove i). Since $Q_\lambda >0$, up to a subsequence, as $\lambda \to 0^+$ we have that $Q_\lambda \to l \in [0, +\infty]$. Suppose that $Q_\lambda \to 0$. Since by Step \ref{As:terzocaso} we have that $M^\beta_{\lambda, -}\tau_\lambda \not \to +\infty$ we get that
\[
0 \leftarrow \left(Q_\lambda\right)^\beta M^\beta_{\lambda,-}\tau_\lambda = M^\beta_{\lambda,+}\tau_\lambda \geq  \sigma_\lambda \geq 0,
\]
which is impossible by $(a)$. Assume now $Q_\lambda \to l \in (0, +\infty)$. By Step \ref{As:primocaso} this implies that $\sigma_\lambda \to +\infty$, but then
\[
+\infty \leftarrow \left(\frac{1}{Q_\lambda}\right)^\beta \sigma_\lambda = M^\beta_{\lambda,-}r_\lambda \leq M^\beta_{\lambda,-}\tau_\lambda,
\]
and this is impossible by Step \ref{As:terzocaso}. Therefore the only possibility is $Q_\lambda \to +\infty$, and $i)$ is proved. For $ii)$, we observe that $i)$ and Step \ref{As:primocaso} imply that, up to a subsequence, the only possibility is $\sigma_\lambda \to +\infty$, as $\lambda \to 0^+$. The proof is complete.
\end{proof}

\begin{prop}
Under the hypotheses of Theorem \ref{mainteo2}, up to a subsequence, as $\lambda \to 0^+$ we have that the function
\[
\tilde u^+_\lambda (x) = \frac{1}{M_{\lambda,+}}u^+_\lambda\left( \frac{x}{M_{\lambda,+}^\beta}\right), 
\]
converges to $U(x) = k_{\hat \mu} \frac{\hat \mu^{n-2s}}{(\hat \mu^2 + |x|^2)^{\frac{n-2s}{2}}}$ in $C^{0, \alpha}_{loc}(\R^n)$, for every fixed $\alpha \in (0,s)$, and strongly in $\mathcal{D}^s(\R^n)$, where
\[
\hat \mu = S_s^{\frac{1}{2s}}\left(\int_{\R^n}\frac{1}{(1+|x|^2)^n}\de x\right)^{-\frac{1}{n}},
\]
and $k_{\hat \mu}$ is as in \eqref{optbubbcost}.
\end{prop}
\begin{proof}
Let $\tilde u_\lambda$ be the rescaling defined in \eqref{riscalglob}. Notice that also in this case the norms of $\tilde u_\lambda$ and $u_\lambda$ are related as in \eqref{As:rescaling}. Then we get that $(\tilde u_\lambda)$ is a bounded sequence in $\mathcal{D}^s(\R^n)$ by Lemma \ref{As:energas}. Hence, up to a subsequence, $\tilde u_\lambda$ weakly converges to $\tilde u_0$ in $\mathcal{D}^s(\R^n)$, strongly in $L^p_{loc}(\R^n)$ for every $p \in (1, 2^*_s)$ and also almost everywhere in $\R^n$. The same holds for $\tilde u_\lambda^\pm$, and in particular $\tilde u_\lambda^\pm \to \tilde u^\pm_0$ a.e. As a consequence of Lemma \ref{asintoticaraggi} we have that $\tilde u^+_\lambda \to \tilde u_0$ and $\tilde u^-_\lambda \to 0$ almost everywhere, so that $\tilde u_0 \geq 0$. On the other hand the function $\tilde u_\lambda$ weakly satisfies 
\begin{equation}\label{profilobubble1}
\begin{cases}
(-\Delta)^s \tilde u_\lambda = \frac{\lambda}{M_{\lambda,+}^{2^*_s-2}}\tilde u_\lambda + |\tilde u_\lambda|^{2^*_s-2}\tilde u_\lambda & \hbox{in} \ B_{M^\beta_{\lambda,+}R},\\
\tilde u_\lambda =0 & \hbox{in} \ \R^n \setminus B_{M^\beta_{\lambda,+}R},
\end{cases}
\end{equation}
and  thanks to Proposition \ref{strauss} the point where the maximum of $\tilde u_\lambda$ is achieved stays in a compact subset $K \subset \subset \R^n$.
By Lemma \ref{asintoticaraggi} and the definition of $\tilde u_\lambda$ we have $|\tilde u_\lambda|_\infty \leq 1$, and hence by a standard argument (as seen in Remark \ref{gilbcont}, but here $s$ is fixed) 
we obtain that $\tilde u_\lambda \to u_*$ in $C^{0,\alpha}_{loc}(\R^n)$ for every fixed $\alpha \in (0,s)$, and $u_* = \tilde u_0$ thanks to the a.e. convergence. Therefore, since the maximum of $\tilde u_\lambda$ definitely stay in compact subset $K$ of $\R^n$, there exists $\overline x \in K$ such that $\tilde u_0(\overline x) =1$ so that $\tilde u_0$ is not trivial.
Passing to the limit in \eqref{profilobubble1} we deduce that $\tilde u_0$ weakly solves 
\begin{equation}\label{sasasa}
(-\Delta)^s \tilde u_0 = |\tilde u_0|^{2^*_s-2}\tilde u_0 \quad \hbox{in}\ \R^n.
\end{equation}
Since $\tilde u_0$ is a non trivial solution of \eqref{sasasa} we obtain 
\[
S_s \leq \frac{\|\tilde u_0\|_s^2}{|\tilde u_0|^2_{2^*_s}} = |\tilde u_0|^{2^*_s-2}_{2^*_s}.
\]
On the other hand by Fatou's lemma we have that $|\tilde u_0 |^{2^*_s}_{2^*_s} \leq \liminf_{\lambda \to 0^+}|\tilde u^+_\lambda|^{2^*_s}_{2^*_s} = S_s^{\frac{n}{2s}}$.

Therefore $|\tilde u_\lambda|^{2^*_s}_{2^*_s} \to |\tilde u_0|^{2^*_s}_{2^*_s}$ and then $\tilde u^+_\lambda \to \tilde u_0$ strongly in $L^{2^*_s}(\R^n)$ thanks to the Brezis-Lieb's Lemma. 
Hence we obtain that $\tilde u_0$ is a minimizer for $S_s$ and a solution of \eqref{sasasa}, thus it has to be of the form \eqref{eq:bubble}, for some $\mu \in \R$, $x_0 \in \R^n$. Since $\tilde u$ is the limit of radial functions, then  $\tilde u$ is radial, and thus it must be $x_0 = 0$. Then, by construction we get that $u(0) = 1$ and that $\mu = S_s^{\frac{1}{2s}}\left(\int_{\R^n}\frac{1}{(1+|x|^2)^n}\de x\right)^{-\frac{1}{n}} $, hence the proposition is proved.  
\end{proof}
\end{section}

\appendix

\begin{section}{Some technical results}

Let $s \in (\frac{2}{3}, 1)$, let $n >6s$, let $\lambda \in (0, \lambda_{1,s})$ and let $R>0$. Let $u_{s, \lambda}$ be a least energy radial sign-changing solution of Problem \eqref{fracBrezis} in $B_R$, and let $W_{s, \lambda} = E_{s, \lambda}u_{s, \lambda}$ be its extension (see \eqref{extensiondefinition}).  Let $\delta \in (0,R)$ and define $A_\delta$ as the set 
\begin{equation}\label{adelta}
A_\delta = \{ x \in \R^n \ |\ |R - |x||>\delta \}.
\end{equation}
A standard computation shows that there exists $C>0$ which depends on $n$, $s$, $\lambda_{1,s}$ and $\delta$ such that
\begin{equation}\label{app1}
\sup_{x \in A_\delta}|(-\Delta)^s u_{s, \lambda}(x)| \leq C. 
\end{equation}
Indeed this is trivial when $x \in A_\delta \cap (\R^n \setminus B_R)$, while, when $x \in A_\delta \cap B_R$ it is a consequence of \cite[Corollary 1.6, (a)]{HolderReg} and keeping in mind that the fractional Laplacian can be written in the alternative form 
\[
(-\Delta)^s u(x) = -\frac{C_{n,s}}{2}\int_{\R^n}\frac{u(x+y) + u(x-y) - 2u(x)}{|y|^{n+2s}}\de y.
\]

As a consequence, for any $x \in \R^n$ such that $|x| \neq R$, the following pointwise relations hold:
\begin{equation}\label{app2}
\lim_{\varepsilon \to 0^+}-d_s \varepsilon^{1-2s}\frac{\partial W_{s, \lambda}}{\partial y}(x, \varepsilon) = \lim_{\varepsilon \to 0^+}-2sd_s\frac{W_{s, \lambda}(x, \varepsilon) - W_{s, \lambda}(x, 0)}{\varepsilon^{2s}} = (-\Delta)^s u_{s, \lambda}(x).
\end{equation}

Moreover, being $u_{s, \lambda}$ a solution of Problem \eqref{fracBrezis} and thanks to \eqref{app1}, \eqref{app2}, taking into account of the relation $p_{n,s}\int_{\R^n}P_{n,s}(x- \xi, y) \de \xi = 1$, for any $y>0$, then some simple computations lead to the following standard result. 
\begin{lemma}\label{app3}
Let $\delta >0$ and $A_\delta$ be as in \eqref{adelta}. For every $\varphi \in L^2(A_\delta)$ such that $\text{supp } \varphi$ is bounded in $A_\delta$, it holds that 
\[
\begin{aligned}
&\lim_{\varepsilon \to 0^+}\int_{A_\delta}-2sd_s\frac{W_{s, \lambda}(x, \varepsilon) - W_{s, \lambda}(x, 0)}{\varepsilon^{2s}}\varphi(x) \de x \\
&=\lim_{\varepsilon \to 0^+}\int_{A_\delta} - d_s\varepsilon^{1-2s} \frac{\partial W_{s, \lambda}}{\partial y}(x, \varepsilon) \varphi(x) \de x = \int_{A_\delta}(-\Delta)^s u_{s, \lambda}(x) \varphi(x) \de x. 
\end{aligned}
\]

\end{lemma}

The main result of the section is the following Theorem. 

\begin{teo}\label{subsol}
Let $n$, $s$, $R$, and $\lambda$ be as above. Let $u_{s, \lambda}$ be a least energy radial sign-changing solution of Problem \eqref{fracBrezis} in $B_R$ which changes sign exactly twice and such that $u_{s, \lambda} \geq 0$ in a neighborhood of the origin. Then $u_{s,\lambda}$ is a weak sub-solution of
\[
(-\Delta)^s u_{s, \lambda} \leq \lambda u_{s, \lambda} + |u_{s, \lambda}|^{2^*_s-2}u_{s, \lambda} \quad \text{ in }\R^n
\] 
i.e., for every $\varphi \in C^\infty_c(\R^n)$ such that $\varphi \geq 0$ it holds that
\[
(u_{s, \lambda}, \varphi)_s \leq \int_{\R^n}(\lambda u_{s, \lambda} + |u_{s, \lambda}|^{2^*_s-2}u_{s, \lambda}) \varphi \de x.
\]
\end{teo}
\begin{proof}
Let $\varphi \in C^\infty_c(\R^n)$, $\varphi \geq 0$. We observe that there always exists a function $\phi \in C^\infty_c(\overline{\R^{n+1}_+})$ such that $\phi \geq 0$ and $\phi(x, 0) = \varphi (x)$. By Lemma \ref{extreg} we have that 
\begin{equation}\label{equazagg1}
(u_{s, \lambda}, \varphi)_s = \int_{\R^{n+1}_+} y^{1-2s} \nabla W_{s, \lambda}\cdot \nabla \phi \de x \de y. 
\end{equation}
Let $0<r_s^1<r_s^2<R$ be the nodes of $u_{s,\lambda}$. As proved in Lemma \ref{lemmatecnico} (see \eqref{asintotre1}), for every $\overline \rho  \in (r^2_s,R)$ such that $u_{s, \lambda}(\overline \rho) >0$ there exists $\overline \varepsilon>0$ such that
\begin{equation}\label{segno}
W_{s, \lambda}(x, y) \geq 0 \quad \text{ in } \{x \ |\ |x|>\overline\rho\} \times (0, \overline \varepsilon).
\end{equation}
Let us fix $\overline \rho \in (r_s^2,R)$ and let $\delta := R-\overline \rho$. For $\rho \in \left(0, \frac{\delta}{2}\right)$, let $T_\rho$ be the semitoroidal open set defined by
\[
T_\rho = \{ (x, y) \in \R^{n+1}_+ \ | \  \sqrt{(R-|x|)^2 + y^2} < \rho\}.
\]
Since $W_{s, \lambda}$, $\phi \in \mathcal{D}^{1,s}(\R^{n+1}_+)$, by the Lebesgue's dominated convergence theorem we deduce that
\begin{equation}\label{opiccolorho}
\lim_{\rho \to 0^+} \int_{T_\rho} y^{1-2s} \nabla W_{s, \lambda}\cdot \nabla \phi \de x \de y = 0.
\end{equation}
Let $\varepsilon < \min\left\{\frac{\rho}{2}, \overline \varepsilon\right\}$. We have that 
\[
\begin{aligned}
&\int_{\R^{n+1}_+\setminus T_\rho} y^{1-2s} \nabla W_{s, \lambda}\cdot \nabla \phi \de x \de y \\
=&\ \int_{(\R^{n+1}_+ \setminus T_\rho)\cap \{y\geq  \varepsilon\}} y^{1-2s} \nabla W_{s, \lambda}\cdot \nabla \phi \de x \de y + \int_{(\R^{n+1}_+ \setminus T_\rho)\cap \{y<\varepsilon\}} y^{1-2s} \nabla W_{s, \lambda}\cdot \nabla \phi \de x \de y \\
=&\ \int_{(\R^{n+1}_+ \setminus T_\rho)\cap \{y \geq \varepsilon\}} y^{1-2s} \nabla W_{s, \lambda}\cdot \nabla \phi \de x \de y + o_\rho(1)
\end{aligned}
\]
where the last equality is obtained arguing as in \eqref{opiccolorho}, and $o_\rho(1)$ is a function of $\rho$ and $\varepsilon$ such that $\lim_{\varepsilon \to 0^+}o_\rho(1)=0$ for any $\rho \in (0,\frac{\delta}{2})$. Integrating by parts, using that $\phi \in C^\infty_c(\overline{\R^{n+1}_+})$ we obtain
\[
\begin{aligned}
&\int_{\R^{n+1}_+\setminus T_\rho} y^{1-2s} \nabla W_{s, \lambda}\cdot \nabla \phi \de x \de y \\
=& \int_{|x|\leq R - \sqrt{\rho^2 - \varepsilon^2} } -\varepsilon^{1-2s} \frac{\partial W_{s, \lambda}}{\partial y}(x, \varepsilon) \phi(x, \varepsilon) \de x + \ \int_{\partial T_{\rho, \varepsilon}} y^{1-2s} \nabla W_{s, \lambda}\cdot \nu \phi \de \sigma  \\
+&\ \int_{|x|\geq R + \sqrt{\rho^2 - \varepsilon^2} } -\varepsilon^{1-2s} \frac{\partial W_{s, \lambda}}{\partial y}(x, \varepsilon) \phi(x, \varepsilon) \de x  + o_\rho(1), 
\end{aligned}
\] 
where $T_{\rho, \varepsilon} = T_{\rho}\cap \{y \geq \varepsilon\}$, $\nu$ is the exterior normal to $\partial T_{\rho, \varepsilon}$, 
and $d\sigma$ is the surface measure of $\partial T_{\rho, \varepsilon}$. Notice that thanks to the choice of $\varepsilon$ we have $T_{\rho, \varepsilon} \neq \emptyset$. Moreover, using the definition via Fourier transform of the Gagliardo semi-norm, we infer that
\[
\left|\int_{\R^n} -\varepsilon^{1-2s} \frac{\partial W_{s, \lambda}}{\partial y}(x, \varepsilon) (\phi(x, \varepsilon)-\varphi(x)) \de x\right| \leq \left\|-\varepsilon^{1-2s} \frac{\partial W_{s, \lambda}}{\partial y}(x, \varepsilon) \right\|_{-s}\|\phi(x, \varepsilon)- \varphi(x)\|_s, 
\] 
so that, by \eqref{dualconv} and since $\|\phi(x, \varepsilon)- \varphi(x)\|_s \to 0$ as $\varepsilon \to 0^+$, we get that
\[
\begin{aligned}
\int_{|x|\leq R - \sqrt{\rho^2 - \varepsilon^2} } -\varepsilon^{1-2s} \frac{\partial W_{s, \lambda}}{\partial y}(x, \varepsilon) \phi(x, \varepsilon) \de x = 
\int_{|x|\leq R - \sqrt{\rho^2 - \varepsilon^2} } -\varepsilon^{1-2s} \frac{\partial W_{s, \lambda}}{\partial y}(x, \varepsilon) \varphi(x) \de x + o_\rho(1), \\
 \int_{|x|\geq R + \sqrt{\rho^2 - \varepsilon^2} } -\varepsilon^{1-2s} \frac{\partial W_{s, \lambda}}{\partial y}(x, \varepsilon) \phi(x, \varepsilon) \de x =  \int_{|x|\geq R + \sqrt{\rho^2 - \varepsilon^2} } -\varepsilon^{1-2s} \frac{\partial W_{s, \lambda}}{\partial y}(x, \varepsilon) \varphi(x) \de x + o_\rho(1),
\end{aligned}
\]
and then
\begin{equation}\label{equazagg}
\begin{aligned}
&\int_{\R^{n+1}_+\setminus T_\rho} y^{1-2s} \nabla W_{s, \lambda}\cdot \nabla \phi \de x \de y \\
=&\ \int_{|x|\leq R - \sqrt{\rho^2 - \varepsilon^2} } -\varepsilon^{1-2s} \frac{\partial W_{s, \lambda}}{\partial y}(x, \varepsilon) \varphi(x) \de x + \int_{\partial T_{\rho, \varepsilon}} y^{1-2s} \nabla W_{s, \lambda}\cdot \nu \phi \de \sigma  \\
+&\ \int_{|x|\geq R + \sqrt{\rho^2 - \varepsilon^2} } -\varepsilon^{1-2s} \frac{\partial W_{s, \lambda}}{\partial y}(x, \varepsilon) \varphi(x) \de x  + o_\rho(1) \\
=&\ (I) + (II) + (III) + o_\rho(1).  
\end{aligned}
\end{equation}

For the term $(I)$, by Lemma \ref{app3}, and since $u_{s,\lambda}$ is a solution of Problem \eqref{fracBrezis}, we obtain that 
\begin{equation}\label{pezzoI}
\lim_{\varepsilon \to 0^+} (I) = \int_{B_{R-\rho}}(-\Delta)^s u_{s, \lambda}\varphi \de x = \int_{B_{R-\rho}}\left(\lambda u_{s, \lambda} + |u_{s, \lambda}|^{2^*_s-2}u_{s, \lambda}\right) \varphi \de x.
\end{equation}

For $(III)$, again by Lemma \ref{app3} we have
\[
\lim_{\varepsilon \to 0^+}(III) = \lim_{\varepsilon \to 0^+} \int_{\R^n \setminus B_{R+\sqrt{\rho^2-\varepsilon^2}}}-2sd_s \frac{W_{s, \lambda}(x, \varepsilon) - W_{s, \lambda}(x, 0)}{\varepsilon^{2s}}\varphi(x) \de x.
\]
We observe that, as a consequence of \eqref{segno}, when $x \in \R^n \setminus B_{R+\sqrt{\rho^2-\varepsilon^2}}$ and $\varepsilon < \overline \varepsilon$ we have $W_{s, \lambda}(x, \varepsilon) = W_{s, \lambda}(x, \varepsilon) - W_{s, \lambda}(x, 0) \geq 0$. Moreover, by \eqref{app1} and \eqref{app2}, we infer that the limit $\lim_{\varepsilon \to 0^+}(III)$ exists and it is finite. 
In particular we have
\begin{equation}\label{pezzoII}
\lim_{\varepsilon \to 0^+}(III) = h(\rho) \leq 0,
\end{equation}
where $h$ is a non-positive function which depends on $\rho$.

As a consequence of \eqref{equazagg}, \eqref{pezzoI} and \eqref{pezzoII} there exists $q = q(\rho)$ such that 
\begin{equation}\label{limfinito}
\lim_{\varepsilon \to 0^+}(II) = q(\rho).
\end{equation}
In particular, $q$ does not depends on $\varepsilon$. In addition, in Lemma \ref{stimaintegrale} we will show that there exists $C>0$ which does not depends on $\rho$ such that 
\begin{equation}\label{pezzoIII}
-C\rho^{1-s} \leq \liminf_{\varepsilon \to 0^+}(II) \leq \limsup_{\varepsilon \to 0^+}(II) \leq C \rho^{1-s}.
\end{equation}

We can now conclude the proof of the Theorem. Taking in account of \eqref{equazagg1}, \eqref{opiccolorho}, \eqref{pezzoI}, \eqref{pezzoII} and \eqref{pezzoIII} we get that
\begin{eqnarray*}
(u_{s, \lambda}, \varphi)_s &=& \limsup_{\varepsilon \to 0^+}((I) + (II) + (III) + o_\rho(1)) + o(1)\\
& \leq& \int_{B_{R-\rho}}\left(\lambda u_{s, \lambda} + |u_{s, \lambda}|^{2^*_s-2}u_{s, \lambda}\right) \varphi \de x + C\rho^{1-s}+o(1)\\
&\leq&  \int_{B_{R}}\left(\lambda u_{s, \lambda} + |u_{s, \lambda}|^{2^*_s-2}u_{s, \lambda}\right) \varphi \de x + C\rho^{1-s}+o(1),
\end{eqnarray*}
where we have used that $u_{s,\lambda}\geq0$ in $(r_s^2,R)$ and $\varphi\geq0$, and where $o(1)$ is the term coming from \eqref{opiccolorho} and thus depending on $\rho$ but not on $\varepsilon$, and such that $\lim_{\rho \to 0^+}o(1) = 0$. Hence, passing to the limit as $\rho \to 0^+$ we get the desired result and the proof is complete.

\end{proof}

\begin{lemma}\label{stimaintegrale}
There exists a constant $C>0$ such that for all sufficiently small $\rho>0$
\[
-C\rho^{1-s}\leq \liminf_{\varepsilon \to 0^+}(II) \leq \limsup_{\varepsilon \to 0^+}(II) \leq C \rho^{1-s}.
\]
\end{lemma}
\begin{proof}
Let $\Sf^{n-1}$ be the unit sphere in $\R^n$. For any $\mathbf{e} \in \Sf^{n-1}$, since $W_{s, \lambda}$ is cylindrical symmetric, we have $W_{s, \lambda}(x, y)  = W_{s, \lambda}(\mathbf{e}|x|, y),\ \hbox{for any}\ x \in \R^n, y\geq 0.$

In particular, without loss of generality, $W_{s, \lambda}(x, y)  = W_{s, \lambda}(e_1|x|, y),\ \hbox{for any}\ x \in \R^n, y\geq 0$, where $e_1= (1, 0, \ldots, 0)$.
Let $\delta >0$, $\rho \in \left(0, \frac{\delta}{2}\right)$ and $\varepsilon \in \left(0, \min\{\frac{\rho}{2}, \overline \varepsilon\}\right)$ as in the proof of Theorem \ref{subsol}.
Since $\rho < R$, we can express the set $\partial T_{\rho, \varepsilon}$ in the following way
\begin{equation}\label{parametrizT}
\partial T_{\rho, \varepsilon} = \left \{\left((R-\rho \cos \theta)\mathbf{e}, \rho \sin \theta \right) \ | \ \theta \in (\theta_\rho(\varepsilon), \pi - \theta_\rho(\varepsilon)), \mathbf{e} \in \Sf^{n-1}\right\},
\end{equation}
where $\theta_\rho(\varepsilon)= \arcsin \frac{\varepsilon}{\rho}$. We notice that since $\varepsilon \in \left(0, \frac{\rho}{2}\right)$, then $\theta_\rho(\varepsilon) \in \left(0, \frac{\pi}{6}\right)$. Moreover $\varepsilon\mapsto\theta_\rho(\varepsilon)$ is continuous, monotone and  $\lim_{\varepsilon\to 0^+}\theta_\rho(\varepsilon) = 0$, for any fixed $\rho \in (0,\frac{\delta}{2})$. Since all the estimates that we are going to prove will be uniform with respect to $\theta \in (0,\pi)$ we drop for brevity the subscript $\rho$ in $\theta_\rho(\epsilon)$.

Exploiting the cylindrical symmetry of $W_{s, \lambda}$ we notice that when $(x, y) \in \partial T_{\rho, \varepsilon}$ we can express $W_{s, \lambda}$ just using using the coordinates $\rho$, $\theta$, obtaining that
\[
W_{s, \lambda}(\rho, \theta) = \int_{\R^n}\frac{(\rho \sin \theta)^{2s}}{(\rho \sin \theta)^2 + |(R- \rho \cos \theta)e_1- \xi^2|)^{\frac{n+2s}{2}}}u_{s, \lambda}(\xi) \de \xi.
\]

Now, denoting by $\nu$ the exterior normal to the surface $\partial T_{\rho, \varepsilon}$, and taking account of the orientations, by a simple computation we have that 
\[
\nabla W_{s, \lambda}(x, y) \cdot \nu(x, y)\big|_{\partial T_{\rho, \varepsilon}} = - \frac{\partial W_{s, \lambda}}{\partial \rho}(\rho, \theta).
\]  

Therefore, by a slight abuse of notation, parametrizing the hypersurface $\partial T_{\rho, \varepsilon}$ as in \eqref{parametrizT} with the coordinates $(\theta, \mathbf{e})$, we obtain 
\[
(II) = \int_{\Sf^{n-1}} \int_{\theta(\varepsilon)}^{\pi - \theta(\varepsilon)}-(\rho \sin \theta)^{1-2s}\frac{\partial W_{s, \lambda}}{\partial \rho}(\rho, \theta) \phi(\rho, \theta, \mathbf{e})\rho \Psi(\rho, \theta, \mathbf{e}) \de \theta \de \mathbf{e},
\]
where $\rho \Psi(\rho, \theta, \mathbf{e})$ is a positive factor coming from the definition of surface measure. In particular, $\Psi$ is uniformly bounded when $\rho$ is small, and Lipschitz continuous with respect to $\theta$, uniformly in $\rho$ and $\mathbf{e}$.
 More precisely, using for $\mathbb{S}^{n-1}$ the atlas $\{(U_l,\psi_l^{-1})\}_{l=1,\ldots,n}$ whose charts are the spherical coordinates 
$\psi_l^{-1}:U_l \to V$ (see \cite{AbateTovena}, Example 2.1.29), where $V :=\{(\theta_1,...,\theta_{n-1})\in \R^{n-1} | \ 0<\theta_1<2\pi,\ 0<\theta_h <\pi\ \hbox{for}\  h=2,\ldots,n-1\}$, then, local parametrizations for $\partial T_{\rho, \varepsilon}$ are given by the maps
$\tilde \psi_l:V\times (\theta(\varepsilon),\pi-\theta(\varepsilon)) \to \partial T_{\rho, \varepsilon}$, $\tilde \psi_l (\theta_1,\ldots,\theta_{n-1},\theta):=((R-\rho\cos\theta)\psi_l(\theta_1,\ldots,\theta_{n-1}),\rho \sin \theta)$. Now, since the matrix $(g_{ij})_{i,j=1,\ldots,n-1}$ of the induced metric on $\mathbb{S}^{n-1}$ by the spherical coordinates is diagonal and given by $g_{ij}=\delta_{ij} (\sin \theta_{i+1}\cdots\sin \theta_{n-1})^2$ (see \cite{AbateTovena}, Example 6.5.22), then, for each parametrization $\tilde \psi_l$ the determinant of the matrix $(\tilde g_{ij})_{i,j=1,\ldots,n}$ of the induced metric on $ \partial T_{\rho, \varepsilon}$ is $\rho^2 (\cos\theta)^2 (R-\rho\cos\theta)^{2(n-1)}\Pi_{i=1,\ldots,n-1}  (\sin \theta_{i+1}\cdots\sin \theta_{n-1})^2$ and thus its square root is $\rho |(\cos\theta)| (R-\rho\cos\theta)^{(n-1)}\Pi_{i=1,\ldots,n-1}  |(\sin \theta_{i+1}\cdots\sin \theta_{n-1})|$. Therefore $\Psi(\rho, \theta,\theta_1,\ldots,\theta_{n-1})$ has the desired properties. For brevity we will use the more compact notation $\Psi(\rho, \theta, \mathbf{e})$.

Now, by an elementary computation we obtain that
\begin{equation}\label{derivrho}
\begin{aligned}
&-(\rho \sin \theta)^{1-2s}\frac{\partial W_{s, \lambda}}{\partial \rho}(\rho, \theta) \rho\\
=&\ -p_{n,s}2s \rho \sin \theta \int_{B_R} \frac{u_{s, \lambda}(\xi)}{((\rho \sin \theta)^2 + |(R- \rho \cos \theta)e_1- \xi|^2)^{\frac{n+2s}{2}}} \de \xi \\
+&\ p_{n,s}(n+2s) \rho^2 \sin \theta \int_{B_R} \frac{u_{s, \lambda}(\xi) ( \rho - R\cos \theta + \cos \theta (e_1, \xi))}{((\rho \sin \theta)^2 + |(R- \rho \cos \theta)e_1- \xi|^2)^{\frac{n+2s+2}{2}}}\de \xi,
\end{aligned}
\end{equation}
where we used that $u_{s, \lambda}(\xi)  = 0$ when $\xi \in \R^n \setminus B_R$. Let us define the set
\[
C_\delta = \{ \xi \in B_R \ |\  |\xi - Re_1| < \delta\}. 
\]
We can split the integrals appearing in \eqref{derivrho} taking as domains of integrations $C_\delta$ and $B_R \setminus C_\delta$. Since $\rho \in \left(0, \frac{\delta}{2} \right)$, when $\xi \in B_R \setminus C_\delta$ the relation $|(R-\rho \cos \theta)e_1 - \xi| > \frac{\delta}{2}$ holds, and thus all the quantities appearing in the integrals over $B_R \setminus C_\delta$ are bounded from above and below, respectively, by $\pm C$, where $C>0$ is a constant which depends only on $n$, $s$, $R$, $|u_{s, \lambda}|_\infty$ and $\delta$ (but not on $\rho$, $\theta$, and hence neither on $\varepsilon$). 

Taking this into account and performing the change of variable $\xi = \eta + Re_1$, after some computations we can write
\begin{equation}\label{integrale1}
\begin{aligned}
&-(\rho \sin \theta)^{1-2s}\frac{\partial W}{\partial \rho}(\rho, \theta) \rho\\
=&\ O (\rho) + p_{n,s}\rho \sin \theta \int_{C_\delta-Re_1} \frac{u_{s, \lambda}(\eta +Re_1)(n\rho^2 - 2s|\eta|^2 + (n-2s)\rho|\eta|\cos \theta (e_1, \hat \eta))}{(\rho^2 + |\eta|^2 + 2 \rho|\eta|\cos \theta(e_1, \hat \eta))^{\frac{n+2s+2}{2}}}\de \eta,
\end{aligned}
\end{equation}
where $\hat \eta$ is such that $\eta = |\eta|\hat \eta$, and $O(\rho)$ does not depends on $\theta$. 
We notice that if $\eta \in C_\delta - R e_1$ then $(e_1, \hat \eta) <0$, so that $(e_1, \hat \eta) = -|(e_1, \hat \eta)|$.  

Since $\rho < \frac{\delta}{2}$, it holds that for every fixed $\tau \in (0,1)$ we have $B_{\tau \rho}((R-\rho)e_1) \subset C_\delta$ or equivalently $B_{\tau \rho}(-\rho e_1) \subset C_\delta - Re_1$. Moreover, when $\eta \in B_{\tau \rho}(-\rho e_1)$ the following inequalities hold:
\begin{equation}\label{stimepalletta}
\begin{aligned}
&(1-\tau)\rho < |\eta| < (1+ \tau)\rho; \\
&|(e_1, \hat \eta)| \geq \frac{1-\tau}{1+\tau};\\
&|\eta+\rho e_1| \leq \sqrt{\rho^2 + |\eta|^2 -2\rho|\eta|\cos \theta|(e_1, \hat \rho)|}. 
\end{aligned}
\end{equation}

Writing  $u_{s, \lambda}(x) = \frac{u_{s, \lambda}(x)}{\gamma^s(x)}\gamma^s(x) =: g_{s, \lambda}(x) \gamma^s(x)$, where $\gamma(x): = \text{dist}(x, \partial B_R)$, by Theorem \ref{boundregRO} we have that  $g_{s, \lambda}(x)$ is bounded in $C_\delta$. Moreover we also have $g_{s, \lambda}\geq 0$ in $C_\delta$ because $u_{s, \lambda} \geq 0$ in $C_\delta$.  
As a consequence when $\eta \in C_\delta -Re_1$ we get that 
\begin{equation}\label{stimau}
0 \leq u_{s, \lambda}(\eta + Re_1) \leq \sup_{C_\delta}|g_{s, \lambda}| |\eta|^s.
\end{equation}


Let us estimate the integral in the right hand side of \eqref{integrale1}. To this end we divide the domain of integration $C_\delta- Re_1$ in two parts. Let us fix $\tau \in (0,1)$. In the set $(C_\delta - R e_1) \setminus B_{\tau \rho}(-\rho e_1)$ it holds that 
\[
\sqrt{\rho^2 + |\eta|^2 - 2 \rho|\eta|\cos \theta |(e_1, \hat \eta)|} \geq |\eta + \rho e_1| \geq \tau \rho.
\]
Hence, performing the change of variables $\eta = \rho k$, we get
\[
\begin{aligned}
&\left| p_{n,s}\rho \sin \theta \int_{(C_\delta-Re_1)\setminus B_{\tau \rho}(-\rho e_1)} \frac{u_{s, \lambda}(\eta +Re_1)(n\rho^2 - 2s|\eta|^2 - (n-2s)\rho|\eta|\cos \theta |(e_1, \hat \eta)|)}{(\rho^2 + |\eta|^2 - 2 \rho|\eta|\cos \theta|(e_1, \hat \eta)|)^{\frac{n+2s+2}{2}}}\de \eta \right| \\
\leq&\ C\rho \int_{\R^n\setminus B_{\tau \rho}(-\rho e_1)} \frac{|\eta|^s (n\rho^2 + 2s|\eta|^2 + (n-2s)\rho|\eta|)}{|\eta + \rho e_1|^{n+2s+2}}\de \eta \\
=& \ C \rho^{1-s} \int_{\R^n \setminus B_\tau(-e_1)}\frac{|k|^s(n +2s|k|^2 + (n-2s)|k|)}{|k + e_1|^{n+2s+2}}\de k \leq C \rho^{1-s},
\end{aligned}
\]
where $C>0$ depends on $n$, $s$, $|g_{s, \lambda}|$ and $\tau$, but does not depends on $\rho$ and $\theta$. Therefore we obtain

\[
\begin{aligned}
&-(\rho \sin \theta)^{1-2s}\frac{\partial W}{\partial \rho}(\rho, \theta) \rho \\
=&\ O (\rho^{1-s}) + p_{n,s}\rho^{1-s} \sin \theta \int_{|k + e_1|< \tau} \frac{\frac{u_{s, \lambda}(\rho k +Re_1)}{\rho^s}(n - 2s|k|^2 - (n-2s)|k|\cos \theta |(e_1, \hat k)|)}{(1 + |k|^2 - 2 |k|\cos \theta|(e_1, \hat k)|)^{\frac{n+2s+2}{2}}}\de k.
\end{aligned}
\]
Using the relation 
\[
\begin{aligned}
&n-2s |k|^2 - (n-2s)|k| \cos \theta |(e_1, \hat k)| \\
=&\ \frac{n+2s}{2}(1 - |k|)(1+|k|) + \frac{n-2s}{2}(1 + |k|^2 - 2|k| \cos \theta |(e_1, \hat k)|),
\end{aligned}
\]
we then deduce that 
\[
\begin{aligned}
&(II)=\\
&p_{n,s}\int_{\Sf^{n-1}} \int_{\theta(\varepsilon)}^{\pi - \theta(\varepsilon)} \int_{|k + e_1|< \tau} \frac{\frac{n+2s}{2}\rho^{1-s}\sin \theta \frac{u_{s, \lambda}(\rho k +Re_1)}{\rho^s}(1-|k|)(1+|k|)\phi(\rho, \theta, \mathbf{e}) \Psi(\rho, \theta, \mathbf{e})}{(1 + |k|^2 - 2 |k|\cos \theta|(e_1, \hat k)|)^{\frac{n+2s+2}{2}}} \de k\de \theta \de \mathbf{e}  \\
&+ p_{n,s}\int_{\Sf^{n-1}} \int_{\theta(\varepsilon)}^{\pi - \theta(\varepsilon)} \int_{|k + e_1|< \tau} \frac{\frac{n - 2s}{2}\rho^{1-s} \sin \theta\frac{u_{s, \lambda}(\rho k +Re_1)}{\rho^s}\phi(\rho, \theta, \mathbf{e}) \Psi(\rho, \theta, \mathbf{e}) }{(1 + |k|^2 - 2 |k|\cos \theta|(e_1, \hat k)|)^{\frac{n+2s}{2}}}\de k \de \theta \de \mathbf{e} +O(\rho^{1-s})  \\
&= (i) + (ii) + O(\rho^{1-s}), 
\end{aligned}
\]
where $O(\rho^{1-s})$ is uniform with respect to $\varepsilon$.
We start showing that $|(ii)| \leq C\rho^{1-s}$, where $C$ does not depends on $\rho$ and $\varepsilon$. Indeed, applying Fubini-Tonelli's theorem and integrating by parts we get that 
\[
\begin{aligned}
&\left|\int_{\theta(\varepsilon)}^{\pi - \theta(\varepsilon)}\frac{ \sin \theta\phi(\rho, \theta, \mathbf{e})\Psi(\rho, \theta, \mathbf{e})}{(1 + |k|^2 - 2 |k|\cos \theta|(e_1, \hat k)|)^{\frac{n+2s}{2}}} \de \theta \right| \\
=&\ \bigg|\left[ - \frac{2}{n+2s-2} \frac{1}{2|k||(e_1, \hat k)|}\frac{\phi(\rho, \theta, \mathbf{e})\Psi(\rho, \theta, \mathbf{e})}{(1 + |k|^2 - 2|k|\cos \theta|(e_1, \hat k)|)^{\frac{n+2s-2}{2}}}\right]^{\pi - \theta(\varepsilon)}_{\theta(\varepsilon)}  \\
+&\ \int_{\theta(\varepsilon)}^{\pi- \theta(\varepsilon)}\frac{2}{n+2s-2} \frac{1}{2|k||(e_1, \hat k)|}\frac{\partial_\theta[\phi(\rho, \theta, \mathbf{e})\Psi(\rho, \theta, \mathbf{e})]}{(1 + |k|^2 - 2|k|\cos \theta|(e_1, \hat k)|)^{\frac{n+2s-2}{2}}}\de \theta \bigg| \\
\leq&\ C \frac{1}{|k+e_1|^{n+2s-2}}
\end{aligned}
\]
where $C>0$ depends only on $n$, $s$, $\phi$, $\Psi$ and $\tau$ but not on $\rho$ and $\varepsilon$, and where we used the inequalities \eqref{stimepalletta} evaluated at $\eta = \rho k$ and the fact that $\frac{\partial (\phi \Psi)}{\partial \theta}$ is uniformly bounded. Recalling \eqref{stimau} we get that $\frac{u_{s, \lambda}(\rho k +Re_1)}{\rho^s}\leq \sup_{C_\delta}|g_{s,\lambda}||k|^s \leq \sup_{C_\delta}|g_{s,\lambda}|(1+ \tau)^s$ and then
\[
|(ii)| \leq C\rho^{1-s} \int_{|k+e_1|< \tau} \frac{1}{|k + e_1|^{n+2s-2}} \de k \leq C_1 \rho^{1-s}.
\]
where $C_1>0$ depends on $n$, $s$, $\phi$, $\Psi$ and $\tau$ but not on $\rho$ and $\varepsilon$. 

Now, we can write 
\begin{equation}\label{calcolo12}
\begin{aligned}
(i)=& \ \int_{\Sf^{n-1}} \int_{\theta(\varepsilon)}^{\pi - \theta(\varepsilon)}\bigg[p_{n,s}\frac{n+2s}{2}\rho^{1-s} \sin \theta \int_{|k + e_1|< \tau} \frac{\frac{u_{s, \lambda}((R-\rho) e_1)}{\rho^s}(1-|k|)(1+|k|)}{(1 + |k|^2 - 2 |k|\cos \theta|(e_1, \hat k)|)^{\frac{n+2s+2}{2}}}\de k \\
+&\ p_{n,s}\frac{n+2s}{2}\rho^{1-s} \sin \theta \int_{|k + e_1|< \tau} \frac{\frac{u_{s, \lambda}(\rho k +Re_1)- u_{s, \lambda}((R-\rho)e_1)}{\rho^s}(1-|k|)(1+|k|)}{(1 + |k|^2 - 2 |k|\cos \theta|(e_1, \hat k)|)^{\frac{n+2s+2}{2}}}\de k\bigg]\phi(\rho, \theta, \mathbf{e}) \Psi(\rho, \theta, \mathbf{e}) \de \theta \de \mathbf{e}
\end{aligned}
\end{equation}

Thanks to \cite[Corollary 1.6, (a)]{HolderReg} there exists $C>0$ which depends on $n$, $s$ and $u_{s, \lambda}$ such that for $\alpha \in [s,1+2s)$
\[
[u_{s, \lambda}]_{0, \alpha; K} \leq C \text{dist}(K, \partial B_R)^{s - \alpha}.
\] 
Taking $\alpha = 1$ and $K = B_{\tau \rho}((R-\rho) e_1)$, since $\text{dist }(K, \partial B_R) = \rho(1-\tau)$, we have 
\[
|u_{s, \lambda}(x) - u_{s, \lambda}(y)| \leq C \rho^{s-1}(1-\tau)^{s-1}|x-y|.
\]
for every $x, y$ in $B_{\tau \rho}((R-\rho)e_1)$. 
Therefore, when $x = \rho k + Re_1$ and $y = (R-\rho)e_1$ we obtain 
\[
|u_{s, \lambda}(\rho k +Re_1)- u_{s,\lambda}((R-\rho)e_1)| \leq C \rho^s |k + e_1|
\]
where $C>0$ depends on $n$, $s$ and $\tau$ but not on $\rho$ nor on $\theta$. As a consequence we get that
\[
\left| \frac{u_{s, \lambda}(\rho k +Re_1)- u_{s, \lambda}((R-\rho)e_1)}{\rho^s}\right|\leq C|k+e_1| \leq C\sqrt{1 + |k|^2 - 2|k| \cos \theta |(e_1, \hat k)|},
\] 
where $C>0$ does not depends on $\rho$ and $\theta$, but only on $n$, $s$ and $\tau$. 
Moreover, since 
\begin{equation}\label{calcolo2}
|(1-|k|)(1 + |k|)| \leq (2+\tau)\sqrt{1 + |k|^2 - 2|k| \cos \theta |(e_1, \hat k)|}, 
\end{equation}
then, arguing as we have done for $(ii)$, we deduce that the second term in \eqref{calcolo12} is $O(\rho^{1-s})$ uniformly in $\varepsilon$. Now we can further refine the estimate noticing that 
\[
\begin{aligned}
&|\phi(\rho, \theta, \mathbf{e}) - \phi(\rho, 0, \mathbf{e})|\\
=&\ |\phi((R-\rho \cos \theta)\mathbf{e}, \rho \sin \theta) - \phi ((R-\rho)\mathbf{e}, 0)| \\
\leq&\ C\sqrt{|\rho(1-\cos \theta)|^2 + \rho^2 (\sin \theta)^2} = C \rho \sqrt{2(1-\cos \theta)}.
\end{aligned}
\]
Moreover, since in $|k + e_1|< \tau$ it holds $|k| > (1-\tau)$ (see \eqref{stimepalletta}), we get that
\begin{equation}\label{stimavarphi}
\begin{aligned}
|\phi(\rho, \theta, \mathbf{e}) - \phi(\rho, 0, \mathbf{e})|& \leq C \sqrt{2|k|(1-\cos \theta)} \leq C\sqrt{(|k|-1)^2+ 2|k|(1-\cos \theta)}\\
& \leq C\sqrt{1 + |k|^2 - 2 |k| \cos \theta|(e_1, \hat k)|}.
\end{aligned}
\end{equation}
Arguing as before, using again \eqref{calcolo2}, we get that
\[
\begin{aligned}
&\left|\int_{\theta(\varepsilon)}^{\pi - \theta(\varepsilon)}\frac{ \sin \theta(\phi(\rho, \theta, \mathbf{e})- \phi (\rho, 0, \mathbf{e})) (1- |k|)(1+ |k|)\Psi(\rho, \theta, \mathbf{e})}{(1 + |k|^2 - 2 |k|\cos \theta|(e_1, \hat k)|)^{\frac{n+2s+2}{2}}}\de \theta\right|  \\
\leq&\ \int_{\theta(\varepsilon)}^{\pi - \theta(\varepsilon)}\frac{ \sin \theta|\phi(\rho, \theta, \mathbf{e})- \phi (\rho, 0, \mathbf{e})| |1- |k||(1+ |k|)\Psi(\rho, \theta, \mathbf{e})}{(1 + |k|^2 - 2 |k|\cos \theta|(e_1, \hat k)|)^{\frac{n+2s+2}{2}}}\de \theta\\
\leq &\ C \frac{1}{|k+e_1|^{n+2s-2}}.
\end{aligned}
\]

As a consequence we have a further negligible term, and thus
\[
\begin{aligned}
(II) = O(\rho^{1-s}) &+ p_{n,s}\frac{n+2s}{2}\rho^{1-s}g_{s,\lambda}(\rho e_1) \cdot \\
&\cdot\int_{\Sf^{n-1}}\int_{\theta(\varepsilon)}^{\pi - \theta(\varepsilon)} \int_{|k + e_1|< \tau} \frac{\sin \theta(1-|k|)(1+|k|)\phi (\rho, 0, \mathbf{e})\Psi(\rho, \theta, \mathbf{e})}{(1 + |k|^2 - 2 |k|\cos \theta|(e_1, \hat k)|)^{\frac{n+2s+2}{2}}} \de k \de \theta \de \mathbf{e}.
\end{aligned}
\]

To conclude we have to get rid of the dependence from $\theta$ in the function $\Psi$. Arguing as in \eqref{stimavarphi} and since $\Psi$ is Lipschitz continuous in $\theta \in [0, \pi]$, uniformly in $\rho$ and $\mathbf{e}$, we get that  
\[
|\Psi(\rho, \theta, \mathbf{e}) - \Psi(\rho, 0, \mathbf{e})| \leq C \theta \leq 4C \sqrt{1 - \cos \theta} \leq 4C\sqrt{1 + |k|^2 -2|k| \cos \theta |(e_1, \hat \eta)|},
\]
where $C$ does not depend on $\theta$ and $\rho$. Hence, arguing as before, we obtain 
\[
\begin{aligned}
(II) = O(\rho^{1-s}) + &p_{n,s}\frac{n+2s}{2}\rho^{1-s}g_{s,\lambda}(\rho e_1) \left(\int_{\Sf^{n-1}}\phi(\rho, 0, \mathbf{e})\Psi(\rho, 0, \mathbf{e})\de \mathbf{e}\right)\cdot \\
&\cdot\int_{\theta(\varepsilon)}^{\pi - \theta(\varepsilon)} \int_{|k + e_1|< \tau} \frac{\sin \theta(1-|k|)(1+|k|)}{(1 + |k|^2 - 2 |k|\cos \theta|(e_1, \hat k)|)^{\frac{n+2s+2}{2}}} \de k \de \theta,
\end{aligned}
\]
where $O(\rho^{1-s})$ is uniform with respect to $\varepsilon$. Let us set $F(\theta,k):=\frac{\sin \theta(1-|k|)(1+|k|)}{(1 + |k|^2 - 2 |k|\cos \theta|(e_1, \hat k)|)^{\frac{n+2s+2}{2}}}$.
Thanks to the previous equality and \eqref{limfinito} we infer that both $$ \liminf_{\varepsilon \to 0^+}\int_{\theta(\varepsilon)}^{\pi - \theta(\varepsilon)} \int_{|k + e_1|< \tau} F(\theta,k)\de k \de \theta =  \liminf_{t \to 0^+} \int_{t}^{\pi - t} \int_{|k + e_1|< \tau}F(\theta,k) \de k \de \theta,$$ and $$\limsup_{\varepsilon \to 0^+}\int_{\theta(\varepsilon)}^{\pi - \theta(\varepsilon)} \int_{|k + e_1|< \tau} F(\theta,k) \de k \de \theta =  \limsup_{t \to 0^+} \int_{t}^{\pi - t} \int_{|k + e_1|< \tau}F(\theta,k)\de k \de \theta,$$ 
are finite and they do not depend on $\rho$, where for the last equality we used the properties of $\theta(\varepsilon)=\theta_\rho(\varepsilon)$ and the definition of $\liminf$, $\limsup$.

 At the end, since the quantity $g_{s, \lambda}(\rho e_1) \left(\int_{\Sf^{n-1}}\phi(\rho, 0, \mathbf{e})\Psi(\rho, 0, \mathbf{e})\de \mathbf{e}\right)$ is uniformly bounded with respect to $\rho$, we conclude that
\[
-C\rho^{1-s}\leq \liminf_{\varepsilon \to 0^+}(II) \leq \limsup_{\varepsilon \to 0^+}(II) \leq C \rho^{1-s},
\]
for some constant $C>0$ which does not depends on $\rho$. The proof is complete.
\end{proof}
\end{section}

\section*{Acknowledgements} 
The authors wish to thank Prof. S. Terracini for the useful discussions.

\end{document}